\documentclass[11pt]{article}
\usepackage[margin=0.7in]{geometry}
\usepackage[utf8]{inputenc}
\usepackage{float}
\usepackage{amsfonts}
\usepackage{amsthm}
\usepackage{amsmath}
\usepackage{appendix}
\usepackage{enumerate}
\usepackage{multicol}
\usepackage{amssymb}
\usepackage{xfrac} 
\usepackage{tabularx}
\usepackage[dvipsnames]{xcolor}
\usepackage[colorlinks,citecolor=green, linkbordercolor=blue,linkcolor=blue, urlcolor=blue,bookmarks=false,hypertexnames=true]{hyperref} 
\usepackage{graphicx}
\usepackage[T1]{fontenc}
\usepackage{tgtermes}
\graphicspath{ {./images/} }
\usepackage{tikz}
\usepackage{standalone}
\usepackage{authblk}

\theoremstyle{plain}
\newtheorem{thm}{Theorem}[section]
\newtheorem{lem}[thm]{Lemma}
\newtheorem{prop}[thm]{Proposition}
\newtheorem{cor}[thm]{Corollary}
\newtheorem{rem}[thm]{Remark}

\theoremstyle{definition}

\numberwithin{equation}{section}
\numberwithin{thm}{section}
\numberwithin{defn}{section}

\title{Orthodiagonal Maps, Tilings of Rectangles, and their Convergence to Conformal Maps}
\author[1]{Ilia Binder}
\author[2]{David Pechersky}
\affil[1]{University of Toronto}
\affil[2]{Beijing Institute of Mathematical Sciences and Applications (BIMSA)}
\date{\today}

\begin{document}

\maketitle

\begin{abstract}
\noindent A classic result of Brooks, Smith, Stone and Tutte \cite{BSST40} associates to any finite planar network with distinguished source and sink vertices, a tiling of a rectangle by smaller subrectangles whose aspect ratios are given by the conductances of corresponding edges in the network. This tiling can be viewed as a discrete analogue of the uniformizing conformal map that maps a simply connected domain with four distinguished prime ends to a rectangle, so that the four prime ends are mapped to the four corners of the rectangle. \\ \\
We make this intuition precise by showing that if $\Omega$ is a simply connected domain with four distinguished prime ends $A,B,C,D$ in counterclockwise order and $(\Omega_{n})_{n\geq{1}}$ is a sequence of orthodiagonal maps with distinguished boundary vertices $A_{n}, B_{n}, C_{n}, D_{n}$ in counterclockwise order, that are finer and finer approximations of $\Omega$ with its distinguished boundary points $A,B,C,D$, then the corresponding ``rectangle tiling maps" converge uniformly on compacts to the aforementioned conformal map on $\Omega$.
\end{abstract}

\tableofcontents{}

\section{Introduction}
\label{sec: Introduction}

\noindent In the past two decades, discrete complex analysis has proven to be a powerful tool for solving problems in critical 2D statistical physics. For instance, it was used to prove Cardy's formula for crossing probabilities in critical percolation \cite{SS01}, the convergence of interfaces in the critical Ising model to $SLE_{3}$ \cite{Ising}, and the convergence of percolation interfaces to $SLE_{6}$ \cite{CN07}. Here $SLE_{\kappa}$ refers to Schramm Loewner evolution, a $1$- parameter family of random curves parametrized by $\kappa>0$ that is characterized by conformal invariance and a certain Markov property. \\ \\
A recurring theme in thermodynamics is the notion of universality. Universality is a phenomenon where, due to some overarching principle, physical phenomenon with radically different microscopic details (i.e. magnets, porous media, and growing polymer chains in 2D), give rise to the same macroscopic behavior when you zoom out (i.e. $SLE$ curves). In the context of $2D$ statistical physics, universality suggests that for the same model, it shouldn't matter how we chop up our space- i.e. it doesn't matter if we define our model on a triangular or a square lattice- the resulting large-scale behavior should be the same.\\ \\
Orthodiagonal maps \cite{GJN20} are a very general class of discretizations of continuous $2D$ space that accommodate a notion of discrete complex analysis. Thus, at least in principle, the same techniques that were used previously to prove convergence of observables and interfaces in $2D$ critical statistical physics models for isoradial and more restrictive lattices, could be generalized to demonstrate the same convergence when our statistical physics happens on an orthodiagonal map. For By Theorem 1.1 of \cite{YY11} and Theorem 1.1 of \cite{GJN20}, loop-erased random walk on orthodiagonal maps converges to $SLE_{2}$. While this is not yet proven, the same is expected to be true of harmonic explorer: harmonic explorer on orthodiagonal maps should converge to $SLE_{4}$. Similarly, the fluctuations of the dimer model height function on orthodiagonal maps are expected to converge to the Gaussian free field (GFF).\footnote{In fact, this convergence is expected to be true in the more general setting of $t$-embeddings, subject to some very mild regularity assumptions. For more details, see Theorem 1.4 of \cite{CLR23}.}
 \\ \\
In this work, we leverage discrete complex analysis to solve a purely deterministic problem in this very general setting. A classic paper of Brooks, Smith, Stone and Tutte describes how planar electrical networks give rise to tilings of rectangles by smaller subrectangles \cite{BSST40}. Each subrectangle in the tiling corresponds to an edge of the network and its aspect ratio is precisely the conductance of this corresponding edge. These tilings can be thought of as discrete analogues of the uniformizing conformal map that maps a simply connected domain to a rectangle so that four distinguished points on the boundary of our simply connected domain are mapped to the four corners of the rectangle. We make this idea rigorous by showing that for any simply connected domain, if we have an increasingly fine sequence of orthodiagonal approximations, the associated tilings converge to the corresponding uniformizing conformal map. This significantly improves on a previous result of Georgakopoulos and Panagiotis who prove this convergence in the case where the approximating orthodiagonal map is just a chunk of the square grid \cite{GP20}. Furthermore, our approach is significantly different from the one in \cite{GP20} which relies heavily on the fact that reflected random walks on $\delta\mathbb{Z}^{2}$ converge in law to reflected Brownian motion as $\delta\rightarrow{0}$. To our knowledge, this result is not known for any other lattices. \\ \\
In recent work, Albin, Lind, and Poggi-Corradini provide an explicit rate of convergence for these tilings to the limiting conformal map in the general orthodiagonal setting (this is effectively Theorem 3 of \cite{ALP23}), subject to certain assumptions on the smoothness of the boundary of the simply connected domain that is being approximated. They then use this to prove convergence of the probabilistic interpretation of modulus as well as convergence of discrete extremal length to continuous extremal length in this setting (see Theorems 2 and 4 of \cite{ALP23}). By modifying their approach, we manage to avoid making any assumptions about the smoothness of the boundary of our simply connected domain, at the expense of providing an explicit rate of convergence. \\ \\
Our result can also be interpreted as the rectangle tiling analogue of similar results that are known for circle packings. The Koebe-Andreev-Thurston theorem tells us that any finite triangulation can be realized as the tangency graph of a circle packing in the plane. With this in mind, Bill Thurston made the observation that if you fill a simply connected domain with circles packed together, the Koebe-Andreev-Thurston theorem gives you a natural way to repack these circles in the unit disk in a way that preserves tangency. Since this ``repacking map" sends circles to circles, if we fill our simply connected domain with smaller and smaller circles, the corresponding circles in the images should also get smaller and smaller. In the limit, these repacking maps should converge to a function that sends infinitisimal circles to infinitisimal circles. In other words, a conformal map. Thus, Thurston conjectured that circle packings should give us a way to approximate the uniformizing conformal map from a simply connected domain to the unit disk \cite{Th85}. This was proven by Rodin and Sullivan when the circle packings in the simply connected domain consist of circles, all having the same radii, packed together in a honeycomb pattern \cite{RS87}.  This was later generalized to circle packings with arbitrary combinatorics by He and Schramm \cite{HS96}. For more on circle packings and their connection to complex analysis, see \cite{Stephenson}. \\ \\
Finally, it is worth noting that the closely related tilings of cylinders have been the object of recent study by Benjamini and Schramm \cite{BS96}, Georgakopoulos \cite{G16}, and Hutchcroft and Peres \cite{HP17} in connection with the Poisson boundaries of infinite planar graphs, by Hersonsky \cite{H18}, in connection with computational geometry, and by Bertacco, Gwynne, and Sheffield \cite{BGS24} in connection with random planar geometry. In particular, Theorem 1.4 of \cite{BGS24} and Theorem 3.13 of \cite{H18} are of a similar flavor to our main result, Theorem \ref{thm: convergence of rectangle tiling maps}. Theorem 1.4 of \cite{BGS24} considers a sequence of weighted graphs and their duals embedded in an infinite cylinder so that simple random walk on both the primal and the dual graphs converge to Brownian motion on the cylinder. In this setting, it asserts that certain corresponding tilings of this cylinder by rectangles (which can be thought of as discrete conformal maps) are close to the identity map. Theorem 3.13 of \cite{H18} can be viewed as a version of our result for doubly connected domains. Namely, given a doubly connected domain in the plane with piecewise linear boundary and a sequence of Delaunay triangulations of this domain with finer and finer mesh that are $\tau$-quasiuniform (see Definition 3.3 in \cite{H18}), Theorem 3.13 of \cite{H18} asserts that certain discrete conformal maps (which correspond to tilings of an annulus) converge uniformly on compacts to the conformal map that maps our doubly connected domain to an annulus with inradius 1. Note that Delaunay triangulations are a special case of the orthodiagonal maps we describe in Section \ref{subsec: OD map intro}. For details, see Section 1.2 of \cite{BG24}. 

\subsection{Organization} 
\label{subsec: Organization}

A brief outline of the proof of our main result (Theorem \ref{thm: convergence of rectangle tiling maps}) and therefore the organization of the rest of this paper is as follows: 
\begin{itemize}
    \item In Section \ref{sec: preliminaries} we review some preliminary material that we need to state and prove our main result, Theorem \ref{thm: convergence of rectangle tiling maps}, which appears at the end of this section. In particular, we see that the tiling maps in the statement of Theorem \ref{thm: convergence of rectangle tiling maps}, are effectively discrete holomorphic functions on orthodiagonal maps.   
    \item In Section \ref{sec: Precompactness of the Tiling Maps}, we show that our tiling maps are precompact with respect to the topology of uniform convergence on compacts. This gives us subsequential limits. We use the fact that for a family of continuous functions, precompactness is equivalent to uniform boundedness and equicontinuity. Uniform boundedness follows from the extremal length estimates in Section \ref{subsec: Bounds for EL and Dirichlet Energy}. Equicontinuity follows from the modulus of continuity estimates in Section \ref{subsec: modulus of continuity for tiling maps}. To motivate our argument in Section \ref{subsec: modulus of continuity for tiling maps}, in Section \ref{subsec: Modulus of Continuity for Limiting Conformal Map}, we use the same argument to prove analogous modulus of continuity estimates for conformal maps.
    \item In Section \ref{sec: Limits of Discrete Holomorphic Functions are Holomorphic}, we show that limits of discrete holomorphic functions on orthodiagonal maps are holomorphic. While this has already been proven in the more general setting of $t$-embeddings in \cite{CLR23} (see Proposition 6.15 of \cite{CLR23}), for completeness, we provide an alternative proof of this fact in the simpler orthodiagonal setting. Thus, discrete holomorphicity of our tiling maps gives us holomorphicity for any subsequential limit.
    \item Having established that any subsequential limit of our tiling maps is holomorphic, in Section \ref{sec: Convergence of Tiling Maps}, we uniquely identify the limiting holomorphic function by its boundary behavior, thereby completing the proof of our main result.  
\end{itemize}

\section*{Acknowledgements} Many thanks to Dmitry Chelkak for discussing the results of \cite{CLR23} with us and commenting on the state-of-the-art in discrete complex analysis more broadly. 

\section{Preliminaries}
\label{sec: preliminaries}

\subsection{The Theory of Electrical Networks}

Following $\cite{GJN20}$, a finite \textbf{network} is a finite graph $G=(V,E)$ along with a weight function $c:E\rightarrow{\mathbb{R}_{>0}}$. For any edge $e\in{E}$ we say that $c(e)$ is the \textbf{conductance} of that edge. The reciprocal $r(e)=\frac{1}{c(e)}$ is the \textbf{resistance} of that edge.  \\ \\
A function $\theta:\vec{E}\rightarrow{\mathbb{R}}$ is said to be \textbf{antisymmetric} if $\theta(-\vec{e})=-\theta(\vec{e})$ for all $\vec{e}\in{\vec{E}}$. Intuitively, antisymmetric functions on a network $G$ are the discrete analogues of vector fields. Let $\ell_{-}^{2}(\vec{E})$ denote the space of antisymmetric functions on $\vec{E}$ with the inner product: 
\begin{equation*}
\langle\theta,\psi\rangle_{r}:=\frac{1}{2}\sum_{\vec{e}\in{\vec{E}}}r(e)\theta(\vec{e})\psi(\vec{e})
\end{equation*}
The \textbf{energy} of $\theta\in{\ell_{-}^{2}(\vec{E})}$ is: 
\begin{equation*}
\mathcal{E}(\theta)=\|\theta\|_{r}^{2}=\langle\theta,\theta\rangle_{r}
\end{equation*}
Given $f:V\rightarrow{\mathbb{R}}$, its \textbf{gradient} $cdf:\vec{E}\rightarrow{\mathbb{R}}$ is given by: 
\begin{equation*}
(cdf)(\vec{e})=c(e)(f(e^{+})-f(e^{-}))
\end{equation*}
for any directed edge $\vec{e}=(e^{-},e^{+})$. For any function $f:V\rightarrow{\mathbb{R}}$, the gradient is antisymmetric. Thus, we can define the energy of a function $f:V\rightarrow{\mathbb{R}}$ as the energy of its gradient: 
\begin{equation*}
\mathcal{E}(f)=\mathcal{E}(cdf)=\frac{1}{2}\sum_{\vec{e}\in{\vec{E}}}c(e)\big(f(e^{+})-f(e^{-})\big)^{2}
\end{equation*}
Given a function in $\ell^{2}_{-}(\vec{E})$, we are often interested in the energy of its restriction to some subgraph of $G$. To make it clear where it is we are computing the energy, if $\theta\in{\ell_{-}^{2}(\vec{E})}$ and $G'=(V',E')$ is a subgraph of $G=(V,E)$ with the same edge weights, then:   
\begin{equation*}
    \mathcal{E}(\theta;G')=\frac{1}{2}\sum_{\vec{e}\in{\vec{E'}}}r(e)\theta(\vec{e})^{2}
\end{equation*}
Similarly, if $f:V\rightarrow{\mathbb{R}}$,
\begin{equation*}
    \mathcal{E}(f;G')=\frac{1}{2}\sum_{\vec{e}\in{\vec{E'}}}c(e)\big(f(e^{+})-f(e^{-})\big)^{2}
\end{equation*}
A function $\theta\in{\ell_{-}^{2}(\vec{E})}$ satisfies the \textbf{cycle law} if for any directed cycle $\gamma=(\vec{e}_{1}, \vec{e}_{2}, ..., \vec{e}_{m})$ in $G$,
\begin{equation*}
\sum_{i=1}^{m}r(e_{i})\theta(\vec{e}_{i})=0
\end{equation*}
It is not hard to see that $\theta\in{\ell_{-}^{2}(\vec{E})}$ satisfies the cycle law if and only if $\theta=cdf$ for some function $f:V\rightarrow{\mathbb{R}}$. 
Given $\theta\in{\ell_{-}^{2}(\vec{E})}$, its \textbf{divergence} $\text{div}(\theta):V\rightarrow{\mathbb{R}}$ is given by: 
\begin{equation*}
\big(\text{div}(\theta)\big)(x)=\sum_{e^{-}=x}\theta(\vec{e})
\end{equation*}
Similar to the continuous setting, the divergence of $\theta$ at $x$ measures the net flow out of $x$ by $\theta$. Given distinct vertices $a,z\in{V}$, a function $\theta\in{\ell_{-}^{2}(\vec{E})}$ is a \textbf{flow} from $a$ to $z$ if:
\begin{equation*}
\big(\text{div}(\theta)\big)(x)=0 \hspace{5pt}\text{for all $x\in{V\setminus{\{a,z\}}}$.}  
\end{equation*}
Given a flow $\theta$ from $a$ to $z$ its \textbf{strength}, denoted by $\|\theta\|$, is defined as follows: 
\begin{equation*}
\|\theta\|=\sum_{x:x\sim{a}}\theta(a,x)=\big(\text{div}(\theta)\big)(a)
\end{equation*}
For every flow $\theta$ from $a$ to $z$, 
\begin{equation*}
\|\theta\|=\sum_{y:y\sim{z}}\theta(y,z)=-\big(\text{div}(\theta)\big)(z)
\end{equation*}
This is because: 
\begin{equation*}
\sum_{\vec{e}\in{\vec{E}}}\theta(\vec{e})=0=\sum_{x\in{V}}\sum_{y:y\sim{x}}\theta(x,y)=\sum_{x\in{V}}\big(\text{div}(\theta)\big)(x)=\big(\text{div}(\theta)\big)(a)+\big(\text{div}(\theta)\big)(z)
\end{equation*}
The first equality follows from the antisymmetry of $\theta$. Given $f:V\rightarrow{\mathbb{R}}$, its \textbf{Laplacian} $\Delta{f}:V\rightarrow{\mathbb{R}}$ is given by: 
\begin{equation*}
\Delta{f}(x)=\big(\text{div}(cdf)\big)(x)=\sum_{e^{-}=x}c(e)\big(f(e^{+})-f(e^{-})\big)=\sum_{y:y\sim{x}}c(x,y)\big(f(y)-f(x)\big)
\end{equation*}
If $\Delta{f}(x)=0$, we say that $f$ is \textbf{harmonic} at $x$. 
Equivalently, $f$ is harmonic at $x$ if: 
\begin{equation*}
f(x)=\frac{1}{\pi_{x}}\sum_{y:y\sim{x}}c(x,y)f(y)
\end{equation*}
where: 
\begin{equation*}
\pi_{x}=\sum\limits_{y:y\sim{x}}c(x,y)
\end{equation*}
From this formula, it is immediate that harmonic functions satisfy the maximum principle: 
\begin{prop}
Suppose that $G=(V,E,c)$ is a finite network and $h:V\rightarrow{\mathbb{R}}$ is harmonic on $U\subsetneq{V}$. Define: 
\begin{equation*}
\partial{U}=\{w\in{V}\setminus{U}:w\sim{u} \hspace{5pt}\text{for some}\hspace{5pt} u\in{U} \}
\end{equation*}
Then: 
\begin{equation*}
\max_{u\in{U}}h(u)\leq{\max_{v\in{\partial{U}}}}h(v)
\end{equation*}
\end{prop}
\noindent Following \cite{pmrwcp}, a simple random walk on the network $G=(V,E,c)$ is the discrete time Markov process $(X_{n})_{n\geq{0}}$ with transition probabilities: 
\begin{equation*}
P(x,y)=\frac{c(x,y)}{\pi_{x}}1_{(x\sim{y})}
\end{equation*}
Given a function $f:V\rightarrow{\mathbb{R}}$ and vertices $a,z\in{V}$, it is clear that $cdf$ is a flow from $a$ to $z$ if and only if $\Delta{f}(x)=0$ for all $x\in{V\setminus{\{a,z\}}}$. We call such a function a \textbf{voltage}. Since the discrete boundary value problem: \begin{align*}
    h(a)&=\alpha \\ 
    h(b)&=\beta \\ 
    \Delta{h}(x)&=0 \hspace{5pt}\text{for all $x\in{V\setminus{\{a,z\}}}$}
\end{align*}  
has a unique solution for any choice of $\alpha, \beta\in{\mathbb{R}}$, voltages form a two-parameter family. The flow $cdh$ corresponding to any voltage $h:V\rightarrow{\mathbb{R}}$ is known as the corresponding \textbf{current flow}. Given distinct vertices $a,z\in{V}$, the \textbf{effective resistance} between $a$ and $z$ in $G$, denoted by $R_{\text{eff}}(a\leftrightarrow{z};G)$, is given by: 
\begin{equation*}
R_{\text{eff}}(a\leftrightarrow{z};G)=\frac{h(z)-h(a)}{\|cdh\|}
\end{equation*}
where $h$ is any nonconstant voltage. To see that this quantity is well- defined, just observe that adding a constant doesn't affect the voltage difference between $a$ and $z$, $h(z)-h(a)$, or the current flow $cdh$. Similarly, multiplying $h$ by a nonzero constant scales the voltage difference and the strength of the corresponding current flow by the same factor, leaving the effective resistance unchanged.\\ \\
More generally, given disjoint sets of vertices $A,Z\subseteq{V}$, we can define a new network by identifying the vertices of $A$ to a single vertex $a$ and identifying the vertices of $Z$ to a single vertex $z$. Then the effective resistance $R_{\text{eff}}(A\leftrightarrow{Z};G)$ between $A$ and $Z$ in $G$, is given by the electrical resistance between the vertices $a$ and $z$ in this new network.
\\ \\
In this paper, we will frequently need to bound effective resistances from above and below. To do this, we will use the following pair of variational formulas. Dirichlet's Principle allows us to bound effective resistances from below by finding functions with small discrete Dirichlet energy: 
\begin{prop}[Dirichlet's Principle]
If $G=(V,E,c)$ is a finite network with distinct vertices $a,z\in{V}$ then: 
\begin{equation*}
R_{\text{eff}}(a\leftrightarrow{z};G)=\sup\left\{\frac{1}{\mathcal{E}(h)}: h:V\rightarrow{\mathbb{R}},  \hspace{2pt} h(a)=0, \hspace{2pt} h(z)=1\right\}
\end{equation*}
\end{prop}
\noindent Thomson's Principle allows us to bound effective resistances from above, by finding low- energy flows:
\begin{prop}[Thomson's Principle]
If $G=(V,E,c)$ is a finite network with distinct vertices $a,z\in{V}$ then: 
\begin{equation*}
R_{\text{eff}}(a\leftrightarrow{z};G)=\inf\left\{\mathcal{E}(\theta): \|\theta\|=1, \hspace{2pt} \text{$\theta$ is a flow from $a$ to $z$} \right\}
\end{equation*}
\end{prop}
\noindent Given $f:V\rightarrow{\mathbb{R}}$ and $A,B\subseteq{V}$ nonempty, disjoint sets of vertices, we define the quantity $\text{gap}_{A,B}(f)$ as follows: 
$$
\text{gap}_{A,B}(f)=\min\limits_{b\in{B}}f(b) - \max\limits_{a\in{A}}f(a)
$$
Additionally, recall Proposition 4.11 of \cite{GJN20} which tells us that: 
\begin{prop}
\label{Skopenkov's trick}
If $G=(V,E,c)$ is a finite network, $A,B\subseteq{V}$ are disjoint, nonempty sets of vertices then for any flow $\theta$ on $G$ and any function $f:V\rightarrow{\mathbb{R}}$ such that $\text{gap}_{A,B}(f)\geq{0}$,
$$
\|\theta\|\cdot{\text{gap}_{A,B}(f)}\leq{\mathcal{E}(\theta)^{1/2}\mathcal{E}(f)^{1/2}}
$$
\end{prop}
\noindent This inequality follows almost immediately from the Cauchy- Schwarz inequality on $\ell^{2}_{-}(\vec{E})$. Furthermore, Dirichlet's Principle and Thomson's Principle can both be recovered cheaply as corollaries of this inequality. 

\subsection{Extremal Length and Planar Networks}
\label{Extremal Length and Planar Networks}

\noindent Suppose $G=(V, E, c)$ is a finite network and $\Gamma$ is a nonempty collection of paths in $G$. Then the \textbf{extremal length} of the collection of paths $\Gamma$ in $G$ is given by the following variational formula: 
\begin{equation}
\label{def_EL}
\lambda(\Gamma, G):=\sup_{\rho}\frac{\ell^{2}(\rho,\Gamma)}{A(\rho)}
\end{equation}
where our supremum is taken over all nonzero metrics $\rho:E\rightarrow{\mathbb{R}_{\geq{0}}}$, and:
\begin{align*}
    \ell(\rho, \Gamma)&:=\min\{\sum_{e\in{\gamma}}\rho(e): \gamma\in{\Gamma}\}, & A(\rho)&:=\sum_{e\in{E}}c(e)\rho(e)^{2}
\end{align*}
Note that the quantity $\ell^{2}(\rho,\Gamma)/A(\rho)$ doesn't change if we replace $\rho$ by some scalar multiple $\lambda\rho$ where $\lambda>0$. Thus:
\begin{equation*}
\lambda(\Gamma,G)=\sup_{\rho}\frac{l^{2}(\rho,\Gamma)}{A(\rho)}=\sup_{A(\rho)=1}l^{2}(\rho,\Gamma)=\sup_{l(\rho,\Gamma)=1}\frac{1}{A(\rho)}=\big(\inf_{l(\rho,\Gamma)=1}A(\rho)\big)^{-1}
\end{equation*}
The set of metrics $\rho$ on $G$ such that $l(\rho,\Gamma)=1$ is referred to as the set of \textbf{admissible metrics} and is denoted by $\mathcal{A}(\Gamma)$. We say that a metric $\rho$ on $G$ is \textbf{extremal} for $\lambda(\Gamma,G)$ if $\lambda(\Gamma,G)=l^{2}(\Gamma,G)/A(\rho)$. Looking at the second equality above, we see that when we compute the extremal length of the path family $\Gamma$, we are optimizing a continuous function, $\rho\mapsto{l^{2}(\rho,\Gamma)}$, over the set  of $\rho\in{\mathbb{R}_{\geq{0}}^{E}}$ such that $A(\rho)=1$. This is a compact subset of $\mathbb{R}^{E}$ with respect to the standard topology on $\mathbb{R}^{E}$. Thus, in contrast to the continuous setting (see Exercise IV.9 of \cite{GM05}), for a finite network we always have an extremal metric. \\ \\ 
This extremal metric is unique up to multiplication by a scalar. This follows by the same argument as in the continuous setting: suppose that $\rho_{1}$ and $\rho_{2}$ are both extremal for $\lambda(\Gamma,G)$. First we rescale so that $A(\rho_{1})=A(\rho_{2})=1$. It follows that $l^{2}(\Gamma,\rho_{1})=l^{2}(\Gamma,\rho_{2})=\lambda(\Gamma,G)$. Consider the metric $\nu:=\frac{1}{2}(\rho_{1}+\rho_{2})$. Trivially: 
\begin{equation}
l(\nu,\Gamma)\geq{\frac{1}{2}\big(l(\rho_{1},\Gamma)+l(\rho_{2},\Gamma)\big)}=\sqrt{\lambda(\Gamma,G)} \hspace{5pt} \implies \hspace{5pt} l^{2}(\Gamma, \nu)\geq{\lambda(\Gamma,G)}
\end{equation}
On the other hand, by Cauchy- Schwartz, $A(\nu)\leq{\frac{1}{2}\big(A(\rho_{1})+A(\rho_{2})\big)}=1$ with equality iff $\rho_{1}$ is a scalar multiple of $\rho_{2}$. If $A(\nu)<1$ then $l^{2}(\nu,\Gamma)/A(\nu)>\lambda(\Gamma,G)$. This is not possible since $\lambda(\Gamma,G)$ is the supremum of $l^{2}(\rho,\Gamma)/A(\rho)$ over all metrics $\rho$. Thus, $\rho_{2}$ is a scalar multiple of $\rho_{1}$. Since $A(\rho_{1})=A(\rho_{2})=1$ we actually have that $\rho_{1}=\rho_{2}$.   \\ \\
In all of the cases we're interested in, the path family $\Gamma$ will be the set of paths $\gamma$ in $G$ that start at a vertex of $S$ and end at a vertex of $T$ for $S$, $T$ nonempty disjoint subsets of $V$. We denote the extremal length of this path family by $\lambda(S\leftrightarrow{T};G)$. It turns out that the quantity $\lambda(S\leftrightarrow{T};G)$ is precisely the effective resistance between $S$ and $T$ from the theory of electrical networks:
\begin{prop}
\label{Effective Resistance is Extremal Length}
(Theorem 2 of \cite{D62}) Suppose $G=(V,E,c)$ is a finite network and $S, T$ are nonempty, disjoint subsets of $V$. Then:
$$
\lambda(S\leftrightarrow{T};G)=R_{\text{eff}}(S\leftrightarrow{T};G)
$$
\end{prop}
\noindent One nice property of extremal length is blocking duality. Given, $S,T$ nonempty, disjoint sets of vertices in $G$, we say that a set $F\subset{E}$ is an \textbf{$S$- $T$ cut} if $F$ separates $S$ from $T$ in $G$. That is, if we remove the edges of $F$ from $G$, there is no nearest- neighbor path in $G$ starting at a vertex of $S$ and ending at a vertex of $T$. \\ \\
Let $B(S,T;G)$ denote the set of $S$- $T$ cuts in $G$. Analogous to how we defined the extremal length of a path family, we can talk about the extremal length of the set of $S$- $T$ cuts in $G$. This is denoted by $\lambda(S\not\leftrightarrow{T};G)$ and defined as follows: 
$$
\lambda(S\not\leftrightarrow{T};G)=\sup_{\rho}\frac{\ell^{2}(\rho,B(S,T;G))}{A(\rho)}
$$
where our supremum is taken over all nonzero metrics $\rho:E\rightarrow{\mathbb{R}_{\geq{0}}}$, and:
\begin{align}
    \ell(\rho, B(S,T;G))&=\min\{\sum_{e\in{F}}\rho(e): F\in{B(S,T;G)}\}, & A(\rho)&=\sum_{e\in{E}}c(e)\rho(e)^{2}
\end{align}
More generally, while we initially restricted our attention to path families so as to draw parallels with the continuous theory, it is clear that if we let $\Gamma$ be any family of multisets of edges in $G$, Definition \ref{def_EL} still makes sense. Thus, we can actually talk about the extremal length of any family of multisets of edges of $G$. For instance, the modulus of the set of spanning trees of a network has been the subject of recent study in \cite{AKP20}. \\ \\
A classic result of Ford and Fulkerson relates the extremal length of paths from $S$ to $T$ to the extremal length of the set of $S$- $T$ cuts:
\begin{prop}
(Theorem 1 of \cite{FF56}) If $G=(V,E,c)$ and $H=(V,E,r)$ are finite networks so that $r:E\rightarrow{\mathbb{R}_{>0}}$ is the resistance function corresponding to the conductance function $c:E\rightarrow{\mathbb{R}_{>0}}$ and $S,T$ are nonempty, disjoint sets of vertices in $G$, then: 
$$
\lambda(S\leftrightarrow{T};G)\cdot{\lambda(S\not\leftrightarrow{T};H)}=1
$$
\end{prop}
\noindent For a more modern treatment of this result as well as a generalization to the case of $p$-extremal length, see \cite{ACFP18}. This result is particularly useful in the case where our graph $G$ is planar, in which case we can identify the set of $S$- $T$ cuts with path families in the dual graph. \\ \\
A \textbf{finite planar map} is a finite planar graph $(V,E)$ along with a proper embedding of this graph into the Riemann sphere, viewed up to homeomorphism of the Riemann sphere. Specifying a proper embedding of a graph in the Riemann sphere up to orientation- preserving homeomorphism is equivalent to assigning a coherent system of orientations to the edges about each vertex (for details, see Section 1.1.2 of \cite{stfluor}). Thus, despite the topology present in our initial definition, planar maps can be viewed as purely combinatorial objects. Equivalently, we can think of finite planar maps as gluings of polygons along edges so that the resulting topological manifold is a sphere. \\ \\
A \textbf{quadrangulation with boundary} is a bipartite planar map all of whose faces are quadrilaterals, with the possible exception of some finite number of distinguished faces which we think of as ``holes" in our planar map. Notice that requiring our planar map to be bipartite is equivalent to asking that all of these ``hole" faces have an even number of sides. Given a quadrangulation with boundary $G=(V,E)$, we refer to these distinguished faces as the \textbf{exterior} faces of $G$.  The remaining faces are called the \textbf{interior} faces of $G$. The edges and vertices tangent to the exterior faces of $G$ are known as the boundary vertices and edges of $G$. We denote these by $\partial{V}$ and $\partial{E}$. A quadrangulation with boundary $G=(V,E)$ is \textbf{simply- connected} if $G$ has a unique exterior face whose boundary is a simple closed curve. \\ \\
Since our quadrangulations are bipartite, we have a natural bipartition of the vertices $V=V^{\bullet}\sqcup{V^{\circ}}$. The vertices of $V^{\bullet}$ are known as the \textbf{primal} vertices of $G$ and are typically colored black. The vertices of $V^{\circ}$ are known as the \textbf{dual} vertices and are typically colored white. These give rise to the primal and dual graphs $G^{\bullet}=(V^{\bullet}, E^{\bullet})$ and $G^{\circ}=(V^{\circ}, E^{\circ})$. $G^{\bullet}$ is formed by connecting any pair of primal vertices that share an interior face in $G$. Similarly, $G^{\circ}$ is formed by connecting any pair of dual vertices that share an interior face in $G$. Since the interior faces of $G$ are all quadrilaterals, each interior face corresponds to one primal and one dual edge. In this way, there is a natural correspondence between primal and dual edges.    \\ \\
Based on the paragraph above, it might seem that the setting we are working in is very restrictive. On the contrary, observe that this procedure of recovering a graph $G^{\bullet}$ (and its dual $G^{\circ}$) from a quadrangulation $G$ gives us a one- to- one correspondence between the set of quadrangulations with $n$ faces and the set of planar maps with $n$ edges (see Section 2.2.1 of \cite{stfluor}). In other words, restricting our attention to bipartite quadrangulations with $k$ holes is equivalent to restricting our attention to embeddings of a graph and its dual in the Riemann sphere, up to orientation- preserving homeomorphism, so that the resulting discrete object has the topology of the Riemann sphere with k discs removed. \\ \\
Given a quadrangulation with boundary $G=(V^{\bullet}\sqcup{V^{\circ}}, E)$, a \textbf{conformal metric} on $G$ is a function $c:E^{\bullet}\sqcup{E^{\circ}}\rightarrow{(0,\infty))}$ such that:
$$
c(e^{\circ})=\frac{1}{c(e^{\bullet})}
$$
for $e^{\circ}\in{E^{\circ}}$, $e^{\bullet}\in{E^{\bullet}}$ so that $e^{\circ}$ is the dual edge corresponding to the primal edge $e^{\bullet}$. Let $c^{\bullet}:E^{\bullet}\rightarrow{(0,\infty)}$ and $c^{\circ}:E^{\circ}\rightarrow{(0,\infty)}$ denote the conductances on $G^{\bullet}$ and $G^{\circ}$ produced by restricting $c$ to $E^{\bullet}$ and $E^{\circ}$ respectively. If $\theta\in{\ell_{-}^{2}(\vec{E^{\bullet}})}$ and $f:V^{\bullet}\rightarrow{\mathbb{R}}$, we write: 
\begin{align*}
    \mathcal{E}^{\bullet}(\theta)&=\mathcal{E}(\theta;G^{\bullet}), & \mathcal{E}^{\bullet}(f)&=\mathcal{E}(f;G^{\bullet})
\end{align*}
to emphasize that these energies are being computed on the primal graph $G^{\bullet}$. Similarly, given a subgraph $H$ of $G^{\bullet}$, we write: 
\begin{align*}
    \mathcal{E}^{\bullet}(\theta;H)&=\mathcal{E}(\theta;H), & \mathcal{E}^{\bullet}(f;H)&=\mathcal{E}(f;H)
\end{align*}
Given $\omega\in{l^{2}_{-}(\vec{E^{\circ}})}$, $g:V^{\circ}\rightarrow{\mathbb{R}}$, and a subgraph $H$ of $G^{\circ}$, the quantities $\mathcal{E}^{\circ}(\omega)$, $\mathcal{E}^{\circ}(g)$, $\mathcal{E}^{\circ}(\omega;H)$, $\mathcal{E}^{\circ}(g;H)$ are defined analogously. \\ \\
A \textbf{discrete conformal rectangle} is a simply connected, bipartite quadrangulation with boundary endowed with a conformal metric $c:E^{\bullet}\sqcup{E^{\circ}}\rightarrow{(0,\infty)}$ and four distinguished boundary points $A^{\bullet}, B^{\bullet}, C^{\bullet}, D^{\bullet}\in{\partial{V^{\bullet}}}$, listed in counterclockwise order. Since our quadrangulation with boundary is simply connected, it has a unique exterior face $f$, so $A^{\bullet}, B^{\bullet}, C^{\bullet}, D^{\bullet}$ must all lie along  $f$. Furthermore, having embedded our quadrangulation in the Riemann sphere, we can talk about orientation. \\ \\
The distinguished boundary points $A^{\bullet}, B^{\bullet}, C^{\bullet}, D^{\bullet}\in{\partial{V^{\bullet}}}$ of our discrete conformal rectangle give rise to primal boundary arcs $[A^{\bullet},B^{\bullet}], [C^{\bullet}, D^{\bullet}]\subseteq{\partial{V^{\bullet}}}$. $[A^{\bullet},B^{\bullet}]$ refers to the set of primal vertices that lie along the counterclockwise path from $A^{\bullet}$ to $B^{\bullet}$ along the boundary of $f$. Similarly, $[C^{\bullet},D^{\bullet}]$ is the set of primal vertices that lie along the counterclockwise path from $C^{\bullet}$ to $D^{\bullet}$ along the boundary of $f$. These primal boundary arcs have corresponding dual arcs $[B^{\circ},C^{\circ}], [D^{\circ}, A^{\circ}]\subseteq{\partial{V^{\circ}}}$ where $[B^{\circ},C^{\circ}]$ consists of the set of dual vertices that lie along the counterclockwise path from $B^{\bullet}$ to $C^{\bullet}$ along the boundary of $f$. Similarly, $[D^{\circ}, A^{\circ}]$ is the set of dual vertices that lie along the counterclockwise path from $D^{\bullet}$ to $A^{\bullet}$ along the boundary of $f$.
\begin{figure}[H]
\label{fig: discrete conformal rectangle}
\centering 
\includegraphics[scale=0.55]{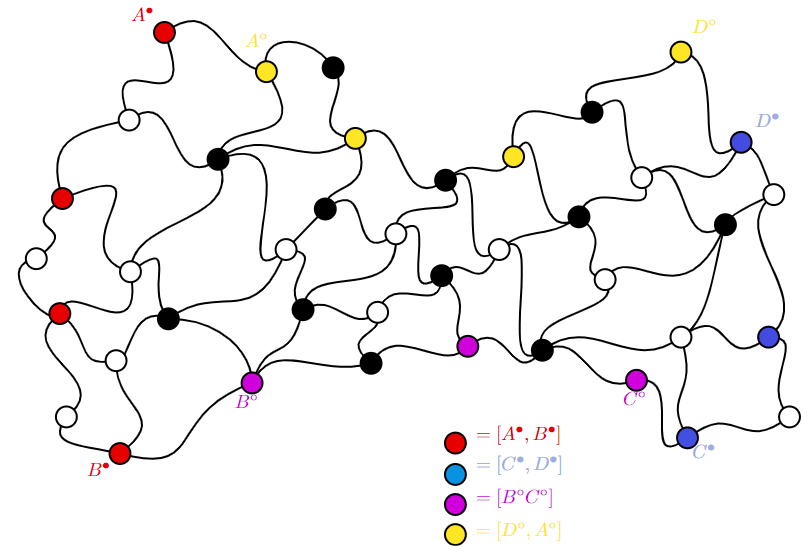}
\caption{A discrete conformal rectangle with its four distinguished boundary arcs.}
\end{figure}
\noindent We say that an $S$- $T$ cut, $F\subseteq{E}$, is \textbf{minimal} if for any edge $e\in{F}$, $F\setminus{\{e\}}$ is no longer an $S$- $T$ cut.
\begin{lem}
\label{lem: min cuts are dual paths rectangles}
(Lemma VIII.1 of \cite{D62}) If $(G,c)$ is a discrete conformal rectangle with distinguished boundary points $A^{\bullet}, B^{\bullet}, C^{\bullet}, D^{\bullet}\in{\partial{V^{\bullet}}}$ giving rise to primal boundary arcs $[A^{\bullet}, B^{\bullet}]$, $[C^{\bullet}, D^{\bullet}]\subseteq{\partial{V^{\bullet}}}$ with corresponding dual arcs $[B^{\circ},C^{\circ}], [D^{\circ}, A^{\circ}]\subseteq{\partial{V^{\circ}}}$, then the set of minimal $[A^{\bullet}, B^{\bullet}]$- $[C^{\bullet}, D^{\bullet}]$ cuts in $G^{\bullet}$ is in one- to- one correspondence with the set of simple paths from $[B^{\circ},C^{\circ}]$ to $[D^{\circ}, A^{\circ}]$ in $G^{\circ}$. 
\end{lem}
\noindent When computing the extremal length of the family of paths between disjoint vertex sets $S$ and $T$ in $G$, for a fixed metric $\rho$, we are interested in the quantity: 
$$
\inf_{\gamma}\left\{\sum_{e\in{\gamma}}\rho(e)\right\}
$$
where our infimum is taken over all paths $\gamma$ in $G$ between $S$ and $T$. Since any such path that isn't simple has a simple subpath of smaller $\rho$- weight, when taking this infimum, it actually suffices to restrict our attention to simple paths $\gamma$ from $S$ to $T$. Similarly, when we compute the extremal length of the set of $S$- $T$ cuts in a network, rather than taking an infimum over all $S$- $T$ cuts, it suffices to restrict our attention only to minimal $S$- $T$ cuts. Thus, as an immediate corollary of Lemma \ref{lem: min cuts are dual paths rectangles}, we have that:
\begin{cor}
\label{FF_discrete_rectangles}
If $(G,c)$ is a discrete rectangle with distinguished boundary points $A^{\bullet}, B^{\bullet}, C^{\bullet}, D^{\bullet}\in{\partial{V^{\bullet}}}$ giving rise to primal boundary arcs $[A^{\bullet}, B^{\bullet}], [C^{\bullet}, D^{\bullet}]\subseteq{\partial{V^{\bullet}}}$ with corresponding dual arcs $[B^{\circ},C^{\circ}], [D^{\circ}, A^{\circ}]\subseteq{\partial{V^{\circ}}}$, then: 
$$
\lambda([A^{\bullet}, B^{\bullet}]\not\leftrightarrow{[C^{\bullet}, D^{\bullet}]};(G^{\bullet},c^{\bullet}))=\lambda([B^{\circ},C^{\circ}]\leftrightarrow{[D^{\circ}, A^{\circ}]};(G^{\circ},c^{\circ}))
$$
\end{cor}
\noindent Suppose $G=(V^{\bullet}\sqcup{V^{\circ}},E)$ is a bipartite quadrangulation with boundary endowed with a conformal metric $c:E^{\bullet}\sqcup{E^{\circ}}\rightarrow{\mathbb{R}}$. We say that $h:V^{\bullet}\rightarrow{\mathbb{R}}$ is harmonic on $G^{\bullet}$ if $h$ is harmonic on $\text{Int}(V^{\bullet})=V^{\bullet}\setminus{\partial{V^{\bullet}}}$. $\widetilde{h}:V^{\circ}\rightarrow{\mathbb{R}}$ is a \textbf{harmonic conjugate} of $h$ on $G$ if for any interior face $f$ of $G$, we have that: 
\begin{equation}
\label{discrete_CR}
\big(\widetilde{h}(w_{2})-\widetilde{h}(w_{1})\big)=c^{\bullet}(v_{1}, v_{2})\big(h(v_{2})-h(v_{1})\big)
\end{equation}
where $v_{1}, w_{1}, v_{2}, w_{2}$ are the vertices of $f$ listed in counterclockwise order so that $v_{1}, v_{2}\in{V^{\bullet}}$ and $w_{1}, w_{2}\in{V^{\circ}}$. Equation \ref{discrete_CR} is a discrete analogue of the Cauchy- Riemann equations for a quadrangulation with boundary $G$, endowed with a conformal metric. It is easy to check that the conjugate of a harmonic function on $G^{\bullet}$ is harmonic on $G^{\circ}$. Additionally, since $c$ is a conformal metric, if $\widetilde{h}$ is a harmonic conjugate of $h$, then $h$ is a harmonic conjugate of $\widetilde{h}$. The next two propositions are well-known, though they are rarely stated in this generality: 
\begin{prop}
\label{uniqueness of harmonic conjugate}
    Suppose $G=(V^{\bullet}\sqcup{V^{\circ}},E)$ is a bipartite quadrangulation with boundary endowed with a conformal metric $c:E^{\bullet}\sqcup{E^{\circ}}\rightarrow{\mathbb{R}}$. If $h:V^{\bullet}\rightarrow{\mathbb{R}}$ is harmonic on $G^{\bullet}$, then the harmonic conjugate of $h$, if it exists, is unique up to an additive constant. 
\end{prop}
 
\begin{prop}
Suppose $G=(V^{\bullet}\sqcup{V^{\circ}},E)$ is a bipartite quadrangulation with boundary endowed with a conformal metric $c:E^{\bullet}\sqcup{E^{\circ}}\rightarrow{\mathbb{R}}$ and $h:V^{\bullet}\rightarrow{\mathbb{R}}$ is harmonic on $G^{\bullet}$. Let $H=(V_{H}^{\bullet}\sqcup{V_{H}^{\circ}}, E_{H})$ be submap of $G$ which is itself a simply- connected, bipartite quadrangulation with boundary. Then $h$ has a harmonic conjugate $\widetilde{h}:V_{H}^{\circ}\rightarrow{\mathbb{R}}$ on $H$.
\end{prop}
\begin{rem}
Ford- Fulkerson duality is a general statement that holds for any finite network. However, in the case of discrete rectangles, it has a particularly simple proof stemming from the fact that if $h$ is the function on $V^{\bullet}$ that is equal to $0$ on $[A^{\bullet}, B^{\bullet}]$, $1$ on $[C^{\bullet},D^{\bullet}]$ and is harmonic elsewhere so that: 
\begin{equation*}
    \lambda^{\bullet}=\lambda([A^{\bullet},B^{\bullet}]\leftrightarrow{[C^{\bullet}, D^{\bullet}]};G)=\frac{1}{\mathcal{E}^{\bullet}(h)}
\end{equation*}
its harmonic conjugate $\widetilde{h}$ is (up to an additive constant) equal to $0$ on $[B^{\circ}, C^{\circ}]$, $\frac{1}{\lambda^{\bullet}}$ on $[D^{\circ}, A^{\circ}]$ and is harmonic elsewhere. Furthermore, by the discrete Cauchy-Riemann equations, $\mathcal{E}^{\circ}(\widetilde{h})=\mathcal{E}^{\bullet}(h)$. Hence:  
\begin{equation*}
    \lambda^{\circ}=\lambda([B^{\circ}, C^{\circ}]\leftrightarrow{[D^{\circ},A^{\circ}]};G^{\circ})=\frac{(1/\lambda^{\bullet})^{2}}{\mathcal{E}^{\circ}(\widetilde{h})}=\frac{(1/\lambda^{\bullet})^{2}}{(1/\lambda^{\bullet})}=\frac{1}{\lambda^{\bullet}}
\end{equation*}
\end{rem}

\subsection{Tilings of Rectangles}
\label{subsec: Tilings of Rectangles}

Suppose $\Omega\subseteq{\mathbb{C}}$ is a Jordan domain with analytic boundary and distinguished boundary points $A, B, C, D$ listed in counterclockwise order. Let $[A,B]_{\partial{\Omega}}, [B,C]_{\partial{\Omega}}, [C,D]_{\partial{\Omega}}, [D,A]_{\partial{\Omega}}$ denote the closed boundary arcs stretching counterclockwise from $A$ to $B$, $B$ to $C$, $C$ to $D$ and $D$ to $A$ along $\partial{\Omega}$. Let $(A,B)_{\partial{\Omega}}, (B,C)_{\partial{\Omega}}, (C,D)_{\partial{\Omega}}, (D,A)_{\partial{\Omega}}$ be the corresponding open boundary arcs. Let $\widetilde{h}$ be the solution to the following boundary value problem: 
\begin{align*}
    \widetilde{h}(x)&=0 \text{ for } x\in{[B,C]_{\partial{\Omega}}} \\
    \widetilde{h}(x)&=1 \text{ for } x\in{[A,D]_{\partial{\Omega}}} \\
    \Delta{\widetilde{h}(x)}&=0 \text{ for } x\in{\Omega} \\
    \partial_{n}\widetilde{h}(x)&=0 \text{ for } x\in{(B, C)_{\partial{\Omega}}\cup{(D, A)_{\partial{\Omega}}}}
\end{align*}
The existence and uniqueness of the solution to this problem is clear by conformal invariance. Namely, if $L$ is the extremal length from $[A,B]_{\partial{\Omega}}$ to $[C,D]_{\partial{\Omega}}$ in $\Omega$ and $\mathcal{R}_{L}:=(0,L)\times{(0,1)}$ then there is a unique conformal map $\phi:\Omega\rightarrow{\mathcal{R}_{L}}$ sending $A,B,C,D$ to the corners of $\mathcal{R}_{L}$. Since our boundary value problem is conformally invariant, $\widetilde{h}(x)=Im(\phi(x))$. Furthermore, the conjugate harmonic function $h(x)=Re(\phi(x))$ satisfies:
\begin{align*}
    h(x)&=0 \text{ for } x\in{[A,B]_{\partial{\Omega}}} \\
    h(x)&=L \text{ for } x\in{[C,D]_{\partial{\Omega}}} \\
    \partial_{n}h(x)&=0 \text{ for } x\in{(B,C)_{\partial{\Omega}}\cup{(D,A)_{\partial{\Omega}}}}
\end{align*}
Thus, if we are presented with a Jordan domain $\Omega$ with four distinguished boundary points $A, B, C, D$ listed in counterclockwise order and $\phi:\Omega\rightarrow{\mathcal{R}_{L}}$ is the conformal map that maps $A, B, C, D$ to the corners of $\mathcal{R}_{L}$, we can intuitively think of the real and imaginary parts of this conformal map as solving the aforementioned boundary value problems. \\ \\
Suppose $(G,c)$ is a discrete rectangle with distinguished boundary points $A^{\bullet}, B^{\bullet}, C^{\bullet}, D^{\bullet}\in{\partial{V^{\bullet}}}$ listed in counterclockwise order, giving rise to primal boundary arcs $[A^{\bullet},B^{\bullet}],[C^{\bullet},D^{\bullet}]\subseteq{\partial{V^{\bullet}}}$ and corresponding dual boundary arcs $[B^{\circ},C^{\circ}], [D^{\circ}, A^{\circ}]\subseteq{\partial{V^{\circ}}}$. Let $\widetilde{h}$ be the solution to the following boundary value problem on $(G^{\circ},c^{\circ})$:
\begin{align*}
    \widetilde{h}(x)&=0 \hspace{2pt}\text{ for $x\in{[B^{\circ},C^{\circ}]}$} \\
    \widetilde{h}(x)&=1 \hspace{2pt}\text{ for $x\in{[D^{\circ}, A^{\circ}]}$}\\
    \Delta^{\circ}\widetilde{h}(x)&=0 \hspace{2pt}\text{ for $x\in{V^{\circ}\setminus{([B^{\circ},C^{\circ}]\cup{[D^{\circ},A^{\circ}]})}}$}
\end{align*}
Let $h$ be the solution to the following boundary value problem on $(G^{\bullet},c^{\bullet})$:
\begin{align*}
    h(x)&=0 \hspace{2pt}\text{ for $x\in{[A^{\bullet}, B^{\bullet}]}$}\\
    h(x)&=L \hspace{2pt}\text{ for $x\in{[C^{\bullet},D^{\bullet}]}$} \\
    \Delta^{\bullet}{h}(x)&=0 \hspace{2pt}\text{ for  $x\in{V^{\bullet}\setminus{([A^{\bullet}, B^{\bullet}]\cup{[C^{\bullet},D^{\bullet}]})}}$}
\end{align*}
where $L$ is the effective resistance between $[A^{\bullet}, B^{\bullet}]$ and $[C^{\bullet},D^{\bullet}]$ in $(G^{\bullet},c^{\bullet})$. Just as in the continuous setting, $\widetilde{h}$ is the harmonic conjugate of $h$.
Since they are defined in terms of analogous boundary value problems, the functions $h$ and $\widetilde{h}$ are discrete analogues of the real and imaginary parts of the uniformizing conformal map that takes a simply connected domain with four distinguished prime ends to a rectangle so that the four distinguished prime ends are mapped to the four corners of the rectangle. \\ \\
Suppose $(G,c)$ is a discrete rectangle with distinguished boundary points $A^{\bullet}, B^{\bullet}, C^{\bullet}, D^{\bullet}\in{\partial{V^{\bullet}}}$ listed in counterclockwise order, giving rise to primal boundary arcs $[A^{\bullet},B^{\bullet}],[C^{\bullet},D^{\bullet}]\subseteq{\partial{V^{\bullet}}}$ and corresponding dual boundary arcs $[B^{\circ},C^{\circ}], [D^{\circ}, A^{\circ}]\subseteq{\partial{V^{\circ}}}$. Let $h:V^{\bullet}\rightarrow{\mathbb{R}}$, $\widetilde{h}:V^{\circ}\rightarrow{\mathbb{R}}$ be conjugate harmonic functions defined as above. For any interior face $f$ of $G$ with incident vertices $x,y,u,v$ where $x,y\in{V^{\bullet}}$, $u,v\in{V^{\circ}}$, the image of $f$ under the \textbf{tiling map} $\phi$ is defined as follows: 
$$
\phi(f)=[h(x),h(y)]\times{[\widetilde{h}(u), \widetilde{h}(v)]}
$$
where the order of $x,y$ and $u,v$ is chosen so that: 
\begin{align*}
h(x)&\leq{h(y)}, & \widetilde{h}(u)&\leq{\widetilde{h}(v)}
\end{align*}
As the name suggests, $\phi$ corresponds to a tiling of the rectangle $\mathcal{R}_{L}$ by smaller subrectangles: 
\begin{thm} 
    (Theorem 4.31 of \cite{BSST40}) Suppose $\phi$ is the tiling map associated with the discrete rectangle $(G,c)$ with distinguished boundary points $A^{\bullet}, B^{\bullet}, C^{\bullet}, D^{\bullet}\in{\partial{V^{\bullet}}}$ listed in counterclockwise order. Then for any pair of distinct inner faces $f, f'$ of $G$, the rectangles $\phi(f)$ and $\phi(f')$ have disjoint interiors. Furthermore, if $F_{\text{in}}$ is the set of interior faces of $G$, 
    $$ \bigcup\limits_{f\in{F_{\text{in}}}}\phi(f)=[0,L]\times{[0,1]}
    $$
\end{thm}
\noindent Since $h$ and $\widetilde{h}$ are conjugate, the aspect ratio of the rectangle $\phi(f)$ corresponding to the face $f$ of $G$ with incident vertices $x,y\in{V^{\bullet}}$ and $u,v\in{V^{\circ}}$ is precisely the resistance of the primal edge $\{x,y\}\in{E^{\bullet}}$ or equivalently, the conductance of the dual edge $\{u,v\}\in{E^{\circ}}$. 
\subsection{Orthodiagonal Approximations of Planar Domains} 
\label{subsec: OD map intro}

Our reason for working in the level of generality that we did in Sections \ref{Extremal Length and Planar Networks} and \ref{subsec: Tilings of Rectangles} was to illustrate the scope of the theory of planar electrical networks. That said, in order to get a sequence of discrete conformal maps that converge to the relevant uniformizing comformal map in the limit, we need to do two things:
\begin{enumerate}
    \item Fix an embedding of our graph in the complex plane.
    \item Pick a conformal metric that is tied to the geometry of this embedding.
\end{enumerate}
With this in mind, an \textbf{orthodiagonal map} is a finite, bipartite quadrangulation with boundary $G=(V^{\bullet}\sqcup{V^{\circ}}, E)$ with a fixed, proper embedding in the  plane so that: \begin{itemize}
    \item Each edge is a straight line segment.
    \item Each interior face is a quadrilateral with orthogonal diagonals.
\end{itemize}
We allow non- convex quadrilaterals, whose diagonals do not intersect. We endow this with a conformal metric $c:E^{\bullet}\sqcup{E^{\circ}}\rightarrow{(0,\infty)}$ defined as follows: 
\begin{align*}
    c(e^{\bullet})&=\frac{|e^{\circ}|}{|e^{\bullet}|}, & c(e^{\circ})&=\frac{|e^{\bullet}|}{|e^{\circ}|} 
\end{align*}
for $e^{\circ}\in{E^{\circ}}$, $e^{\bullet}\in{E^{\bullet}}$ so that $e^{\circ}$ is the dual edge corresponding to the primal edge $e^{\bullet}$. For any edge $e\in{E^{\bullet}\sqcup{E^{\circ}}}$, $|e|$ is the length of the edge $e$ in our embedding. Recall that $E^{\bullet}$ and $E^{\circ}$ are the edges of $G^{\bullet}$ and $G^{\circ}$ respectively, not the edges of $G$. That is, they correspond to the diagonals of interior faces of $G$.   \\ \\
To make it clear that the discussion that follows is not totally vacuous, observe that the square lattice, the triangular lattice and the hexagonal lattice all have this property that primal and dual edges are orthogonal. More generally, finite subdomains of isoradial lattices, which have been widely studied in the context of critical statistical physics in 2D (i.e. see \cite{Ke02} and \cite{CS11}), are precisely orthodiagonal maps whose faces are all rhombii. Furthermore, as a consequence of the double circle packing theorem, a wide variety of planar graphs admit an orthodiagonal embedding (see Section 2 of \cite{GJN20}).   
\\ \\
While our choice of conformal metric might seem strange at first, observe that if $G$ is an orthodiagonal map with conformal metric $c$ as above:
\begin{itemize}
     \item simple random walk on $(G^{\bullet}, c^{\bullet})$ and $(G^{\circ},c^{\circ})$ is a martingale. 
     \item as a Markov chain, simple random walk on $(G^{\bullet}, c^{\bullet})$ and $(G^{\circ},c^{\circ})$ is reversible.
\end{itemize}
The orthogonality of edges and dual edges gives us a natural way to write down the Cauchy- Riemann equations on an orthodiagonal map. Namely, a function $F:V^{\bullet}\sqcup{V^{\circ}}\rightarrow{\mathbb{C}}$ is said to be \textbf{discrete holomorphic} if for every interior face $Q$ of $G$ with primal diagonal $e^{\bullet}=\{v_{1},v_{2}\}$ and dual diagonal $e^{\circ}=\{w_{1}, w_{2}\}$ we have: 
\begin{equation}
\label{OD_CR}
    \frac{F(v_{2})-F(v_{1})}{v_{2}-v_{1}}= \frac{F(w_{2})-F(w_{1})}{w_{2}-w_{1}}
\end{equation}
From this definition, it follows that:
\begin{itemize}
    \item discrete contour integrals vanish if the integrand is discrete holomorphic. That is, if $F$ is discrete holomorphic on $G$ and $\gamma$ is a simple, closed, directed curve in $G$ so that  the faces of $G$ enclosed by $\gamma$ are all interior faces of $G$, then: 
    $$
    \sum_{\substack{\vec{e}\hspace{1pt}\in{\gamma} \\ \vec{e}=(e^{-}, e^{+})}}\big(F(e^{-})+F(e^{+})\big)(e^{+}-e^{-})=0
    $$
    \item the real and imaginary parts of any discrete holomorphic function are harmonic with respect to the edge weights $c^{\bullet}$ and $c^{\circ}$. That is, $\text{Re}(F)|_{V^{\bullet}}, \text{Im}(F)|_{V^{\bullet}}$ are harmonic on $(G^{\bullet}, c^{\bullet})$ and $\text{Re}(F)|_{V^{\circ}}, \text{Im}(F)|_{V^{\circ}}$ are harmonic on $(G^{\circ}, c^{\circ})$. Moreover, $\text{Im}(F)|_{V^{\circ}}$ is the conjugate harmonic function of $\text{Re}(F)|_{V^{\bullet}}$ and $\text{Re}(F)|_{V^{\circ}}$ is the conjugate harmonic function of $\text{Im}(F)|_{V^{\bullet}}$.
\end{itemize}
In short, orthodiagonal maps provide us with a notion of discrete complex analysis. \\ \\
An \textbf{orthodiagonal rectangle} is an orthodiagonal map $G$ with a unique distinguished outer face whose boundary is a simple closed curve with four distinguished boundary points $A^{\bullet},B^{\bullet},C^{\bullet},D^{\bullet}\subseteq\partial{V}^{\bullet}$ listed in counterclockwise order. As in the case of discrete conformal rectangles, these give rise to primal boundary arcs $[A^{\bullet}, B^{\bullet}],[C^{\bullet}, D^{\bullet}]$ and corresponding dual arcs $[B^{\circ}, C^{\circ}], [D^{\circ}, A^{\circ}]$. Given an orthodiagonal map $G$, let $\widehat{G}$ denote the subdomain of $\mathbb{C}$ formed by taking the interior of the union of the faces, vertices and edges of $G$. 
\begin{figure}[H]
\centering
\includegraphics[scale=0.55]{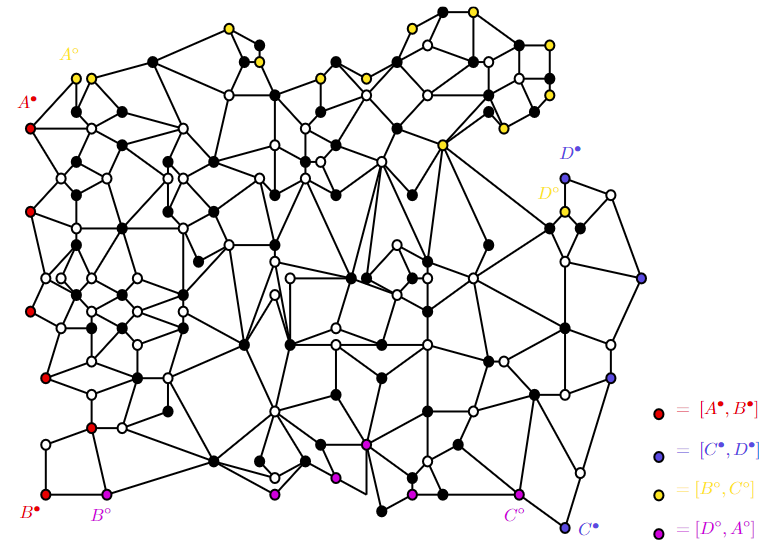}
\caption{An orthodiagonal rectangle with its four distinguished boundary arcs.}
\end{figure}
\noindent Suppose $\Omega$ is a connected, proper subdomain of $\mathbb{C}$. $\gamma\subseteq{\overline{\Omega}}$ is a \textbf{crosscut} of $\Omega$ if $\gamma=\eta([0,1])$ for some injective, continuous function $\eta:[0,1]\rightarrow{\overline{\Omega}}$ such that $\eta(0,1)\subseteq{\Omega}$ and $\eta(0), \eta(1)\in{\partial{\Omega}}$, where $\eta(0)\neq{\eta(1)}$. If $\gamma$ is a crosscut of $\Omega$, $\Omega\setminus{\gamma}$ has two connected components. By the Jordan curve theorem, the same is true of $\Omega\setminus{\gamma}$ when $\gamma$ is a simple closed curve in $\Omega$. Given disjoint subsets $A,B\subseteq{\Omega}$ we say that a simple closed curve or crosscut $\gamma$ \textbf{separates} $A$ and $B$ in $\Omega$ if $\gamma\cap{A}=\gamma\cap{B}=\emptyset$ and $A$ and $B$ lie in distinct connected components of $\Omega\setminus{\gamma}$. \\ \\
Fix $z_{0}\in{\Omega}$. For any $z,w\in{\Omega\setminus{\{z_{0}\}}}$, their \textbf{Carath\'eodory distance} with respect to the reference point $z_{0}$ is given by:
\begin{equation*}
    d_{Cara}^{z_{0}}(z,w):=\inf\{\text{length}(\gamma): \text{$\gamma$ is a simple closed curve or crosscut that separates $z$ and $w$ from $z_{0}$}\}
\end{equation*}
$d_{Cara}^{z_{0}}$ is a metric on $\Omega\setminus{\{z_{0}\}}$ that is locally equivalent to the usual Euclidean metric. The \textbf{Carath\'eodory compactification} $\Omega^{*}$ of $\Omega$ is the completion of $\Omega\setminus{\{z_{0}\}}$ with respect to $d_{Cara}^{z_{0}}$. As a topological space, the Carath\'eodory compactification $\Omega^{*}$ is independent of our choice of reference point $z_{0}\in{\Omega}$.  $\partial{\Omega^{*}}$ is known as the space of \textbf{prime ends} of $\Omega$. The prime ends of $\Omega$ can be interpretted geometrically as equivalence classes of chains of open sets in $\Omega$ converging to a point on the boundary. For details, see Section 3.1 of \cite{BRY14} or Section 2.4 of \cite{pommerenke}. Given disjoint subsets $A,B\subseteq{\partial{\Omega^{\ast}}}$, we say that a crosscut $\gamma$ of $\Omega$ \textbf{joins} $A$ and $B$ in $\Omega$ if one of the endpoints of $\gamma$ lies in $A$ and the other lies in $B$. If $A\subseteq{\Omega}$, $B\subseteq{\partial{\Omega^{\ast}}}$, we say that a crosscut $\gamma$ of $\Omega$  \textbf{joins} $A$ and $B$ in $\Omega$ if $A\cap{\gamma}=\emptyset$, $A$ is contained in one of the two connected components of $\Omega^{\ast}\setminus{\gamma}$, and $B$ intersects the connected component of $\Omega^{\ast}\setminus{\gamma}$ containing $A$.\\ \\
If $\Omega_{1}, \Omega_{2}$ are proper, simply connected subdomains of $\mathbb{C}$ and $\phi:\Omega_{1}\rightarrow{\Omega_{2}}$ is conformal, then $\phi$ extends to a homeomorphism $\phi:\Omega_{1}^{*}\rightarrow{\Omega_{2}^{*}}$. This tells us that, from the standpoint of complex analysis, the space of prime ends is the right notion of boundary for a proper, simply connected subdomain of $\mathbb{C}$. In particular, if $\phi:\Omega\rightarrow{\mathbb{D}}$ is the uniformizing conformal map that maps a simply connected domain  $\Omega\subsetneq{\mathbb{C}}$ to the unit disk, we have the following estimates for the modulus of continuity of $\phi$ and $\phi^{-1}$ with respect to the Carath\'eodory metric: 
\begin{prop}
    Suppose $\Omega\subsetneq{\mathbb{C}}$ is a bounded simply connected domain, $z_{0}\in{\Omega}$ and $\phi:\Omega\rightarrow{\mathbb{D}}$ is a uniformizing conformal that maps $\Omega$ to the unit disk so that $\phi(z_{0})=0$. Then there exists an absolute constant $C_{1}>0$ such that for any $x,y\in{\Omega}$:
    \begin{equation}
    \label{eqn: modulus of continuity of map to disk}
        |\phi(y)-\phi(x)|\leq{C_{1}\sqrt{\frac{d^{z_{0}}_{Cara}(x,y)}{|\phi'(z_{0})|}}}
    \end{equation}
   If additionally we know that $\Omega$ is bounded, there exists an absolute constant $C_{2}>0$ such that for any $x,y\in{\mathbb{D}}$, we have that:
    \begin{equation}
    \label{eqn: modulus of continuity of inverse map}
        d_{Cara}^{z_{0}}(\phi^{-1}(x), \phi^{-1}(y))\leq{C_{2}\sqrt{\frac{\text{Area}(\Omega)}{\log\big(\frac{1}{|x-y|}\big)}}}
    \end{equation}
\end{prop}
\noindent The modulus of continuity for $\phi$ in Equation \ref{eqn: modulus of continuity of map to disk} is a consequence of Beurling's estimate (see Proposition 3.85 of \cite{CIPbook}). The modulus of continuity for $\phi^{-1}$ follows from Wolff's lemma (see Proposition 2.2 of \cite{pommerenke}).  As a consequence, if $\Omega\subsetneq{\mathbb{C}}$ is simply connected, $\Omega^{*}$ is homeomorphic to the closed unit disk $\overline{\mathbb{D}}$ and the space of prime ends $\partial{\Omega^{*}}$ is homeomorphic to $S^{1}$. If $x,y\in{\partial{\Omega^{\ast}}}$ are prime ends, let $[x,y]_{\partial{\Omega^{\ast}}}$ denote the arc along $\partial{\Omega^{\ast}}$ that travels from $x$ to $y$, counterclockwise.\\ \\
A \textbf{conformal rectangle} is a bounded, simply connected domain $\Omega$ along with four distinguished prime ends $A,B,C,D$, listed in counterclockwise order. Recalling our discussion in Section \ref{subsec: Organization}, given a conformal rectangle $(\Omega,A,B,C,D)$ and a sequence of orthodiagonal rectangles $\big((G_{n},A_{n}^{\bullet}, B_{n}^{\bullet}, C_{n}^{\bullet}, D_{n}^{\bullet})\big)_{n=1}^{\infty}$ that are better and better approximations of $(\Omega,A,B,C,D)$, we want to show that the associated tiling maps converge to the conformal map from $\Omega$ to a rectangle $\mathcal{R}_{L}$ so that the prime ends $A,B,C,D$ are mapped to the four corners of $\mathcal{R}_{L}$, and in particular, $\phi(A)=i$. This of course begs the question: what does it mean for an orthodiagonal rectangle $(G,A^{\bullet}, B^{\bullet}, C^{\bullet}, D^{\bullet})$ to be a good approximation of $(\Omega,A,B,C,D)$? One natural requirement is that the boundary arcs $[A^{\bullet},B^{\bullet}]$, $[B^{\circ},C^{\circ}]$, $[C^{\bullet},D^{\bullet}]$, $[D^{\circ},A^{\circ}]$ of $G$ should be close to the corresponding continuous boundary arcs $[A,B]_{\partial{\Omega^{\ast}}}$, $[B,C]_{\partial{\Omega^{\ast}}}$, $[C,D]_{\partial{\Omega^{\ast}}}$, $[D,A]_{\partial{\Omega^{\ast}}}$ of $\Omega$. To be precise, since the Carath\'eodory metric is the right notion of distance for a general simply connected domain, a natural requirement is that the discrete boundary arcs are close to the corresponding continuous boundary arcs in Carath\'eodory metric. However, to define the Carath\'eodory metric, we need to introduce a reference point, which a priori isn't part of our setup. To avoid this, we will instead use a closely related quantity, whose definition doesn't require the introduction of a reference point.\\ \\  
If $\gamma$ is a crosscut of $\Omega$, for any $z\in{\Omega\setminus{\gamma}}$, let $\mathcal{N}_{\gamma}^{z}$ denote the component of $\Omega\setminus{\gamma}$ containing $z$. Let $S, S'$ be disjoint compact subsets of $\partial{\Omega^{\ast}}$. Then for any $z\in{\Omega^{\ast}}$, we define the \textbf{crosscut distance} $d_{cc}^{\Omega}(z,S;S')$ from $z$ to $S$, away from $S'$ in $\Omega$, by: 
\begin{equation*}
d_{cc}^{\Omega}(z,S;S'):=\inf\{\text{length}(\gamma): \text{$\gamma$ is a crosscut of $\Omega$ that separates $z$ from $S'$ such that $\overline{\mathcal{N}_{\gamma}^{z}}\cap{S}\neq{\emptyset}$}\}
\end{equation*}
where $\overline{\mathcal{N}_{\gamma}^{z}}$ is the closure of $\mathcal{N}_{\gamma}^{z}$ with respect to the Carath\'eodory metric. Similarly, for a compact subset $C$ of $\Omega^{\ast}$, its crosscut distance to $S$, away from $S'$, is given by: 
\begin{equation*}
d_{cc}^{\Omega}(C,S;S'):=\sup_{z\in{C}}d^{\Omega}_{cc}(z,S;S')
\end{equation*}
The supremum on the right hand side is actually a maximum. To see this, observe that $d^{\Omega}_{cc}(z,S;S')$ is locally Lipschitz as a function of $z$. Namely,  
\begin{equation*}
|d_{cc}^{\Omega}(z,S;S')-d^{\Omega}_{cc}(w,S;S')|\leq{2d_{\Omega}(z,w)}
\end{equation*}
where: 
\begin{equation*}
d_{\Omega}(z,w):=\inf\{\text{length}(\gamma): \text{$\gamma$ is a smooth curve from $z$ to $w$ in $\Omega$} \}
\end{equation*}
In particular, if the line segment from $z$ to $w$ is contained in $\Omega$, we have that:
\begin{equation*}
|d_{cc}^{\Omega}(z,S;S')-d_{cc}^{\Omega}(w,S;S')|\leq{2|z-w|}
\end{equation*}
Given a conformal rectangle $(\Omega,A,B,C,D)$, we say that the orthodiagonal rectangle $(G,A^{\bullet}, B^{\bullet}, C^{\bullet}, D^{\bullet})$ is a \textbf{($\delta$, $\varepsilon$)- good interior approximation} of $(\Omega, A,B,C,D)$ if $\widehat{G}\subseteq{\Omega}$, $|e|<\varepsilon$ for all edges $e\in{E}$, and:
\begin{align*} 
    &d_{cc}^{\Omega}([A^{\bullet},B^{\bullet}],[A,B]_{\partial{\Omega^{\ast}}}; [C,D]_{\partial{\Omega^{\ast}}})<\delta, & &d_{cc}^{\Omega}([C^{\bullet},D^{\bullet}], [C,D]_{\partial{\Omega^{\ast}}}; [A,B]_{\partial{\Omega^{\ast}}})<\delta \\
    &d_{cc}^{\Omega}([B^{\circ},C^{\circ}], [B,C]_{\partial{\Omega^{\ast}}}; [D,A]_{\partial{\Omega^{\ast}}})<\delta, & &d_{cc}^{\Omega}([D^{\circ},A^{\circ}], [D,A]_{\partial{\Omega^{\ast}}}; [B,C]_{\partial{\Omega^{\ast}}})<\delta
\end{align*} 
\noindent This notion of what it means for an orthodiagonal map with four distinguished boundary points to approximate a conformal rectangle is more involved than the approximation scheme used in \cite{ALP23} where it is only required that $[A^{\bullet}, B^{\bullet}]$, $[B^{\circ}, C^{\circ}]$, $[C^{\bullet}, D^{\bullet}]$, and $[D^{\circ}, A^{\circ}]$ are close to the corresponding continuous boundary arcs $[A,B]_{\partial{\Omega^{\ast}}}$, $[B,C]_{\partial{\Omega^{\ast}}}$, $[C,D]_{\partial{\Omega^{\ast}}}$, $[D,A]_{\partial{\Omega^{\ast}}}$ in Hausdorff distance. Additionally, in \cite{ALP23}, the approximation doesn't need to be be an interior approximation: it is not required that $\widehat{G}\subseteq{\Omega}$. \\ \\
The upside of our approach is that it is more general. Firstly, observe that if we restrict our attention to interior approximations, if $[A^{\bullet}, B^{\bullet}]$ and $[A,B]_{\partial{\Omega^{\ast}}}$ are within $\delta$ of one another in Hausdorff distance, then: 
\begin{equation}
    d_{cc}^{\Omega}([A^{\bullet},B^{\bullet}], [A,B]_{\partial{\Omega^{\ast}}}; [C,D]_{\partial{\Omega^{\ast}}})<2\delta
\end{equation}
The same is true for the other three boundary arcs. More importantly, using this schema, we can approximate any conformal rectangle $(\Omega,A,B,C,D)$. If we replace crosscut distance with Hausdorff distance, this is no longer true. To unpack this statement, we'll first need to make a brief digression. Recall that a bounded, simply connected domain is \textbf{Jordan} if its boundary is a simple closed curve. If $(\Omega,A,B,C,D)$ is a conformal rectangle, and $\Omega'$ is a Jordan subdomain of $\Omega$ with distinguished boundary points $A',B',C',D'$ in counterclockwise order, we say that $(\Omega',A',B',C',D')$ is a \textbf{$\delta$- good interior approximation} of $\Omega$ if: 
\begin{align*}
    d_{cc}^{\Omega}([A',B']_{\partial{\Omega'}},[A,B]_{\partial_{\Omega^{\ast}}};[C,D]_{\partial{\Omega^{\ast}}})&<\delta & d_{cc}^{\Omega}([B',C']_{\partial{\Omega'}},[B,C]_{\partial_{\Omega^{\ast}}};[D,A]_{\partial{\Omega^{\ast}}})&<\delta \\ 
    d_{cc}^{\Omega}([C',D']_{\partial{\Omega'}},[C,D]_{\partial_{\Omega^{\ast}}};[A,B]_{\partial{\Omega^{\ast}}})&<\delta & d_{cc}^{\Omega}([D',A']_{\partial{\Omega'}},[D,A]_{\partial_{\Omega^{\ast}}};[B,C]_{\partial{\Omega^{\ast}}})&<\delta
\end{align*} 
Our reason for restricting to Jordan domains here has to do with the fact that for a bounded, simply connected domain $\Omega$, $\partial{\Omega}$ and $\partial{\Omega^{\ast}}$ are homeomorphic if and only if $\Omega$ is Jordan. In other words, for a Jordan domain, the space of prime ends coincides with its boundary in the usual sense (with respect to the Euclidean metric). Otherwise, it is not clear what it means for a prime end of $(\Omega')^{\ast}$ to be close to a boundary arc of $\Omega$ in crosscut distance, since a prime end of $(\Omega')^{\ast}$ is not a point in $\Omega$. With this in mind, we have the following result:
\begin{prop}
\label{prop: we can approximate any simply connected domain}
    For any conformal rectangle $(\Omega,A,B,C,D)$ and any $\delta>0$ we can find a Jordan subrectangle $(\Omega',A',B',C',D')$ with analytic boundary arcs so that $(\Omega',A',B',C',D')$ is a $\delta$- good interior approximation of $(\Omega,A,B,C,D)$.
\end{prop}
\noindent For a proof of Proposition \ref{prop: we can approximate any simply connected domain}, see Appendix \ref{appendix: elementary results in topology and complex analysis}. Observe that if the orthodiagonal rectangle $(G,A^{\bullet},B^{\bullet}, C^{\bullet},D^{\bullet})$ is a $(\delta,\varepsilon)$- good interior approximation of the conformal rectangle $(\Omega,A,B,C,D)$, then the corresponding conformal rectangle $(\widehat{G},A^{\bullet},B^{\bullet}, C^{\bullet},D^{\bullet})$, is a $(\delta+2\varepsilon)$-good interior approximation of $(\Omega,A,B,C,D)$. Thus, at least heuristically, Proposition \ref{prop: we can approximate any simply connected domain} tells us that given a conformal rectangle $(\Omega,A,B,C,D)$ and $\delta, \varepsilon>0$, we can find a $(\delta,\varepsilon)$- good orthodiagonal approximation of $(\Omega,A,B,C,D)$, provided the mesh of our orthodiagonal map is sufficiently fine. \\ \\
In contrast, consider the comb domain $\Omega$ defined as follows: 
\begin{equation*}
\Omega:=(0,2)\times{(0,1)}\setminus{\Big(\big(\bigcup_{n=0}^{\infty}\{2^{-2n}\}\times{[2^{-2n-1},1)}\big)\bigcup{\big(\bigcup_{n=0}^{\infty}\{2^{-2n-1}\}\times{(0,1-2^{-2n-1}]}\big)}\Big)}
\end{equation*}
The points $(1,\frac{1}{2})$, $(2,0)$, $(2,1)$ on the boundary of $\Omega$ as well as the line segment $\{0\}\times{[0,1]}$ all correspond to a single prime end. Take $A=\{0\}\times{[0,1]}$, $B=(1,\frac{1}{2})$, $C=(2,1)$, $D=(2,0)$. Suppose we want to approximate the topological rectangle $(\Omega,A,B,C,D)$ by a suborthodiagonal map of $\varepsilon\mathbb{Z}^{2}$ for some very small $\varepsilon>0$. Since the line segment $\{0\}\times{[0,1]}$ corresponds to a single prime end $A$ of $\Omega$, regardless of $\varepsilon$ and our choice of suborthodiagonal map $G$ of $\varepsilon\mathbb{Z}^{2}$, it is not possible to choose $A^{\bullet}, B^{\bullet}, C^{\bullet}, D^{\bullet}\in{\partial{V}^{\bullet}}$ so that the corresponding discrete boundary arcs are all within $\delta$ of their continuum analogues, for $\delta$ sufficiently small.\\ \\ 
\begin{figure}[H]
    \centering
    \includegraphics[scale=0.42]{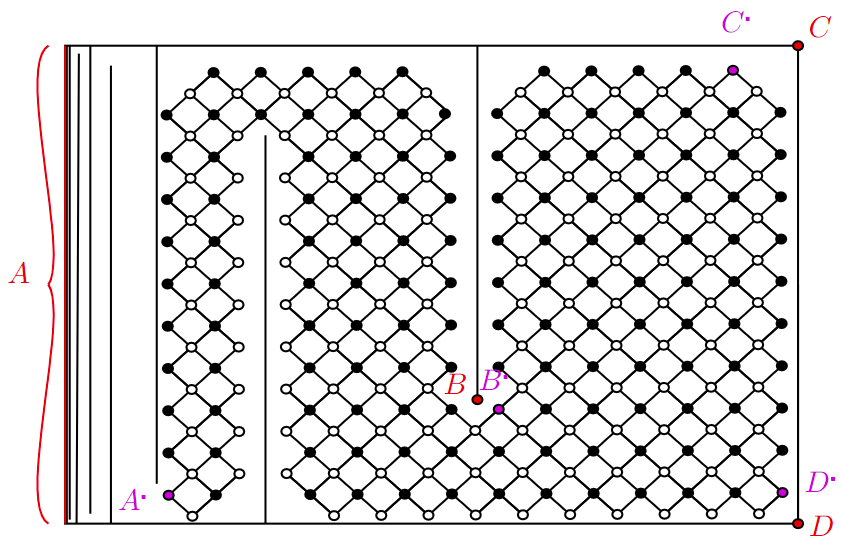}  
    \caption{The comb domain $\Omega$ with its four distinguished prime ends and an orthodiagonal approximation $G$ of $\Omega$. Notice that regardless of how fine the mesh of $G$ is, there is no choice of $A^{\bullet}, B^{\bullet}, C^{\bullet}, D^{\bullet}\in{\partial{V^{\bullet}}}$ so that each of the discrete boundary arcs of $(G,A^{\bullet}, B^{\bullet}, C^{\bullet}, D^{\bullet})$ is close to its continuous analogue.}
\end{figure}
\noindent In Section \ref{subsec: Tilings of Rectangles} we defined the tiling map associated with a discrete rectangle. Without a fixed embedding, the faces of a discrete rectangle are purely combinatorial objects. Having fixed an embedding, the faces of an orthodiagonal rectangle are honest- to- goodness subsets of the plane. Hence, we can think of the tiling map $\phi$ associated with an orthodiagonal rectangle $G$, as a function $\phi:\widehat{G}\rightarrow{\mathbb{C}}$. We do this by choosing for each interior face $f$ a homeomorphism that maps the quadrilateral $f$ to the corresponding rectangle $\phi(f)$. Unfortunately, this means that the tiling map  $\phi:\widehat{G}\rightarrow{\mathbb{C}}$ depends on our choice of homeomorphism. There is also some ambiguity as to the definition of $\phi$ on the edges of $G$, since each edge is shared by two distinct faces. That said, by the regularity estimates in Section \ref{subsec: modulus of continuity for tiling maps} we'll see that this isn't a concern, since our choice of homeomorphism doesn't impact the convergence we are looking for.
\begin{figure}[H]
    \centering
    \includegraphics[scale=0.54]{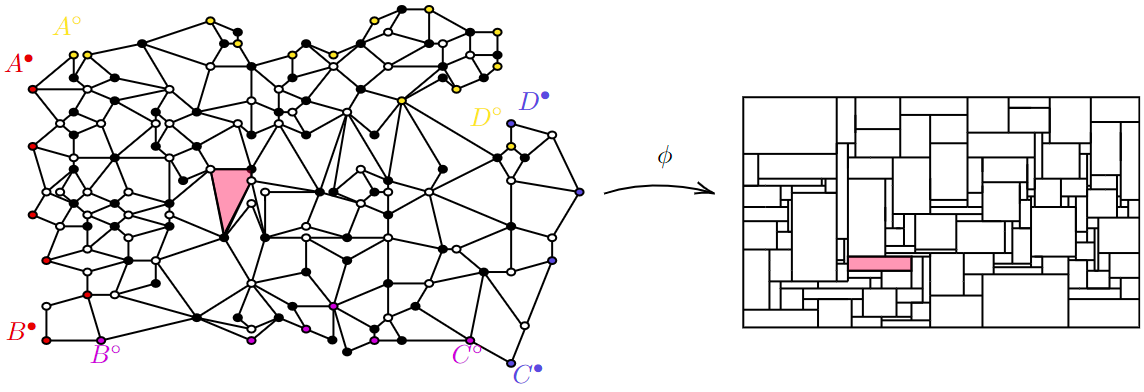}
\caption{The tiling map $\phi$ associated with an orthodiagonal rectangle.} 
\end{figure}
\noindent Having established all the requisite terminology, we can now state precisely the main theorem we intend to prove:  
\begin{thm}
\label{thm: convergence of rectangle tiling maps}
    Suppose $(\Omega,A,B,C,D)$ is a conformal rectangle and $\big((G_{n},A^{\bullet}_{n}, B^{\bullet}_{n},C^{\bullet}_{n},D^{\bullet}_{n})\big)_{n=1}^{\infty}$ is a sequence of orthodiagonal rectangles so that for each $n\in{\mathbb{N}}$, $(G_{n}, A^{\bullet}_{n}, B^{\bullet}_{n}, C^{\bullet}_{n}, D^{\bullet}_{n})$ is a $(\delta_{n}, \varepsilon_{n})$- good interior approximation of $\Omega$, where:
    $$
    (\delta_{n},\varepsilon_{n})\rightarrow{(0,0)} \hspace{5pt}\text{as $n\rightarrow{\infty}$}
    $$
    Let $\phi_{n}:\widehat{G}_{n}\rightarrow{[0,L_{n}]\times{[0,1]}}$ be the corresponding tiling maps, where $L_{n}$ is the discrete extremal length between $[A_{n}^{\bullet},B_{n}^{\bullet}]$ and $[C_{n}^{\bullet}, D_{n}^{\bullet}]$ in $G_{n}^{\bullet}$. Then: 
    \begin{equation*}
    \phi_{n}\rightarrow{\phi} \hspace{5pt}\text{uniformly on compacts as $n\rightarrow{\infty}$}
    \end{equation*}
    where $\phi$ is the conformal map from $\Omega$ to the rectangle $(0,L)\times{(0,1)}$ so that the prime ends $A,B,C,D$ are mapped to the four corners of the rectangle and in particular, $\phi(A)=i$. Here, $L$ is the extremal length between the arcs $[A,B]_{\partial{\Omega^{\ast}}}$ and $[C,D]_{\partial{\Omega^{\ast}}}$ in $\Omega$. In particular, it follows that: 
   \begin{equation*}
   L_{n}\rightarrow{L} \hspace{5pt} \text{as $n\rightarrow{\infty}$}
   \end{equation*}
\end{thm} 

\section{Precompactness of the Tiling Maps}
\label{sec: Precompactness of the Tiling Maps}

To prove Theorem \ref{thm: convergence of rectangle tiling maps} following the framework outlined in Section \ref{subsec: Organization}, we need to show that tiling maps corresponding to finer and finer orthodiagonal approximations of a conformal rectangle $(\Omega,A,B,C,D)$ are equicontinuous and uniformly bounded on compacts in $\Omega$. In this section, we address this by proving estimates for the norm and modulus of continuity of our tiling maps.\\ \\\
For both the norm and modulus of continuity, when doing this, we begin by proving the corresponding result in the continuous setting. The proof in the continuous setting motivates the proof in the discrete setting. Furthermore, we will need the continuous analogues of our tiling map estimates to prove Lemma \ref{lem: K in Omega_n for n sufficiently large} in Appendix \ref{appendix: elementary results in topology and complex analysis} and thereby all the results thereafter that rely on it. 

\subsection{Modulus of Continuity for the Limiting Conformal Map}
\label{subsec: Modulus of Continuity for Limiting Conformal Map}  

Suppose $(\Omega, A,B,C,D)$ is a conformal rectangle. Using the notation of Section \ref{subsec: OD map intro}, if $z,w\in{\Omega}$, 
\begin{align*}
    d_{cc}^{\Omega}(z,w):=\min\{&d^{\Omega}_{cc}(\{z,w\}, [A,B]_{\partial{\Omega^{\ast}}}; [C,D]_{\partial{\Omega^{\ast}}}), d_{cc}^{\Omega}(\{z,w\}, [B,C]_{\partial{\Omega^{\ast}}}; [D,A]_{\partial{\Omega^{\ast}}}), \\ &d_{cc}^{\Omega}(\{z,w\}, [C,D]_{\partial{\Omega^{\ast}}}; [A,B]_{\partial{\Omega^{\ast}}}), d_{cc}^{\Omega}(\{z,w\}, [D,A]_{\partial{\Omega^{\ast}}}; [B,C]_{\partial{\Omega^{\ast}}})\}
\end{align*}
In other words, $d_{cc}(z,w)$ is the length of the shortest crosscut of $\Omega$ that joins $z$ and $w$ to one boundary arc of $(\Omega,A,B,C,D)$ and separates them from the opposite boundary arc. Recall that for $z,w\in{\Omega}$, 
\begin{equation*}
    d_{\Omega}(z,w)=\inf\{\text{length}(\gamma): \text{$\gamma$ is a smooth curve from $z$ to $w$ in $\Omega$}\}
\end{equation*}
That is, $d_{\Omega}$ is the ambient metric on $\Omega$. Let $\phi:\Omega\rightarrow{\mathcal{R}_{L}}$ be the conformal map from $\Omega$ to the rectangle $\mathcal{R}_{L}$ such that the four prime ends $A,B,C,D$ of $\Omega$, listed in counterclockwise order, are mapped to the four corners of $\mathcal{R}_{L}$ and in particular, $\phi(A)=i$. The following theorem gives us a modulus of continuity for the real and imaginary parts of $\phi$ and therefore $\phi$ itself: 
\begin{thm}
\label{thm: modulus of continuity for conformal map}
    Suppose $(\Omega,A,B,C,D)$ is a conformal rectangle and $\phi:\Omega\rightarrow{\mathcal{R}_{L}}$ be the conformal map from $\Omega$ to the rectangle $\mathcal{R}_{L}$ so that the four prime ends $A,B,C,D$ of $\Omega$ are mapped to the four corners of $\mathcal{R}_{L}$ and in particular, $\phi(A)=i$. Here, $L$ is the extremal length between the boundary arcs $[A,B]_{\partial{\Omega^{\ast}}}$ and $[C,D]_{\partial{\Omega^{\ast}}}$ in $\Omega$. Define: 
    \begin{align*}
        d&=\inf\{\text{diameter}(\gamma): \text{$\gamma$ is a crosscut of $\Omega$ joining $[A,B]_{\partial{\Omega^{\ast}}}$ and $[C,D]_{\partial{\Omega^{\ast}}}$ in $\Omega$}\} \\ 
        d'&=\inf\{\text{diameter}(\gamma): \text{$\gamma$ is a crosscut of $\Omega$ joining $[B,C]_{\partial{\Omega^{\ast}}}$ and $[D,A]_{\partial{\Omega^{\ast}}}$ in $\Omega$}\}
    \end{align*}
    Let $h$ and $\widetilde{h}$ be the real and imaginary parts of $\phi$, respectively. Then for any $x,y\in{\Omega}$ so that $(d_{\Omega}(x,y)\wedge{d_{cc}^{\Omega}(x,y)})<\frac{d\wedge{d'}}{2}$, we have that:
    \begin{align*}
        |h(y)-h(x)|&\leq{\frac{2\pi}{\log\Big(\frac{d'}{2(d_{\Omega}(x,y)\wedge{d_{cc}^{\Omega}(x,y)})}\Big)}}, & |\widetilde{h}(y)-\widetilde{h}(x)|&\leq{\frac{2\pi{L}}{\log\Big(\frac{d}{2(d_{\Omega}(x,y)\wedge{d_{cc}^{\Omega}(x,y)})}\Big)}}
    \end{align*}
\end{thm}
\begin{proof}
    As per the theorem statement, let $h$ be the real part of $\phi$. Fix $x,y\in{\Omega}$. If $h(x)=h(y)$, the desired result holds. Otherwise, suppose WLOG that $h(x)<h(y)$. We now consider two cases: \\ \\
    \textbf{\underline{Case 1: $d_{\Omega}(x,y)\leq{d_{cc}^{\Omega}(x,y)}$.}} \\ \\ 
    Consider the region: 
    \begin{equation*}
        \Omega_{x,y}=\{z\in{\Omega}:h(x)<h(z)<h(y)\}
    \end{equation*}
    This is simply connected. Furthermore, since $\phi$ maps $\Omega$ to the rectangle $(0,L)\times{(0,1)}$, $\phi$ maps $\Omega_{x,y}$ to the rectangle $(h(x),h(y))\times{(0,1)}$. Thus, if we think of $\Omega_{x,y}$ as a conformal rectangle with the distinguished boundary arcs: 
    \begin{align*}
        N&=\Omega_{z,w}^{\ast}\cap{[D,A]_{\partial{\Omega^{\ast}}}}, & E&=\{z\in{\Omega}:h(z)=h(y)\} \\ 
        S&= \Omega_{z,w}^{\ast}\cap{[D,A]_{\partial{\Omega^{\ast}}}}, & W&=\{z\in{\Omega}: h(z)=h(x)\}
    \end{align*}
    Here, $N$, $E$, $S$ and $W$ stand for ``North," ``East," ``South," and ``West." This is to emphasize that our picture is as follows: 
    \begin{figure}[H]
    \centering
    \includegraphics[scale=0.55]{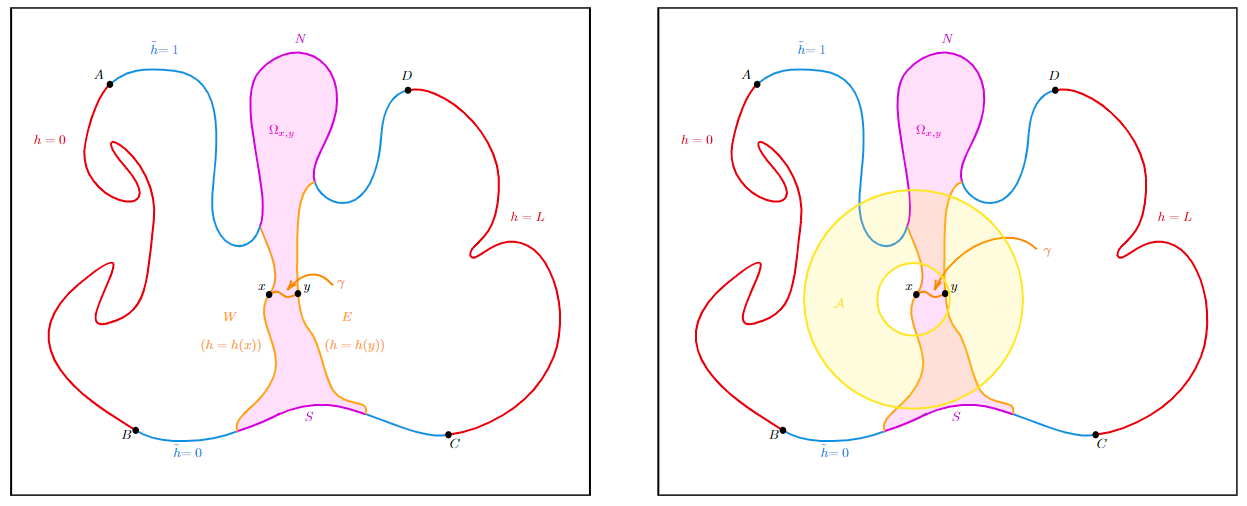}  
    \caption{On the left: The subrectangle $\Omega_{x,y}$ of $\Omega$ with its distinguished boundary arcs $N, E, S$ and $W$. On the right: Notice that any path from $N$ to $S$ in $\Omega_{x,y}$ must cross the annulus $\mathcal{A}$.}
    \end{figure}
    \noindent Then:
    \begin{equation*}
        \lambda(W\leftrightarrow{E};\Omega)=h(y)-h(x)
    \end{equation*}
    Having reinterpreted the quantity we're interested in as an extremal length, we can bound it from above by bounding the dual extremal length from below. \\ \\
    Fix $\varepsilon$ so that $0<\varepsilon<\frac{d'}{2}-d_{\Omega}(x,y)$. By the definition of $d_{\Omega}(x,y)$, we can find a smooth curve $\gamma$ in $\Omega$ from $x$ to $y$ so that $\text{length}(\gamma)<d_{\Omega}(x,y)+\varepsilon$. Since $x$ and $y$ lie on opposite boundary arcs of $\Omega_{x,y}$, while $\gamma$ may not be a crosscut of $\Omega_{x,y}$, there must exist a subarc $\gamma'$ of $\gamma$ with endpoints $x'\in{W}$, $y'\in{E}$ so that $\gamma'$ is a crosscut of $\Omega_{x,y}$. Consider the annulus: 
    \begin{equation*}
        \mathcal{A}=\{u\in{\mathbb{C}}: d_{\Omega}(x,y)+\varepsilon<|u-x'|<\frac{d'}{2}\}
    \end{equation*}
    where: 
    \begin{equation*}
        d'=\inf\{\text{diameter}(\gamma): \text{$\gamma$ is a curve joining $[B,C]_{\partial{\Omega^{\ast}}}$ and $[D,A]_{\partial{\Omega^{\ast}}}$ in $\Omega$}\}
    \end{equation*}
    Observe that: 
    \begin{enumerate}
        \item Since $\text{length}(\gamma')\leq{\text{length}(\gamma)}<d_{\Omega}(x,y)+\varepsilon$, the diameter of $\gamma'$ is at most $d_{\Omega}(x,y)+\varepsilon$. Hence, $\gamma'\subseteq{B(x',d_{\Omega}(x,y)+\varepsilon)}$. Since $\gamma'$ separates the boundary arcs $N$ and $S$ in $\Omega_{x,y}$, any path from $N$ to $S$ in $\Omega_{x,y}$ must intersect $\gamma'$ and therefore $B(x',d_{\Omega}(x,y)+\varepsilon)$. 
        \item On the other hand, since any curve from $N$ to $S$ in $\Omega_{x,y}$ is a curve from $[B,C]_{\partial{\Omega^{\ast}}}$ to $[D,A]_{\partial{\Omega^{\ast}}}$ in $\Omega$, such a curve must have diameter $\geq{d'}$. Hence, any curve from $N$ to $S$ in $\Omega_{x,y}$ must at some point lie outside the ball $B(x',\frac{d'}{2})$.
    \end{enumerate}
    Putting all this together, we see that any curve from $N$ to $S$ in $\Omega_{x,y}$ must cross the annulus $\mathcal{A}$. Consider the metric: 
    \begin{equation*}
        \rho(z)=\frac{1}{|z-x'|}1_{\mathcal{A}\cap{\Omega_{x,y}}}(z)
    \end{equation*}
    on $\Omega_{x,y}$. Observe that if $\eta$ is a $C^{1}$ curve that crosses the annulus $\mathcal{A}$ at least once, then:
    \begin{equation*}
        l_{\rho}(\eta)=\int_{\eta}\frac{|dz|}{|z-x'|}\geq{\int_{d_{\Omega}(x,y)+\varepsilon}^{d'/2}}\frac{dr}{r}=\log\Big(\frac{d'}{2(d_{\Omega}(x,y)+\varepsilon)}\Big)
    \end{equation*}
    Furthermore: 
    \begin{equation*}
        A(\rho)=\int_{\mathcal{A}\cap{\Omega_{x,y}}}\frac{1}{|z-x'|^{2}}dz_{1}dz_{2}\leq{\int_{\mathcal{A}}\frac{1}{|z-x'|^{2}}dz_{1}dz_{2}}=\int_{0}^{2\pi}\int_{d_{\Omega}(x,y)+\varepsilon}^{d'/2}\frac{1}{r}drd\theta=2\pi\log\Big(\frac{d'}{2(d_{\Omega}(x,y)+\varepsilon)}\Big)
    \end{equation*}    
    Hence, plugging $\rho$ into the variational problem for $\lambda(N\leftrightarrow{S};\Omega_{x,y})$, we have that: 
    \begin{align*}
        \lambda(N\leftrightarrow{S};\Omega_{x,y})\geq{\frac{\inf\limits_{\eta}\big(l_{\rho}(\eta)\big)^{2}}{A(\rho)}}\geq{\frac{\Big(\log\Big(\frac{d'}{2(d_{\Omega}(x,y)+\varepsilon)}\Big)\Big)^{2}}{2\pi\log\Big(\frac{d'}{2(d_{\Omega}(x,y)+\varepsilon)}\Big)}}=\frac{1}{2\pi}\log\Big(\frac{d'}{2(d_{\Omega}(x,y)+\varepsilon)}\Big) 
    \end{align*}
    By duality for continuous extremal extremal length: 
    \begin{equation*}
        \lambda(W\leftrightarrow{E};\Omega_{x,y})\cdot\lambda(N\leftrightarrow{S}; \Omega_{x,y})=1
    \end{equation*}
    Hence: 
    \begin{equation*}
        \lambda(N\leftrightarrow{S}; \Omega_{x,y})\leq{\frac{2\pi}{\log\Big(\frac{d'}{2(d_{\Omega}(x,y)+\varepsilon)}\Big) }}
    \end{equation*}
    Since $\varepsilon>0$ was arbitrary, letting $\varepsilon$ tend to $0$ in the above inequality, the desired result follows.\\ \\
    \textbf{\underline{Case 2: $d_{cc}^{\Omega}(x,y)\leq{d_{\Omega}(x,y)}$.}} \\ \\
    Fix $\varepsilon$ so that $0<\varepsilon<\frac{d'}{2}-d_{cc}^{\Omega}(x,y)$. By the definition of $d_{cc}(x,y)$ we can find a crosscut $\gamma$ of $\Omega$ that joins $x$ and $y$ to one of the four distinguished boundary arcs of $(\Omega,A,B,C,D)$ and separates it from the opposite boundary arc, so that $\text{length}(\gamma)<d_{cc}^{\Omega}(x,y)+\varepsilon$. We now split our problem into two further cases, depending on whether the relevant boundary arcs of $(\Omega,A,B,C,D)$ are Dirichlet arcs, where $h$ is constant, or Neumann arcs, along which $h$ is monotone.\\ \\
    \textbf{\underline{Case 2.1}}: $\gamma$ joins $x$ and $y$ to one of the Dirichlet arcs and separates it from the opposite Dirichlet arc. \\ \\
    WLOG, suppose that $\gamma$ joins $x$ and $y$ to $[A,B]_{\partial{\Omega^{\ast}}}$ and separates $x$ and $y$ from $[C,D]_{\partial{\Omega^{\ast}}}$. Let $N_{x,y}^{\gamma}$ denote the connected component of $\Omega^{\ast}\setminus{\gamma}$ containing $x$ and $y$. By the maximum principle for harmonic functions:
    \begin{equation*}
        \sup_{z\in{N_{x,y}^{\gamma}}}h(z)=\max_{z\in{\gamma}}h(z)
    \end{equation*}
    Let $v$ be a point of $\gamma$ so that:
    \begin{equation*}
    h(v)=\max_{z\in{\gamma}}h(z)    
    \end{equation*}
    Since $0<h(x)<h(y)<h(v)$, it follows that:
    \begin{equation*}
        h(y)-h(x)<h(v)
    \end{equation*}
    Consider the region: 
    \begin{equation*}
        \Omega_{v}=\{z\in{\Omega}: 0<h(z)<h(v)\}
    \end{equation*}
    $\phi$ maps $\Omega_{v}$ to the rectangle $(0,h(v))\times{(0,1)}$. Thinking of $\Omega_{v}$ as a conformal rectangle with the distinguished boundary arcs: 
    \begin{align*}
        N&=\Omega_{v}^{\ast}\cap{[D,A]_{\partial{\Omega^{\ast}}}} & E&=\Omega_{v}^{\ast}\cap{[A,B]_{\partial{\Omega^{\ast}}}} \\
        S&=\Omega_{v}^{\ast}\cap{[B,C]_{\partial{\Omega^{\ast}}}} & W&=\{z\in{\Omega}:h(z)=h(v)\}
    \end{align*}
    it follows that: 
    \begin{equation*}
        \lambda(W\leftrightarrow{E};\Omega_{v})=h(v)
    \end{equation*}
   Similar to case 1, having reinterpreted $h(v)$ as an extremal length, we can bound it from above by bounding the dual extremal length from below.
    \begin{figure}[H]
    \centering
    \includegraphics[scale=0.54]{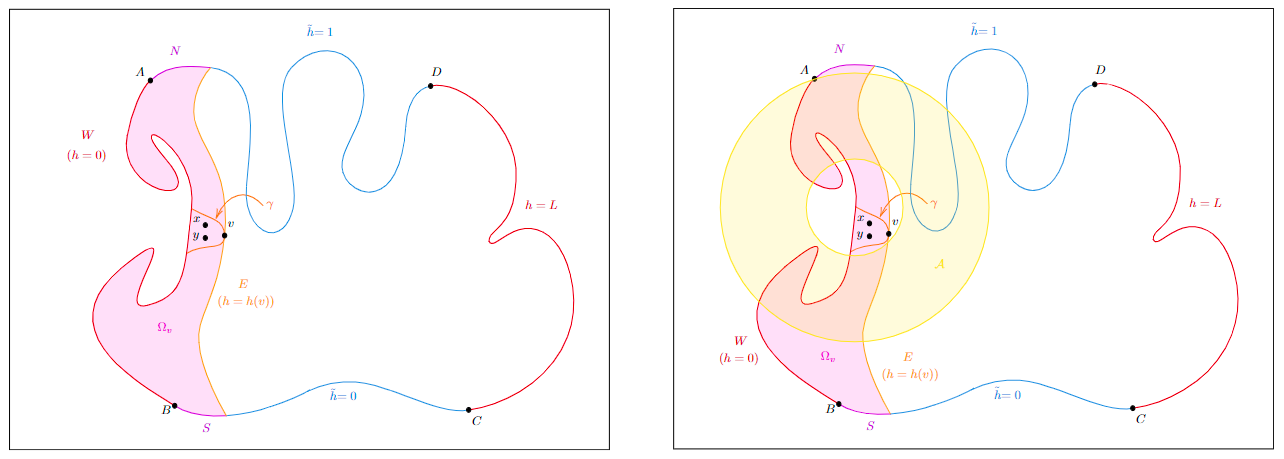}
    \caption{On the left: The subrectangle $\Omega_{v}$ of $\Omega$ with its distinguished boundary arcs $N, E, S$ and $W$. On the right: Observe that any path from $N$ to $S$ in $\Omega_{v}$ must cross the annulus $\mathcal{A}$.}
    \end{figure}
    
    \noindent Since $v\in{W}$ and $v$ lies along a crosscut $\gamma$ of $\Omega$ that starts and ends along $[A,B]_{\partial{\Omega^{\ast}}}$, there must exist a subarc $\gamma'$ of $\gamma$ of length at most $\frac{1}{2}(d_{cc}^{\Omega}(x,y)+\varepsilon)$ travelling from $E$ to $W$ and thereby separating $N$ and $S$ in $\Omega_{v}$. Let $v'$ be the endpoint of $\gamma'$ that lies in $W$. Consider the annulus: 
    \begin{equation*}
        \mathcal{A}=\{u\in{\mathbb{C}}:\frac{d_{cc}^{\Omega}(x,y)+\varepsilon}{2}<|u-v'|<\frac{d'}{2}\}
    \end{equation*}
    Just as in case 1: 
    \begin{itemize}
        \item Since $\text{length}(\gamma')\leq{\frac{d_{cc}^{\Omega}(x,y)+\varepsilon}{2}}$, the diameter of $\gamma'$ is at most $\frac{d_{cc}^{\Omega}(x,y)+\varepsilon}{2}$. Hence, $\gamma\subseteq{B(v',\frac{d_{cc}^{\Omega}(x,y)+\varepsilon}{2})}$. Since $\gamma'$ separates the boundary arcs $N$ and $S$ in $\Omega_{v}$, any path from $N$ to $S$ in $\Omega_{v}$ must intersect $\gamma'$ and therefore $B(v',\frac{d_{cc}^{\Omega}(x,y)+\varepsilon}{2})$
        \item Since any curve from $N$ to $S$ in $\Omega_{v}$ is a curve from $[B,C]_{\partial{\Omega^{\ast}}}$ to $[D,A]_{\partial{\Omega^{\ast}}}$ in $\Omega_{v}$, such a curve must have diameter $\geq{d'}$. In particular, any such curve must at some point lie outside the ball $B(v',\frac{d'}{2})$.
    \end{itemize}
    Putting all this together, we conclude that any curve from $N$ to $S$ in $\Omega_{v}$ must cross the annulus $\mathcal{A}$ at least once. Hence, by the same argument as in case 1, plugging the metric:
    \begin{equation*}
        \rho(z)=\frac{1}{|z-v'|}1_{\mathcal{A}\cap{\Omega_{v}}}(z)
    \end{equation*}
    into the variational problem for $\lambda(N\leftrightarrow{S};\Omega_{v})$, we have that: 
    \begin{equation*}
        \lambda(N\leftrightarrow{S};\Omega_{v})\geq{\frac{1}{2\pi}\log\Big(\frac{d'}{d_{cc}(x,y)+\varepsilon}\Big)}
    \end{equation*}
    By duality for continuous extremal length: 
    \begin{equation*}
        \lambda(W\leftrightarrow{E};\Omega_{v})\cdot\lambda(N\leftrightarrow{S};\Omega_{v})=1
    \end{equation*}
    Hence: 
    \begin{equation*}
        h(y)-h(x)<h(v)\leq{\frac{2\pi}{\log\Big(\frac{d'}{d_{cc}(x,y)+\varepsilon}\Big)}}
    \end{equation*}
    Since $\varepsilon>0$ was arbitrary, letting $\varepsilon$ tend to $0$ in the above inequality, the desired result follows.    \\ \\
    \textbf{\underline{Case 2.2}}: $\gamma$ joins $x$ and $y$ to one of the Neumann arcs and separates it from the opposite Neumann arc.\\ \\
    WLOG, suppose that $\gamma$ joins $x$ and $y$ to $[B,C]_{\partial{\Omega^{\ast}}}$ and separates $x$ and $y$ from $[D,A]_{\partial{\Omega^{\ast}}}$. Let $N_{x,y}^{\gamma}$ denote the connected component of $\Omega^{\ast}\setminus{\gamma}$ containing $x$ and $y$. By the maximum principle for harmonic functions: 
    \begin{align*}
        \sup_{z\in{N_{x,y}^{\gamma}}}h(z)&=\max_{z\in{\gamma}}h(z), &\inf_{z\in{N_{x,y}^{\gamma}}}h(z)&=\min_{z\in{\gamma}}h(z) 
    \end{align*}
    Let $u$ and $v$ be points of $\gamma$ so that: 
    \begin{align*}
        h(u)&=\min_{z\in{\gamma}}h(z) , & h(v)&=\max_{z\in{\gamma}}h(z) 
    \end{align*}
    Since $h(u)<h(x)<h(y)<h(v)$, it follows that: 
    \begin{equation*}
        h(y)-h(x)<h(v)-h(u)
    \end{equation*}
    Consider the region:
    \begin{equation*}
        \Omega_{u,v}=\{z\in{\mathbb{C}}: h(u)<h(z)<h(v)\}
    \end{equation*}
    $\phi$ maps $\Omega_{u,v}$ to the rectangle $(h(u),h(v))\times{(0,1)}$. Thinking of $\Omega_{u,v}$ as a conformal rectangle with the distinguished boundary arcs:
    \begin{align*}
        N&=\Omega_{u,v}^{\ast}\cap{[D,A]_{\partial{\Omega^{\ast}}}} & E&=\{z\in{\Omega}:h(z)=h(v)\} \\ 
        S&=\Omega_{u,v}^{\ast}\cap{[B,C]_{\partial{\Omega^{\ast}}}} & W&=\{z\in{\Omega}:h(z)=h(u)\}
    \end{align*}
    we have that: 
    \begin{equation*}
        \lambda(W\leftrightarrow{E};\Omega_{u,v})=h(v)-h(u)
    \end{equation*}
    \begin{figure}[H]
        \centering
        \includegraphics[scale=0.7]{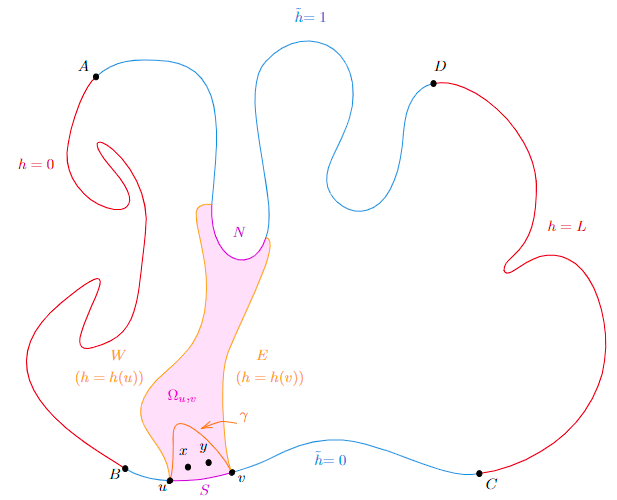}  
        \caption{The subrectangle $\Omega_{u,v}$ of $\Omega$ with its distinguished boundary arcs $N, E, S$ and $W$.}
    \end{figure}
    \noindent Thus, to get an upper bound for $h(v)-h(u)$ and therefore $h(y)-h(x)$, it suffices to bound the dual extremal length $\lambda(N\leftrightarrow{S};\Omega_{u,v})$ from below. Observe that:
    \begin{itemize}
        \item Since $u$ and $v$ both lie along a crosscut of $\Omega$ of length at most $d_{cc}^{\Omega}(x,y)+\varepsilon$ and $u\in{W}$, $w\in{E}$, we can join $W$ and $E$ by a crosscut of $\Omega_{u,v}$ of length at most $d_{cc}^{\Omega}(x,y)+\varepsilon$.
        \item Since any curve from $N$ to $S$ in $\Omega_{u,v}$ is a curve from $[A,D]_{\partial{\Omega^{\ast}}}$ to $[B,C]_{\partial{\Omega^{\ast}}}$ in $\Omega$, such a curve will have diameter $\geq{d'}$. 
    \end{itemize}
    Putting all this together, by the same argument as in case 1, verbatim, it follows that: 
    \begin{equation*}
        \lambda(N\leftrightarrow{S};\Omega_{u,v})\geq{\frac{1}{2\pi}\log\Big(\frac{d'}{2(d_{cc}^{\Omega}(x,y)+\varepsilon)}\Big)}
    \end{equation*}
    Hence: 
    \begin{equation*}
        h(y)-h(x)<h(v)-h(u)\leq{\lambda(W\leftrightarrow{E};\Omega_{u,v})}\leq{\frac{2\pi}{\log\Big(\frac{d'}{2(d_{cc}^{\Omega}(x,y)+\varepsilon)}\Big)}}
    \end{equation*}
    Letting $\varepsilon$ tend to $0$ in the above inequality, the desired result follows. The analogous estimate for $\widetilde{h}$ follows by the same argument.
\end{proof}

\subsection{Modulus of Continuity for Tiling Maps} 
\label{subsec: modulus of continuity for tiling maps}

Suppose $(G,A^{\bullet}, B^{\bullet}, C^{\bullet}, D^{\bullet})$ is an orthodiagonal rectangle and let $\widetilde{h}$ be the unique solution to the following boundary value problem on $G^{\circ}$: 
\begin{align*}
    \widetilde{h}(x)&=0 \hspace{5pt} \text{for all $x\in{[D^{\circ}, A^{\circ}]}$} \\ 
    \widetilde{h}(x)&=1 \hspace{5pt} \text{for all $x\in{[B^{\circ}, C^{\circ}]}$} \\
    \Delta^{\circ}\widetilde{h}(x)&=0 \hspace{5pt} \text{for all $x\in{V^{\circ}\setminus{([D^{\circ},A^{\circ}]\cup{[B^{\circ}, C^{\circ}]})}}$}
\end{align*}
Let $h$ be the solution to the following boundary value problem on $G^{\bullet}$:
\begin{align*}
    h(x)&=0 \hspace{5pt} \text{for all $x\in{[A^{\bullet}, B^{\bullet}]}$} \\ 
    h(x)&=L \hspace{5pt} \text{for all $x\in{[C^{\bullet}, D^{\bullet}]}$} \\
    \Delta^{\bullet}h(x)&=0 \hspace{5pt} \text{for all $x\in{V^{\bullet}\setminus{([A^{\bullet},B^{\bullet}]\cup{[C^{\bullet}, D^{\bullet}]})}}$}
\end{align*}
where $L$ is the effective resistance between $[A^{\bullet},B^{\bullet}]$ and $[C^{\bullet},D^{\bullet}]$ in $G^{\bullet}$. $h$ and $\widetilde{h}$ are the conjugate discrete harmonic functions that correspond to the real and imaginary parts of the tiling map $\phi:\widehat{G}\rightarrow{(0,L)\times{(0,1)}}$. 
In this section we will prove regularity estimates for $h$ and $\widetilde{h}$ analogous to the regularity estimates we proved for the real and imaginary parts of the corresponding conformal map in Section \ref{subsec: Modulus of Continuity for Limiting Conformal Map}. We do this by adapting our argument in Section \ref{subsec: Modulus of Continuity for Limiting Conformal Map} to the discrete, orthodiagonal setting. To do this, we first need to establish the following lemma which gives us an estimate for the gradients of $h$ and $\widetilde{h}$ across an edge:    
\begin{lem}
\label{lem: gradient estimate for tiling map} 
Suppose $(G,A^{\bullet}, B^{\bullet}, C^{\bullet}, D^{\bullet})$ is an orthodiagonal rectangle so that the edges of $G$ all have length at most $\varepsilon$. Let $\widetilde{h}$ be the unique solution to the following boundary value problem on $G^{\circ}$: 
\begin{align*}
\widetilde{h}(x)&=0 \hspace{5pt} \text{for all $x\in{[D^{\circ}, A^{\circ}]}$} \\ 
\widetilde{h}(x)&=1 \hspace{5pt} \text{for all $x\in{[B^{\circ}, C^{\circ}]}$} \\
\Delta^{\circ}\widetilde{h}(x)&=0 \hspace{5pt} \text{for all $x\in{V^{\circ}\setminus{([D^{\circ},A^{\circ}]\cup{[B^{\circ}, C^{\circ}]})}}$}
\end{align*}
Let $h$ be the harmonic conjugate of $\widetilde{h}$ which solves the following boundary value problem on $G^{\bullet}$:
\begin{align*}
h(x)&=0 \hspace{5pt} \text{for all $x\in{[A^{\bullet}, B^{\bullet}]}$} \\ 
h(x)&=L \hspace{5pt} \text{for all $x\in{[C^{\bullet}, D^{\bullet}]}$} \\
\Delta^{\bullet}h(x)&=0 \hspace{5pt} \text{for all $x\in{V^{\bullet}\setminus{([A^{\bullet},B^{\bullet}]\cup{[C^{\bullet}, D^{\bullet}]})}}$}
\end{align*}
where $L$ is the effective resistance between $[A^{\bullet},B^{\bullet}]$ and $[C^{\bullet}, D^{\bullet}]$ in $G^{\bullet}$. Define:
\begin{align*}
    d&= \inf\{\text{diameter}(\gamma): \text{$\gamma$ is a curve in $\widehat{G}$ from $[A^{\bullet},B^{\bullet}]_{\partial{\widehat{G}^{\ast}}}$ to $[C^{\bullet},D^{\bullet}]_{\partial{\widehat{G}^{\ast}}}$}\}\\
    d'&=\inf\{\text{diameter}(\gamma): \text{$\gamma$ is a curve in $\widehat{G}$ from $[B^{\bullet},C^{\bullet}]_{\partial{\widehat{G}^{\ast}}}$ to $[D^{\bullet},A^{\bullet}]_{\partial{\widehat{G}^{\ast}}}$}\}    
\end{align*}
If $\varepsilon\leq{\frac{d\wedge{d'}}{16}}$, there exists an absolute constant $K>0$ so that if $x,y\in{V^{\bullet}}$ are neighboring edges in $G^{\bullet}$ and $u,v\in{V^{\circ}}$ are neighboring edges in $G^{\circ}$, then:
\begin{align*}
    |h(y)-h(x)|&\leq{\frac{K}{\log\big(\frac{d'}{\varepsilon}\big)}} & |\widetilde{h}(u)-\widetilde{h}(v)|&\leq{\frac{K\hspace{1pt}L}{\log\big(\frac{d}{\varepsilon}\big)}}
\end{align*}
\end{lem}
\begin{proof}
    Suppose $(G,A^{\bullet}, B^{\bullet}, C^{\bullet}, D^{\bullet})$ is a orthodiagonal rectangle with edges of length at most $\varepsilon$ and $h:V^{\bullet}\rightarrow{\mathbb{R}}$ be the solution to the Dirichlet- Neumann problem on this orthodiagonal map that is $0$ on $[A^{\bullet}, B^{\bullet}]$, $L$ on $[C^{\bullet}, D^{\bullet}]$ and harmonic elsewhere, where $L$ is the discrete extremal length from $[A^{\bullet}, B^{\bullet}]$ to $[C^{\bullet}, D^{\bullet}]$ in $G^{\bullet}$. Define: 
    \begin{equation*}
        \chi:=\max_{\{x,y\}\in{E^{\bullet}}}|h(y)-h(x)|
    \end{equation*}
    If $\chi=0$, we're done. Otherwise, select neighboring vertices $x,y\in{V^{\bullet}}$ so that $\big(h(y)-h(x)\big)=\chi$. Consider the sets $S_{x}$ and $S_{y}$ defined as follows: 
\begin{align*}
    S_{x}&:=\{z\in{V^{\bullet}}:h(z)\leq{h(x)}\} & S_{y}&:=\{z\in{V^{\bullet}}:h(z)\geq{h(y)}\}
\end{align*}
By the maximum principle for harmonic functions, $S_{x}$, $S_{y}$ and $V^{\bullet}\setminus{(S_{x}\sqcup{S_{y}})}$ are all connected subsets of $G^{\bullet}$. Furthermore, $[A^{\bullet}, B^{\bullet}]\subseteq{S_{x}}$, $[C^{\bullet}, D^{\bullet}]\subseteq{S_{y}}$. Let $H=(V_{H}^{\bullet}\sqcup{V_{H}^{\circ}},E_{H})$ be the suborthodiagonal map of $G$ formed by gluing together all the faces of $G$ that are incident to at least one vertex of $V^{\bullet}\setminus{(S_{x}\sqcup{S_{y}})}$. By the maximum principle, $H$ is simply connected with a unique, distinguished exterior face. Moreover, let: 
\begin{align*}
    O^{\bullet}&:=S_{x}\cap{\partial{V_{H}^{\bullet}}}, & W^{\bullet}&:=S_{y}\cap{\partial{V_{H}^{\bullet}}}
\end{align*}
Then $O^{\bullet}$ and $W^{\bullet}$ are primal boundary arcs of $H$ with corresponding dual arcs: 
\begin{align*}
    N^{\circ}&:=[B^{\circ}, C^{\circ}]\cap{\partial{V_{H}^{\circ}}}, & S^{\circ}&:=[D^{\circ}, A^{\circ}]\cap{\partial{V_{H}^{\circ}}}
\end{align*}
\begin{figure}[H]
    \centering
    \includegraphics[scale=0.60]{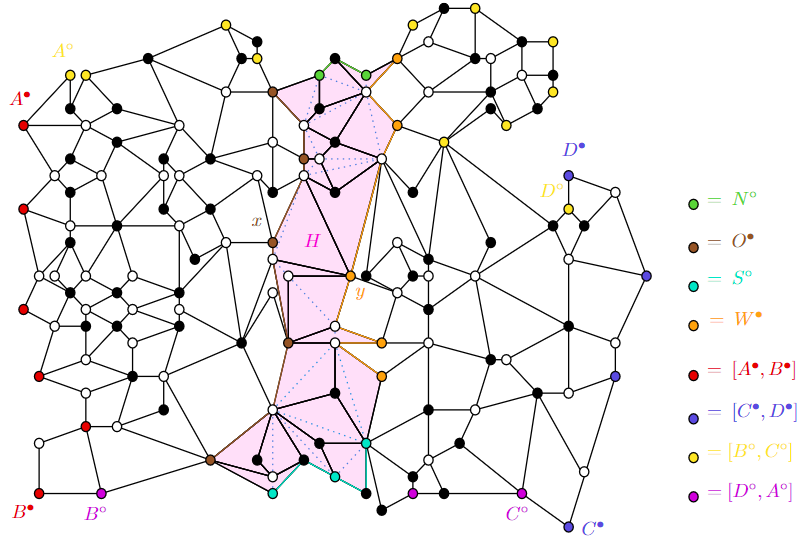}  
    \caption{The suborthodiagonal map $H$ of $G$ with its four distinguished boundary arcs. Notice that any path from $N^{\circ}$ to $S^{\circ}$ in $H^{\circ}$ has to use the dual edge $\{u,v\}\in{E^{\circ}}$ corresponding to $\{x,y\}\in{E^{\bullet}}$. (The blue dotted lines are the edges of $H^{\circ}$.)}
\end{figure}
\noindent Similar to the proof of Theorem \ref{thm: modulus of continuity for conformal map}, $N$, $O$, $S$ and $W$ stand for "North," "Orient," "South," and "West." We'd have used $E$ for "East," however in the discrete setting, $E$ is already being used to denote the edges of $G$. Proposition \ref{Skopenkov's trick} tells us that for any function $g:V^{\bullet}_{H}\rightarrow{\mathbb{R}}$ with $\text{gap}_{O^{\bullet},W^{\bullet}}(g)\geq{0}$ and any flow $\theta$ from $O^{\bullet}$ to $W^{\bullet}$ in $H^{\bullet}$:
\begin{equation*}
    \text{strength}(\theta)\cdot\text{gap}_{O^{\bullet},W^{\bullet}}(g)\leq{\mathcal{E}^{\bullet}(\theta;H)^{1/2}\mathcal{E}^{\bullet}(g;H)^{1/2}}
\end{equation*}
Plugging $g=h$ into the inequality above, we have that for any choice of flow $\theta$ from $O^{\bullet}$ to $W^{\bullet}$ in $H^{\bullet}$:
\begin{equation*}
    \chi=|h(y)-h(x)|\leq{\frac{\mathcal{E}^{\bullet}(\theta;H)^{1/2}\mathcal{E}^{\bullet}(h;H)^{1/2}}{\text{strength}(\theta)}}
\end{equation*}
By Thomson's principle, taking the infimum over all flows $\theta$ from $O^{\bullet}$ to $W^{\bullet}$ in $H^{\bullet}$ in the expression on the RHS, we have that: 
\begin{equation*}
    \chi\leq{\mathcal{E}^{\bullet}(h;H)^{1/2}\cdot{\lambda(O^{\bullet}\leftrightarrow{W^{\bullet}};H^{\bullet})^{1/2}}}
\end{equation*}
In Section \ref{subsec: Tilings of Rectangles} we saw that $\mathcal{E}^{\bullet}(h)$ is the total area of rectangles in the tiling associated with the orthodiagonal rectangle $(G,A^{\bullet},B^{\bullet},C^{\bullet},D^{\bullet})$. Hence, by the definition of $H$, the restriction $\mathcal{E}^{\bullet}(h;H)$ is the total area of rectangles in our tiling that intersect $(h(x),h(y))\times{(0,1)}$. Since $|h(y)-h(x)|=\chi$ and any rectangle in our tiling has width at most $\chi$, it follows that: 
\begin{equation*}
    \mathcal{E}^{\bullet}(h;H)\leq{3\chi}
\end{equation*}
Hence: 
\begin{equation}
\label{eqn: chi bounded by extremal length}    \chi\leq{3\hspace{1pt}\lambda(O^{\bullet}\leftrightarrow{W^{\bullet}};H^{\bullet})}
\end{equation}
By duality: 
\begin{equation*}
    \lambda(N^{\circ}\leftrightarrow{S^{\circ}};H^{\circ})\cdot\lambda(O^{\bullet}\leftrightarrow{W^{\bullet}};H^{\bullet})=1
\end{equation*}
Thus, to bound $\lambda(O^{\bullet}\leftrightarrow{W^{\bullet}};H^{\bullet})$ 
 and therefore $\chi$ from above, it suffices to bound the dual extremal length $\lambda(N^{\circ}\leftrightarrow{S^{\circ}};H^{\circ})$ from below. We will do this by picking by picking a good metric to plug into the variational problem for $\lambda(N^{\circ}\leftrightarrow{S^{\circ}};H^{\circ})$. \\ \\ 
Consider the metric $\rho:E_{H}^{\circ}\rightarrow{[0,\infty)}$ defined as follows: 
\begin{equation*}
    \rho(e^{\circ})=\int_{e^{\circ}}\frac{|dz|}{|z-\frac{x+y}{2}|}
\end{equation*}
for $e^{\circ}\in{E_{H}^{\circ}}$ contained in the annulus $\widetilde{\mathcal{A}}=\{z\in{\mathbb{C}}:4\varepsilon\leq{\left|z-\frac{x+y}{2}\right|}\leq{\frac{d'}{2}}\}$, where we think of $e^{\circ}$ as a line segment in the plane. If $e^{\circ}$ is not contained in $\widetilde{\mathcal{A}}$, then $\rho(e^{\circ})=0$. Let $\{u,v\}\in{E^{\circ}_{H}}$ be the dual edge corresponding to the primal edge $\{x,y\}$. Observe that: 
\begin{enumerate}
    \item \label{edge (u,v) always used} By Lemma \ref{lem: min cuts are dual paths rectangles}, minimal $O^{\bullet}$- $W^{\bullet}$ cuts in $H^{\bullet}$ correspond to simple paths from $N^{\circ}$ to $S^{\circ}$ in $H^{\circ}$ and vice versa. Since $\{x,y\}$ joins $O^{\bullet}$ and $W^{\bullet}$, this edge is part of any $O^{\bullet}$- $W^{\bullet}$ cut. Hence, any path from $N^{\circ}$ to $S^{\circ}$ in $H^{\circ}$ must use the edge $\{u,v\}$.
    \item \label{any path exits annulus} Any path from $N^{\circ}$ to $S^{\circ}$ in $H^{\circ}$ is a path from $[D^{\circ}, A^{\circ}]$ to $[B^{\circ}, C^{\circ}]$ in $G^{\circ}$. Using the notation in the statement of Lemma \ref{lem: gradient estimate for tiling map}, it follows that the diameter of any such path is at least $d'$.
\end{enumerate}
By \ref{edge (u,v) always used}, since $\{u,v\}\subseteq{B(\frac{x+y}{2},4\varepsilon)}$, any path in $H^{\circ}$ from $N^{\circ}$ to $S^{\circ}$ must at some point lie in the bounded component of $\mathbb{C}\setminus{\widetilde{\mathcal{A}}}$. By \ref{any path exits annulus}, any path in $H^{\circ}$ from $N^{\circ}$ to $S^{\circ}$ must at some point lie in the unbounded component of $\mathbb{C}\setminus{\widetilde{\mathcal{A}}}$.  
Putting all this together, we conclude that any path from $N^{\circ}$ to $S^{\circ}$ in $H^{\circ}$ crosses the annulus $\widetilde{\mathcal{A}}$ at least once. If $\gamma$ is a piecewise $C^{1}$ curve that crosses the annulus $\{z\in{\mathbb{C}}:r<|z-z_{0}|<R\}$ at least once, where $z_{0}\in{\mathbb{C}}$, $R>r>0$, then: 
\begin{equation*}
\int_{\gamma}\frac{|dz|}{|z-z_{0}|}\geq{\log\Big(\frac{R}{r}\Big)}
\end{equation*}
If $\gamma$ is a path from $N^{\circ}$ to $S^{\circ}$ in $H^{\circ}$, we know that $\gamma$ must cross the annulus $\widetilde{A}$ at least once. However, when doing this, it is possible that $\gamma$ uses edges of $E^{\circ}_{H}$ that are not entirely contained in $\widetilde{\mathcal{A}}$ and so have zero mass with respect to $\rho$. Since edges of $G$ all have length at most $\varepsilon$, edges of $G^{\circ}$ all have length at most $2\varepsilon$. Thus, we can be sure that all of the edges $e^{\circ}\in{E^{\circ}_{H}}$ that we use when crossing the slightly smaller annulus $\{z\in{\mathbb{C}}:6\varepsilon\leq{\left|z-\frac{x+y}{2}\right|}\leq{\frac{d'}{2}-2\varepsilon}\}$ do indeed have positive mass. Hence: 
\begin{equation*}
\ell_{\rho}(\gamma)=\int_{\gamma}\frac{|dz|}{|z-\frac{x+y}{2}|}\geq{\log{\Big(\frac{d'-4\varepsilon}{12\varepsilon}\Big)}} 
\end{equation*}
Let $F_{in}(H)$ denote the set of inner faces of $H$. Given an inner face $Q\in{F_{in}(H)}$, let $e_{Q}^{\bullet}\in{E^{\bullet}_{H}}, e_{Q}^{\circ}\in{E^{\circ}_{H}}$ denote the primal and dual diagonals of $Q$. Next, we estimate the area of the metric $\rho$: 
\begin{align*}
    A(\rho)&=\sum_{\substack{e^{\circ}\in{E_{H}^{\circ}} \\ \rho(e^{\circ})>0}}\frac{|e^{\bullet}|}{|e^{\circ}|}\Big(\int_{e^{\circ}}\frac{|dz|}{|z-\frac{x+y}{2}|}\Big)^{2}\leq{\sum_{\substack{e^{\circ}\in{E_{H}^{\circ}} \\ \rho(e^{\circ})>0}}\frac{|e^{\bullet}|}{|e^{\circ}|}\cdot{\frac{|e^{\circ}|^{2}}{\min\limits_{z\in{e^{\circ}}}|z-\frac{x+y}{2}|^{2}}}}\stackrel{(i)}{\leq}{2\sum_{\substack{Q\in{F_{in}(H)} \\ \rho(e_{Q}^{\circ})>0}}\frac{\text{Area}(Q)}{\min\limits_{z\in{Q}}|z-\frac{x+y}{2}|^{2}}} \\ 
    &\stackrel{(ii)}{\leq}{2\sum_{\substack{Q\in{F_{in}(H)} \\ \rho(e_{Q}^{\circ})>0}}\int_{Q}\frac{1}{(|z-\frac{x+y}{2}|-2\varepsilon)^{2}}dz_{1}dz_{2}}\stackrel{(iii)}{\leq}{2\int_{3\varepsilon<|z-\frac{x+y}{2}|<\frac{d'}{2}+\varepsilon}\frac{1}{(|z-\frac{x+y}{2}|-2\varepsilon)^{2}}dz_{1}dz_{2} }  \\ 
    &=4\pi\int_{3\varepsilon}^{\frac{d'}{2}+\varepsilon}\frac{s}{(s-2\varepsilon)^{2}}ds=4\pi\int_{\varepsilon}^{\frac{d'}{2}-\varepsilon}\frac{1}{u}du + 8\pi\varepsilon\int_{\varepsilon}^{\frac{d'}{2}-\varepsilon}\frac{1}{u^{2}}du \\
    &=4\pi\log\Big(\frac{d'-2\varepsilon}{\varepsilon}\Big)+8\pi\varepsilon\left(\frac{1}{\varepsilon}-\frac{2}{(d'-2\varepsilon)}\right) \\ 
    &\leq{4\pi\log\Big(\frac{d'-2\varepsilon}{\varepsilon}\Big)+8\pi}
\end{align*}
In the calculations above: 
\begin{itemize}
    \item (i) follows from the trivial observation that $\min\limits_{z\in{Q}}|z-\frac{x+y}{2}|\leq{\min\limits_{z\in{e^{\circ}_{Q}}}|z-\frac{x+y}{2}|}$, as well as the fact that for any inner face $Q$ of $H$, since $Q$ is a quadrilateral with orthogonal diagonals, $\text{Area}(Q)=\frac{1}{2}|e_{Q}^{\bullet}||e_{Q}^{\circ}|$.
    \item (ii) follows from the fact that for any inner face $Q$ of $H$, since the edges of $Q$ all have length at most $\varepsilon$, $Q$ has diameter at most $2\varepsilon$. Hence, $|z-\frac{x+y}{2}|-2\varepsilon\leq{\min\limits_{z\in{Q}}|z-\frac{x+y}{2}|}$ for any $z\in{Q}$. 
    \item (iii) follows from the fact that if $\rho(e^{\circ}_{Q})>0$ for some inner face $Q$ of $H$, this means that the dual edge $e^{\circ}_{Q}$ is contained in the annulus $\widetilde{\mathcal{A}}$. Since the corresponding primal edge $e^{\bullet}_{Q}$ has length at most $\varepsilon$, any point of $Q$ lies in an $\varepsilon$- neighborhood of $e^{\circ}_{Q}$. In particular, any inner face $Q$ of $H$ such that $\rho(e^{\circ}_{Q})>0$ is contained in the slightly larger annulus $\{z\in{\mathbb{C}}:3\varepsilon<|z-\frac{x+y}{2}|<\frac{d'}{2}+\varepsilon\}$.   
\end{itemize}
Since $d'\geq{16\varepsilon}$, it follows that:
\begin{equation*}
    A(\rho)\leq{4\pi\log\Big(\frac{d'-2\varepsilon}{\varepsilon}\Big)+8\pi}\asymp{\log\Big(\frac{d'}{\varepsilon}\Big)}    
\end{equation*}
If $\gamma$ is a path from $N^{\circ}$ to $S^{\circ}$ in $H^{\circ}$, using the fact that $d'\geq{16\varepsilon}$:
$$
\ell_{\rho}(\gamma)\geq{\log{\Big(\frac{d'-4\varepsilon}{16\varepsilon}\Big)}}\asymp{\log\Big(\frac{d'}{\varepsilon}\Big)}
$$
Putting all this together, we have that: 
\begin{equation*}
    \lambda(N^{\circ}\leftrightarrow{S^{\circ}};H^{\circ})\geq{\inf\limits_{\gamma}\frac{\ell^{2}_{\rho}(\gamma)}{A(\rho)}}\gtrsim{\log\Big(\frac{d'}{\varepsilon}\Big)}
\end{equation*}
where our infimum is over all paths $\gamma$ in $H^{\circ}$ from $N^{\circ}$ to $S^{\circ}$. By duality: 
\begin{equation*}
    \lambda(O^{\bullet}\leftrightarrow{W^{\bullet}};H^{\bullet})=\frac{1}{\lambda(N^{\circ}\leftrightarrow{S^{\circ}};H^{\circ})}\lesssim{\frac{1}{\log\big(\frac{d'}{\varepsilon}\big)}}
\end{equation*}
Plugging this into Equation \ref{eqn: chi bounded by extremal length}, the desired result follows. The corresponding estimate for $\widetilde{h}$ follows by the same argument.
\end{proof}
\noindent Having established our estimate for the gradient of $h$ and $\widetilde{h}$ across an edge, we are now ready to state and prove the discrete analogue of Theorem \ref{thm: modulus of continuity for conformal map}:
\begin{thm}
\label{thm: modulus of continuity for tiling maps}
Suppose $(G,A^{\bullet}, B^{\bullet}, C^{\bullet}, D^{\bullet})$ is an orthodiagonal rectangle with edges of length at most $\varepsilon$ and let $h$ and $\widetilde{h}$ be the real and imaginary parts of the associated tiling map. That is, $\widetilde{h}$ is the unique solution to the following boundary value problem on $G^{\circ}$: 
\begin{align*}
    \widetilde{h}(x)&=0 \hspace{5pt} \text{for all $x\in{[D^{\circ}, A^{\circ}]}$} \\ 
    \widetilde{h}(x)&=1 \hspace{5pt} \text{for all $x\in{[B^{\circ}, C^{\circ}]}$} \\
    \Delta^{\circ}\widetilde{h}(x)&=0 \hspace{5pt} \text{for all $x\in{V^{\circ}\setminus{([D^{\circ},A^{\circ}]\cup{[B^{\circ}, C^{\circ}]})}}$}
\end{align*}
and $h$ be the harmonic conjugate of $\widetilde{h}$ which solves the following boundary value problem on $G^{\bullet}$:
\begin{align*}
    h(x)&=0 \hspace{5pt} \text{for all $x\in{[A^{\bullet}, B^{\bullet}]}$} \\ 
    h(x)&=L \hspace{5pt} \text{for all $x\in{[C^{\bullet}, D^{\bullet}]}$} \\
    \Delta^{\bullet}h(x)&=0 \hspace{5pt} \text{for all $x\in{V^{\bullet}\setminus{([A^{\bullet},B^{\bullet}]\cup{[C^{\bullet}, D^{\bullet}]})}}$}
\end{align*}
where $L$ is the effective resistance between $[A^{\bullet},B^{\bullet}]$ and $[C^{\bullet}, D^{\bullet}]$ in $G^{\bullet}$. Define:
\begin{align*}
        d&= \inf\{\text{diameter}(\gamma): \text{$\gamma$ is a curve in $\widehat{G}$ joining $[A^{\bullet},B^{\bullet}]_{\partial{\widehat{G}}}$ and $[C^{\bullet},D^{\bullet}]_{\partial{\widehat{G}}}$}\}\\
        d'&=\inf\{\text{diameter}(\gamma): \text{$\gamma$ is a curve in $\widehat{G}$ joining $[B^{\bullet},C^{\bullet}]_{\partial{\widehat{G}}}$ and $[D^{\bullet},A^{\bullet}]_{\partial{\widehat{G}}}$}\}   
    \end{align*}
Then there exists an absolute constant $K>0$ so that:
\begin{align*}
    |h(y)-h(x)|&\leq{\frac{K}{\log{\big(\frac{d'}{(d_{\widehat{G}}(x,y)\wedge{d_{cc}^{\widehat{G}}(x,y)})\vee{\varepsilon}}\big)}}}, & |\widetilde{h}(u)-\widetilde{h}(v)|&\leq{\frac{KL}{\log{\big(\frac{d}{(d_{\widehat{G}}(u,v)\wedge{d_{cc}^{\widehat{G}}(u,v)})\vee{\varepsilon}}\big)}}}    
\end{align*}
for any $x,y\in{V^{\bullet}}$, $u,v\in{V^{\circ}}$ such that . In particular, in the bulk (when $|x-y|\leq{\text{dist}(x,\partial{\widehat{G}}), \text{dist}(y,\partial{\widehat{G}})}$ and $|u-v|\leq{\text{dist}(u,\partial{\widehat{G}}), \text{dist}(v,\partial{\widehat{G}})}$), we have that: 
\begin{align*}
    |h(y)-h(x)|&\leq{\frac{K}{\log{\big(\frac{d'}{|y-x|\vee{\varepsilon}}\big)}}}, & |\widetilde{h}(u)-\widetilde{h}(v)|&\leq{\frac{KL}{\log{\big(\frac{d}{|u-v|\vee{\varepsilon}}\big)}}}    
\end{align*}
\end{thm}
\noindent A slightly weaker estimate (with $\sqrt{log(\cdot)}$ in place of $\log(\cdot)$) is proven in \cite{ALP23} in the case where our orthodiagonal rectangle $(G,A^{\bullet},B^{\bullet}, C^{\bullet},D^{\bullet})$ is a good approximation of a conformal rectangle $(\Omega,A,B,C,D)$, where $\Omega$ is a Jordan domain and $[A,B]_{\partial{\Omega}}, [B,C]_{\partial{\Omega}}, [C,D]_{\partial{\Omega}}, [D,A]_{\partial{\Omega}}$ are analytic arcs that don't meet at cusps. In the context of \cite{ALP23}, ``good" means that the edges of $G$ have length at most $\varepsilon$ and discrete boundary arcs are $\delta$- close to the corresponding continuous boundary arcs in Hausdorff distance for some small $\varepsilon, \delta>0$. For more details, see Theorem 3 of \cite{ALP23}. \\ \\  
We should also point out that while the estimates for the modulus of continuity of $h$ and $\widetilde{h}$ in Theorem \ref{thm: modulus of continuity for tiling maps} are novel for points near the boundary of $G$, even when our orthodiagonal map is just a chunk of the square grid, in the bulk we can do significantly better. Namely, in \cite{CLR23}, Chelkak, Laslier and Russkikh show that bounded discrete harmonic and discrete holomorphic functions on $t$- embeddings whose corresponding origami map is $\kappa$- Lipschitz on large scales for some $\kappa\in{(0,1)}$, are $\beta$- H\"older in the bulk for some absolute constant $\beta\in{(0,1)}$ (see Proposition 6.13 of \cite{CLR23}). The condition on the origami map here is known as ``$Lip(\kappa,\delta)$" where $\delta>0$ is the scale on which our origami map is $\kappa$- Lipschitz. $t$- embeddings are a more general class of graphs embedded in the plane that accommodate a notion of discrete complex analysis. That is, every orthodiagonal map is a $t$- embedding (see Section 8.1 of \cite{CLR23}). While it is not explicitly stated in their paper, it is not hard to show that if $G$ is an orthodiagonal map with edges of length at most $\varepsilon$, then for any $\kappa\in{(0,1)}$ there exists an absolute constant $c=c(\kappa)>0$ such that, as a $t$ -embedding, $G$ satisfies the condition $Lip(\kappa,c\varepsilon)$. While the estimates in the bulk provided by Theorem \ref{thm: modulus of continuity for tiling maps} are sufficient for proving our main result, Theorem \ref{thm: convergence of rectangle tiling maps}, as an immediate consequence of Proposition 6.13 of \cite{CLR23} we have that:    
\begin{prop}
    Suppose $(G, A^{\bullet}, B^{\bullet}, C^{\bullet}, D^{\bullet})$ is an orthodiagonal rectangle with edges of length at most $\varepsilon$ and $h$ and $\widetilde{h}$ are the real and imaginary parts of the corresponding tiling map. If $x\in{V}$, let $d_{x}$ denote the distance from $x$ to $\partial{\widehat{G}}$. Then there exists absolute constants $C,C'>0$, $\beta\in{(0,1)}$ such that: 
    \begin{align*}
        |h(y)-h(x)|&\leq{CL\left(\frac{|x-y|}{d_{x}\wedge{d_{y}}}\right)^{\beta}} & |\widetilde{h}(u)-\widetilde{h}(v)|&\leq{C\left(\frac{|u-v|}{d_{u}\wedge{d_{v}}}\right)^{\beta}}
    \end{align*}
    for $x,y\in{V^{\bullet}}$, $u,v\in{V^{\circ}}$ such that $|x-y|,|u-v|\geq{C'\varepsilon}$.
\end{prop}
\noindent All that having been said, we now turn to the proof of Theorem \ref{thm: modulus of continuity for tiling maps}:
\begin{proof}
    Suppose $(G,A^{\bullet}, B^{\bullet}, C^{\bullet}, D^{\bullet})$ is a orthodiagonal rectangle with edges of length at most $\varepsilon$ and $h$ is the real part of the corresponding tiling map. Fix $x,y\in{V^{\bullet}}$. Since we are taking the maximum of $d_{\widehat{G}}(x,y)\wedge{d_{cc}^{\widehat{G}}(x,y)}$ and $\varepsilon$ in our estimate for $|h(y)-h(x)|$ in Theorem \ref{thm: modulus of continuity for tiling maps}, we can assume WLOG that $d_{\widehat{G}}(x,y)\wedge{d_{cc}^{\widehat{G}}(x,y)}\geq{\varepsilon}$, since the relevant estimate in the case where $d_{\widehat{G}}(x,y)\wedge{d_{cc}^{\widehat{G}}(x,y)}<\varepsilon$ follows from the case where $d_{\widehat{G}}(x,y)\wedge{d_{cc}^{\widehat{G}}(x,y)}\geq{\varepsilon}$, by the maximum principle. If $h(y)=h(x)$, we're done. Otherwise, suppose WLOG that $h(x)<h(y)$. We now consider two cases: \\ \\
    \underline{\textbf{Case 1:} $d_{\widehat{G}}(x,y)\leq{d_{cc}^{\widehat{G}}(x,y)}$} \\ \\
    Similar to the proofs of Theorem \ref{thm: modulus of continuity for conformal map} and Lemma \ref{lem: gradient estimate for tiling map}, we begin by reinterpreting the quantity we want to estimate, $\big(h(y)-h(x)\big)$, as an extremal length. Consider the sets $S_{x}$ and $S_{y}$ defined as follows: 
\begin{align*}
    S_{x}&:=\{z\in{V^{\bullet}}:h(z)\leq{h(x)}\} & S_{y}&:=\{z\in{V^{\bullet}}:h(z)\geq{h(y)}\}
\end{align*}
By the maximum principle for harmonic functions, $S_{x}$, $S_{y}$ and $V^{\bullet}\setminus{(S_{x}\sqcup{S_{y}})}$ are all connected subsets of $G^{\bullet}$. Furthermore, $[A^{\bullet}, B^{\bullet}]\subseteq{S_{x}}$, $[C^{\bullet}, D^{\bullet}]\subseteq{S_{y}}$. Let $H=(V_{H}^{\bullet}\sqcup{V_{H}^{\circ}},E_{H})$ be the suborthodiagonal map of $G$ formed by gluing together all the faces of $G$ that are incident to at least one vertex of $V^{\bullet}\setminus{(S_{x}\sqcup{S_{y}})}$. By the maximum principle, $H$ is simply connected with a unique, distinguished exterior face. Furthermore, let: 
\begin{align*}
    O^{\bullet}&:=S_{x}\cap{\partial{V_{H}^{\bullet}}}, & W^{\bullet}&:=S_{y}\cap{\partial{V_{H}^{\bullet}}}
\end{align*}
Then $O^{\bullet}$ and $W^{\bullet}$ are primal boundary arcs of $H$ with corresponding dual arcs: 
\begin{align*}
    N^{\circ}&:=[D^{\circ}, A^{\circ}]\cap{\partial{V_{H}^{\circ}}}, & S^{\circ}&:=[B^{\circ}, C^{\circ}]\cap{\partial{V_{H}^{\circ}}}
\end{align*}
Proposition \ref{Skopenkov's trick} tells us that for any function $g:V^{\bullet}_{H}\rightarrow{\mathbb{R}}$ with $\text{gap}_{O^{\bullet},W^{\bullet}}(g)\geq{0}$ and any flow $\theta$ from $O^{\bullet}$ to $W^{\bullet}$ in $H^{\bullet}$:
\begin{equation*}
    \text{strength}(\theta)\cdot\text{gap}_{O^{\bullet},W^{\bullet}}(g)\leq{\mathcal{E}^{\bullet}(\theta;H)^{1/2}\mathcal{E}^{\bullet}(g;H)^{1/2}}
\end{equation*}
Plugging $g=h$ into the inequality above, we have that for any choice of flow $\theta$ from $O^{\bullet}$ to $W^{\bullet}$ in $H^{\bullet}$:
\begin{equation*}
    |h(y)-h(x)|\leq{\frac{\mathcal{E}^{\bullet}(\theta;H)^{1/2}\mathcal{E}^{\bullet}(h;H)^{1/2}}{\text{strength}(\theta)}}
\end{equation*}
By Thomson's principle, taking the infimum over all flows $\theta$ from $O^{\bullet}$ to $W^{\bullet}$ in $H^{\bullet}$ in the expression on the RHS, we have that: 
\begin{equation}
\label{eqn: modulus of continuity in terms of Dirichlet energy and extremal length}
    |h(y)-h(x)|\leq{\mathcal{E}^{\bullet}(h;H)^{1/2}\cdot{\lambda(O^{\bullet}\leftrightarrow{W^{\bullet}};H^{\bullet})^{1/2}}}
\end{equation}
In Section \ref{subsec: Tilings of Rectangles} we saw that $\mathcal{E}^{\bullet}(h)$ is the total area of rectangles in the tiling associated with the orthodiagonal rectangle $(G,A^{\bullet},B^{\bullet},C^{\bullet},D^{\bullet})$. Hence, by the definition of $H$, the restriction $\mathcal{E}^{\bullet}(h;H)$ is the total area of rectangles in our tiling that intersect $(h(x),h(y))\times{(0,1)}$. Since the edges of $G$ have length at most $\varepsilon$, Lemma \ref{lem: gradient estimate for tiling map} tells us that the width of any rectangle in our tiling is at most $K\big(\log\big(\frac{d'}{\varepsilon}\big)\big)^{-1}$, where $K>0$ is an absolute constant. Hence:
\begin{equation}
\label{eqn: estimate for Dirichlet energy}
    \mathcal{E}^{\bullet}(h;H)\leq{|h(y)-h(x)|+\frac{2K}{\log\big(\frac{d'}{\varepsilon}\big)}}
\end{equation}
Using the shorthand $\lambda^{\bullet}=\lambda(O^{\bullet}\leftrightarrow{W^{\bullet}};H^{\bullet})$, plugging our estimate in Equation \ref{eqn: estimate for Dirichlet energy} into Equation \ref{eqn: modulus of continuity in terms of Dirichlet energy and extremal length} and rearranging, we have that:
\begin{equation}  
\label{eqn: |h(y)-h(x)| bounded by extremal length}
    |h(y)-h(x)|\leq{\frac{1}{2}\Big(\lambda^\bullet+\sqrt{(\lambda^{\bullet})^{2}+\frac{8K\lambda^{\bullet}}{\log\big(\frac{d'}{\varepsilon}\big)}}\Big)}\leq{\lambda^{\bullet}+\frac{2K}{\log\big(\frac{d'}{\varepsilon}\big)}}=\lambda^{\bullet}(O^{\bullet}\leftrightarrow{W^{\bullet}};H^{\bullet})+\frac{2K}{\log\big(\frac{d'}{\varepsilon}\big)}
\end{equation}
By duality: 
\begin{equation*}
    \lambda(N^{\circ}\leftrightarrow{S^{\circ}};H^{\circ})\cdot\lambda(O^{\bullet}\leftrightarrow{W^{\bullet}};H^{\bullet})=1
\end{equation*}
Thus, to bound $\lambda(O^{\bullet}\leftrightarrow{W^{\bullet}};H^{\bullet})$ 
 and therefore $|h(y)-h(x)|$ from above, it suffices to bound the dual extremal length $\lambda(N^{\circ}\leftrightarrow{S^{\circ}};H^{\circ})$ from below. We will do this by picking by picking a good metric to plug into the variational problem for $\lambda(N^{\circ}\leftrightarrow{S^{\circ}};H^{\circ})$.
 \begin{figure}[H]
    \centering
    \includegraphics[scale=0.53]{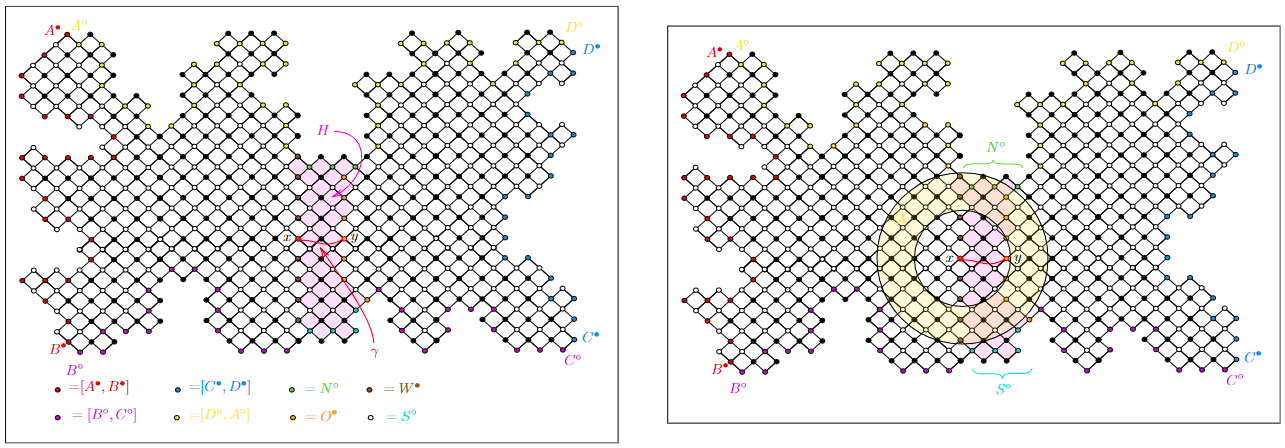}  
    \caption{On the left: The subrectangle $H$ of $G$ with its distinguished boundary arcs $N^{\circ}, O^{\bullet}, S^{\circ}$, and $W^{\bullet}$. On the right: Notice that any path from $N^{\circ}$ to $S^{\circ}$ in $H^{\circ}$ must cross the annulus $\widetilde{\mathcal{A}}$. For simplicity, our orthodiagonal map $G$ here is just a chunk of the square grid but, in general, this need not be the case.}
\end{figure}
\noindent By the definition of $d_{\widehat{G}}(x,y)$ we can find a smooth curve $\gamma$ in $\widehat{G}$ from $x$ to $y$ so that $\text{length}(\gamma)<d_{\Omega}(x,y)+\varepsilon$. By the maximum principle for discrete harmonic functions, $\widehat{G}\setminus{\widehat{H}}$ consists of two connected components $\Omega_{1}$ and $\Omega_{2}$ so that WLOG, $[A^{\bullet},B^{\bullet}]\subseteq{\partial{\Omega_{1}}}$ and $[C^{\bullet},D^{\bullet}]\subseteq{\partial{\Omega_{1}}}$. Thinking of $\widehat{H}$ as a conformal rectangle with the distinguished boundary arcs: 
\begin{align*}
    \widehat{W}&:=\partial{\Omega_{1}}\cap{\partial{\widehat{H}}} & \widehat{O}&:=\partial{\Omega_{2}}\cap{\partial{\widehat{H}}} \\
    \widehat{N}&:=[D^{\bullet}, A^{\bullet}]_{\partial{\widehat{G}}}\cap{\partial{\widehat{H}}} & \widehat{S}&:=[B^{\bullet}, C^{\bullet}]_{\partial{\widehat{G}}}\cap{\partial{\widehat{H}}}
\end{align*}
our curve $\gamma$ from $x$ to $y$ in $\widehat{G}$ must have a subarc $\gamma'$ with endpoints $x'\in{\widehat{W}}, y'\in{\widehat{O}}$, so that $\gamma'$ is a crosscut of $\widehat{H}$. By the definition of $\widehat{H}$, every point of $\widehat{W}$ lies within $\varepsilon$ of a point of $W^{\bullet}$ and every point of $\widehat{O}$ lies within $\varepsilon$ of a point in $O^{\bullet}$. Thus, perturbing $\gamma'$ slightly, we can produce a crosscut $\gamma''$ of $\widehat{H}$ with endpoints $x''\in{\widehat{W}}, y''\in{\widehat{O}}$ so that $\text{length}(\gamma'')\leq{d_{\widehat{G}}(x,y)+3\varepsilon}$. Since $\gamma''$ is a crosscut of $\widehat{H}$ joining $O^{\bullet}$ and $W^{\bullet}$, $\gamma''$ separates $N^{\circ}$ and $S^{\circ}$ in $\widehat{H}$. Hence, any path in $\widehat{H}$ and therefore $H^{\circ}$ from $N^{\circ}$ to $S^{\circ}$ must intersect $\gamma''$. Consider the annulus: 
\begin{equation*}
    \widetilde{\mathcal{A}}:=\{u\in{\mathbb{C}}: d_{\widehat{G}}(x,y)+3\varepsilon<|u-x''|<\frac{d'}{2}\}
\end{equation*}
Observe that: 
\begin{enumerate}
    \item Since $\text{length}(\gamma'')\leq{d_{\widehat{G}}(x,y)+3\varepsilon}$, it follows that $\gamma''\subseteq{B(x'',d_{\widehat{G}}(x,y)+3\varepsilon)}$. Since any path from $N^{\circ}$ to $S^{\circ}$ in $\widehat{H}$ must intersect $\gamma''$, such a path must also intersect the ball $B(x'',d_{\widehat{G}}(x,y)+3\varepsilon)$.
    \item On the other hand, since any path from $N^{\circ}$ to $S^{\circ}$ in $\widehat{H}$ is a also path from $[D^{\bullet}, A^{\bullet}]_{\partial{\widehat{G}}}$ to $[B^{\bullet},C^{\bullet}]_{\partial{\widehat{G}}}$ in $\widehat{G}$, any such curve has diameter $\geq{d'}$. Hence, any curve from $N^{\circ}$ to $S^{\circ}$ in $\widehat{H}$ must at some point exit the ball $B(x'',\frac{d'}{2})$.  
\end{enumerate}
Putting all this together, we have that any path from $N^{\circ}$ to $S^{\circ}$ in $H^{\circ}$ must cross the annulus $\widetilde{\mathcal{A}}$. Having established this fact, consider the metric $\rho:E^{\circ}_{H}\rightarrow{[0,\infty)}$ defined as follows: 
\begin{equation*}
    \rho(e^{\circ})=\int_{e^{\circ}}\frac{|dz|}{|z-x''|}
\end{equation*}
for $e^{\circ}\in{E_{H}^{\circ}}$ contained in the annulus $\widetilde{\mathcal{A}}$, where we think of $e^{\circ}$ as a line segment in the plane. If $e^{\circ}$ is not contained in $\widetilde{\mathcal{A}}$, then $\rho(e^{\circ})=0$. Suppose $\gamma$ is a path in $H^{\circ}$ from $N^{\circ}$ to $S^{\circ}$. We saw earlier that $\gamma$ must cross the annulus $\widetilde{\mathcal{A}}$ at least once. However, when it does this, it is possible that $\gamma$ uses edges $e^{\circ}\in{E^{\circ}_{H}}$ that are not entirely contained in $\widetilde{\mathcal{A}}$ and so have zero mass with respect to $\rho$. Since edges of $G$ have length at most $\varepsilon$, edges of $G^{\circ}$ have length at most $2\varepsilon$. Thus, we can be sure that all of the edges we use when crossing the slightly smaller annulus $\{u\in{\mathbb{C}}:d_{\widehat{G}}(x,y)+5\varepsilon<|u-x''|<\frac{d'}{2}-2\varepsilon\}$ do indeed have positive mass. Hence: 
\begin{equation*}
    l_{\rho}(\gamma)=\int_{\gamma}\frac{|dz|}{|z-x''|}\geq{\log\Big(\frac{d'-4\varepsilon}{2(d_{\widehat{G}}(x,y)+5\varepsilon)}\Big)}
\end{equation*}
By the same argument as in the proof of Lemma \ref{lem: gradient estimate for tiling map}:
\begin{equation*}
    A(\rho)\leq{4\pi\log\Big(\frac{d'-2\varepsilon}{2\hspace{1pt}d_{\widehat{G}}(x,y)}\Big)}+\frac{8\pi\varepsilon}{d_{\widehat{G}}(x,y)}
\end{equation*}
By our assumption that $\varepsilon\leq{d_{\widehat{G}}(x,y)}<\frac{d'}{4}$ and $\varepsilon<\frac{d'}{12}$, we have that: 
\begin{equation*}
    A(\rho)\leq{4\pi\log\Big(\frac{d'-2\varepsilon}{2\hspace{1pt}d_{\widehat{G}}(x,y)}\Big)}+\frac{8\pi\varepsilon}{d_{\widehat{G}}(x,y)}\lesssim{\log\Big(\frac{d'}{d_{\widehat{G}}(x,y)}\Big)}
\end{equation*}
Similarly, if $\gamma$ is a path from $N^{\circ}$ to $S^{\circ}$ in $H^{\circ}$, we have that:
\begin{equation*}
    l_{\rho}(\gamma)=\int_{\gamma}\frac{|dz|}{|z-x''|}\geq{\log\Big(\frac{d'-4\varepsilon}{2(d_{\widehat{G}}(x,y)+5\varepsilon)}\Big)}\gtrsim{\log\Big(\frac{d'}{d_{\widehat{G}}(x,y)}\Big)}
\end{equation*} 
Putting all this together, we have that: 
\begin{equation*}
    \lambda(N^{\circ}\leftrightarrow{S^{\circ}};H^{\circ})\geq{\inf\limits_{\gamma}\frac{\ell^{2}_{\rho}(\gamma)}{A(\rho)}}\gtrsim{\log\Big(\frac{d'}{d_{\widehat{G}}(x,y)}\Big)}
\end{equation*}
where our infimum is over all paths $\gamma$ in $H^{\circ}$ from $N^{\circ}$ to $S^{\circ}$. By duality: 
\begin{equation*}
    \lambda(O^{\bullet}\leftrightarrow{W^{\bullet}};H^{\bullet})=\frac{1}{\lambda(N^{\circ}\leftrightarrow{S^{\circ}};H^{\circ})}\lesssim{\frac{1}{\log\Big(\frac{d'}{d_{\widehat{G}}(x,y)}\Big)}}
\end{equation*}
Plugging this into Equation \ref{eqn: |h(y)-h(x)| bounded by extremal length}, the desired result follows. \\ \\
\underline{\textbf{Case 2:} $d_{cc}^{\widehat{G}}(x,y)\leq{d_{\widehat{G}}(x,y)}$}
By the definition of $d_{cc}^{\widehat{G}}(x,y)$, we can find a crosscut $\gamma$ of $\widehat{G}$ that joins $x$ and $y$ to one of the four distinguished boundary arcs of $(\widehat{G},A^{\bullet},B^{\bullet},C^{\bullet},D^{\bullet})$ and separates it from the opposite boundary arc, such that $\text{length}(\gamma)<d_{cc}^{\widehat{G}}(x,y)+\varepsilon$. We now split our problem into two further cases, depending on whether the relevant boundary arcs of $(\widehat{G},A^{\bullet},B^{\bullet},C^{\bullet},D^{\bullet})$ are Dirichlet arcs where $h$ is constant, or Neumann arcs along which $h$ is monotone. \\ \\
\underline{\textbf{Case 2.1}}: $\gamma$ joins $x$ and $y$ to one of the Dirichlet arcs and separates them from the opposite Dirichlet arc.\\ \\
WLOG, suppose $\gamma$ joins $x$ and $y$ to $[A^{\bullet}, B^{\bullet}]_{\partial{\widehat{G}}}$ and separates them from $[C^{\bullet}, D^{\bullet}]_{\partial{\widehat{G}}}$. Let $\widehat{N}_{x,y}^{\gamma}$ be the connected component of $\widehat{G}\setminus{\gamma}$ containing $x$ and $y$. Define: 
\begin{align*}
    (N_{x,y}^{\gamma})^{\bullet}&:=V^{\bullet}\cap{\widehat{N}_{x,y}^{\gamma}}, & \gamma^{\bullet}&:=\{z\in{V^{\bullet}}: \text{dist}(z,\gamma)\leq{\varepsilon}\}
\end{align*}
By the maximum principle for discrete harmonic functions,
\begin{equation*}
    \max_{z\in{(N_{x,y}^{\gamma})^{\bullet}}}h(z)\leq{\max_{z\in{\gamma^{\bullet}}}h(z)}
\end{equation*}
Let $v$ be a vertex of $\gamma^{\bullet}$ so that $h(v)=\max\limits_{z\in{\gamma^{\bullet}}}h(z)$ and for some neighboring vertex $u$ of $v$ in $V^{\bullet}$, $h(u)<h(v)$. Since $x,y\in{(N_{x,y}^{\gamma})^{\bullet}}$, we have that:
\begin{equation*}
    (h(y)-h(x))<h(v)
\end{equation*}
Similar to our argument in case 1, let $H$ be the suborthodiagonal map of $G$ formed by gluing together all of the quadrilaterals of $G$ that are tangent to a vertex $z$ of $V^{\bullet}$ such that $0\leq{h(z)}<{h(v)}$. By the maximum principle, $H$ is simply connected with a unique, distinguished exterior face. Furthermore, let: 
\begin{align*}
    O^{\bullet}&:=[A^{\bullet}, B^{\bullet}], & W^{\bullet}&:=\{z\in{V^{\bullet}}: h(z)\geq{h(v)}\}\cap{\partial{V_{H}^{\bullet}}}
\end{align*}
Then $O^{\bullet}$ and $W^{\bullet}$ are primal boundary arcs of $H$ with corresponding dual arcs: 
\begin{align*}
    N^{\circ}&:=[B^{\circ}, C^{\circ}]\cap{\partial{V_{H}^{\circ}}}, & S^{\circ}&:=[D^{\circ}, A^{\circ}]\cap{\partial{V_{H}^{\circ}}}
\end{align*}
With these pairs of distinguished primal and dual arcs, $H$ is an orthodiagonal rectangle. By the same argument as in case 1, almost verbatim, we have that:
\begin{equation}
\label{eqn: |h(y)-h(x)| bounded by extremal length case 2.1}
    |h(y)-h(x)|\leq{h(v)}\leq{\lambda^{\bullet}(O^{\bullet}\leftrightarrow{W^{\bullet}};H^{\bullet})+\frac{2K}{\log\big(\frac{d'}{\varepsilon}\big)}}
\end{equation}
By duality: 
\begin{equation*}
    \lambda(N^{\circ}\leftrightarrow{S^{\circ}};H^{\circ})\cdot\lambda(O^{\bullet}\leftrightarrow{W^{\bullet}};H^{\bullet})=1
\end{equation*}
Thus, to bound $\lambda(O^{\bullet}\leftrightarrow{W^{\bullet}};H^{\bullet})$ 
 and therefore $|h(y)-h(x)|$ from above, it suffices to bound the dual extremal length $\lambda(N^{\circ}\leftrightarrow{S^{\circ}};H^{\circ})$ from below. We will do this by picking a good metric to plug into the variational formula for $\lambda(N^{\circ}\leftrightarrow{S^{\circ}};H^{\circ})$. 
 \begin{figure}[H]
    \centering
    \includegraphics[scale=0.53]{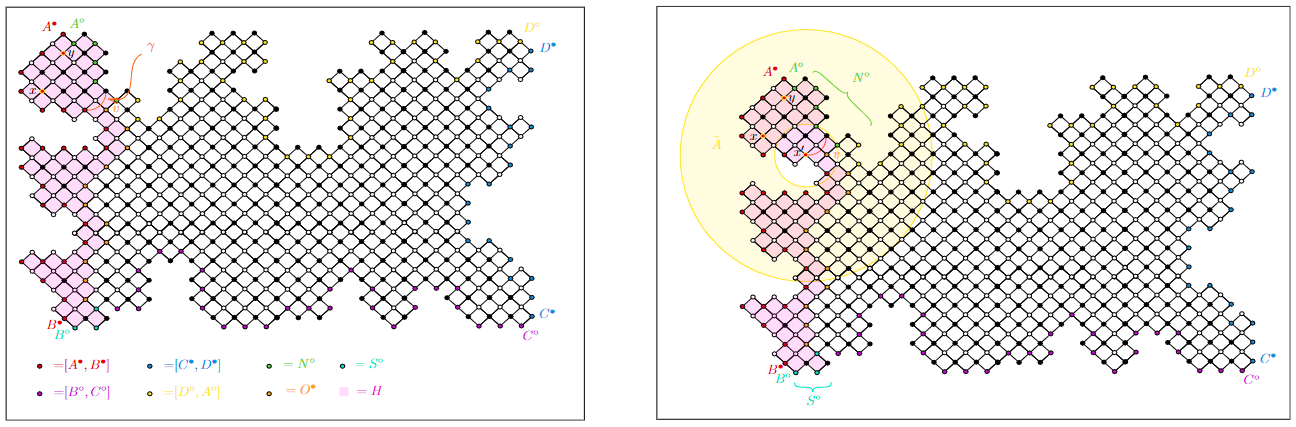}  
    \caption{On the left: The subrectangle $H$ of $G$ with its distinguished boundary arcs $N^{\circ}, O^{\bullet}, S^{\circ}$ and $W^{\bullet}=[A^{\bullet},B^{\bullet}]$. On the right: Notice that any path from $N^{\circ}$ to $S^{\circ}$ in $H^{\circ}$ must cross the annulus $\widetilde{\mathcal{A}}$. For simplicity, our orthodiagonal map $G$ here is just a chunk of the square grid but, in general, this need not be the case.}
\end{figure}
 \noindent Observe that $\gamma$ is a crosscut of $\widehat{G}$ that starts and ends at a point of $[A^{\bullet}, B^{\bullet}]_{\partial{\widehat{G}}}$ and, by the definition of $\gamma^{\bullet}$, $v$ lies within $\varepsilon$ of $\gamma$. Since $v$ has a neighboring vertex $u$ with the property that $h(u)<h(v)$, we also have that $v\in{W^{\bullet}}$. Since the edges of $G$ have length at most $\varepsilon$, putting all this together, it follows that there exists a crosscut $\gamma'$ of $\widehat{H}$ of length at most $\frac{1}{2}(d_{cc}^{\widehat{G}}(x,y)+5\varepsilon)$ joining $W^{\bullet}$ and $O^{\bullet}$ in $\widehat{H}$, thereby separating $N^{\circ}$ and $S^{\circ}$ in $\widehat{H}$. Pick a point $x'$ in $\gamma'$ and consider the annulus: 
\begin{equation*}
    \widetilde{A}=\{u\in{\mathbb{C}}: \frac{1}{2}(d_{cc}^{\widehat{G}}(x,y)+5\varepsilon)<{|u-x'|<\frac{d'}{2}}\}
\end{equation*}
Since any path in $\widehat{G}$ from $[B^{\circ},C^{\circ}]_{\partial{\widehat{G}}}$ to $[A^{\circ}, D^{\circ}]_{\partial{\widehat{G}}}$ has diameter greater than or equal to $d'$, the same is true of any path in $\widehat{H}$ from $N^{\circ}$ to $S^{\circ}$. On the other hand, since $\gamma'$ separates $N^{\circ}$ and $S^{\circ}$ in $\widehat{H}$, any such path must intersect $\gamma'$ and therefore $B(x',\frac{1}{2}(d_{cc}^{\widehat{G}}(x,y)+5\varepsilon))$. Putting all this together, we conclude that any path from $N^{\circ}$ to $S^{\circ}$ in $\widehat{H}$ must cross the annulus $\widetilde{A}$. Hence, if we define the metric $\rho:E_{H}^{\circ}\rightarrow{(0,\infty)}$ by the formula: 
\begin{equation*}
    \rho(e^{\circ})=\begin{cases}
        \int_{e^{\circ}}\frac{|dz|}{|z-x'|} \hspace{5pt}\text{for edges $e^{\circ}\in{E^{\circ}_{H}}$ contained in the annulus $\widetilde{A}$} \\ 
        0 \hspace{5pt}\text{otherwise} 
    \end{cases}
\end{equation*}
 by the same argument as in case 1, almost verbatim, plugging $\rho$ into the variational formula for $\lambda(N^{\circ}\leftrightarrow{S^{\circ}};H^{\circ})$, we have that:
\begin{equation*}
    \lambda(N^{\circ}\leftrightarrow{S^{\circ}};H^{\circ})\gtrsim{\log\Big(\frac{d'}{d_{cc}^{\widehat{G}}(x,y)}\Big)}
\end{equation*}
By duality: 
\begin{equation*}
    \lambda(O^{\bullet}\leftrightarrow{W^{\bullet}};H^{\bullet})\lesssim{\frac{1}{\log\Big(\frac{d'}{d_{cc}^{\widehat{G}}(x,y)}\Big)}}    
\end{equation*}
Plugging this into Equation \ref{eqn: |h(y)-h(x)| bounded by extremal length case 2.1}, the desired result follows. \\ \\
\underline{\textbf{Case 2.2}:} $\gamma$ joins $x$ and $y$ to one of the Neumann arcs and separates them from the opposite Neumann arc. \\ \\
WLOG, suppose $\gamma$ joins $x$ and $y$ to $[B^{\bullet}, C^{\bullet}]_{\partial{\widehat{G}}}$ and separates them from $[D^{\bullet}, A^{\bullet}]_{\partial{\widehat{G}}}$. Similar to case 2.1, let $\widehat{N}_{x,y}^{\gamma}$ be the connected component of $\widehat{G}\setminus{\gamma}$ containing $x$ and $y$ and define: 
\begin{align*}
    (N^{\gamma}_{x,y})^{\bullet}&:=V^{\bullet}\cap{\widehat{N}^{\gamma}_{x,y}}, & \gamma^{\bullet}&:=\{z\in{V^{\bullet}}: \text{dist}(z,\gamma)\leq{\varepsilon}\}    
\end{align*}
By the maximum principle for harmonic functions, since $h^{\bullet}$ is harmonic on the part of the boundary of $(N^{\gamma}_{x,y})^{\bullet}$ that intersects $[B^{\bullet}, C^{\bullet}]$, we have that:
\begin{align*}
    \min_{z\in{\gamma^{\bullet}}}h(z)&\leq{\min_{z\in{(\widehat{N}^{\gamma}_{x,y})^{\bullet}}}h(z)}, & \max_{z\in{(\widehat{N}^{\gamma}_{x,y})^{\bullet}}}h(z)&\leq{\max_{z\in{\gamma^{\bullet}}}h(z)}
\end{align*}
Let $v_{1}$ and $v_{2}$ be vertices of $\gamma^{\bullet}$ so that $h(v_{1})=\min_{z\in{\gamma^{\bullet}}}h(z)$, $h(v_{2})=\max_{z\in{\gamma^{\bullet}}}h(z)$ and $v_{1}$ and $v_{2}$ have neighboring vertices $u_{1}$ and $u_{2}$ so that $h(v_{1})<h(u_{1})$ and $h(v_{2})>h(u_{2})$. Since $x,y\in{(N_{\gamma}^{x,y})^{\bullet}}$, we have that: 
\begin{equation*}
    \big(h(y)-h(x)\big)\leq{\big(h(v_{2})-h(v_{1})\big)}
\end{equation*}
Similar to our argument in case 1, let $H$ be the suborthodiagonal map of $G$ formed by gluing together all of the quadrilaterals of $G$ that are tangent to a vertex $z$ of $V^{\bullet}$ such that $h(v_{1})\leq{h(z)}<{h(v_{2})}$. By the maximum principle, $H$ is simply connected with a unique, distinguished exterior face. Furthermore, let: 
\begin{align*}
    O^{\bullet}&:=\{z\in{V^{\bullet}}: h(z)\geq{h(v_{1})}\}\cap{\partial{V_{H}^{\bullet}}}, & W^{\bullet}&:=\{z\in{V^{\bullet}}: h(z)\leq{h(v_{2})}\}\cap{\partial{V_{H}^{\bullet}}}
\end{align*}
Then $O^{\bullet}$ and $W^{\bullet}$ are primal boundary arcs of $H$ with corresponding dual arcs: 
\begin{align*}
    N^{\circ}&:=[D^{\circ}, A^{\circ}]\cap{\partial{V_{H}^{\circ}}}, & S^{\circ}&:=[B^{\circ}, C^{\circ}]\cap{\partial{V_{H}^{\circ}}}
\end{align*}
With these pairs of distinguished primal and dual arcs, $H$ is an orthodiagonal rectangle. By the same argument as in case 1, almost verbatim, we have that:
\begin{equation}
\label{eqn: |h(y)-h(x)| bounded by extremal length case 2.2}
    |h(y)-h(x)|\leq{h(v_{2})-h(v_{1})}\leq{\lambda^{\bullet}(O^{\bullet}\leftrightarrow{W^{\bullet}};H^{\bullet})+\frac{2K}{\log\big(\frac{d'}{\varepsilon}\big)}}
\end{equation}
where $K>0$ is an absolute constant. By duality: 
\begin{equation*}
    \lambda(N^{\circ}\leftrightarrow{S^{\circ}};H^{\circ})\cdot\lambda(O^{\bullet}\leftrightarrow{W^{\bullet}};H^{\bullet})=1
\end{equation*}
Thus, to bound $\lambda(O^{\bullet}\leftrightarrow{W^{\bullet}};H^{\bullet})$ 
 and therefore $|h(y)-h(x)|$ from above, it suffices to bound the dual extremal length $\lambda(N^{\circ}\leftrightarrow{S^{\circ}};H^{\circ})$ from below.
 \begin{figure}[H]
    \centering
    \includegraphics[scale=0.5]{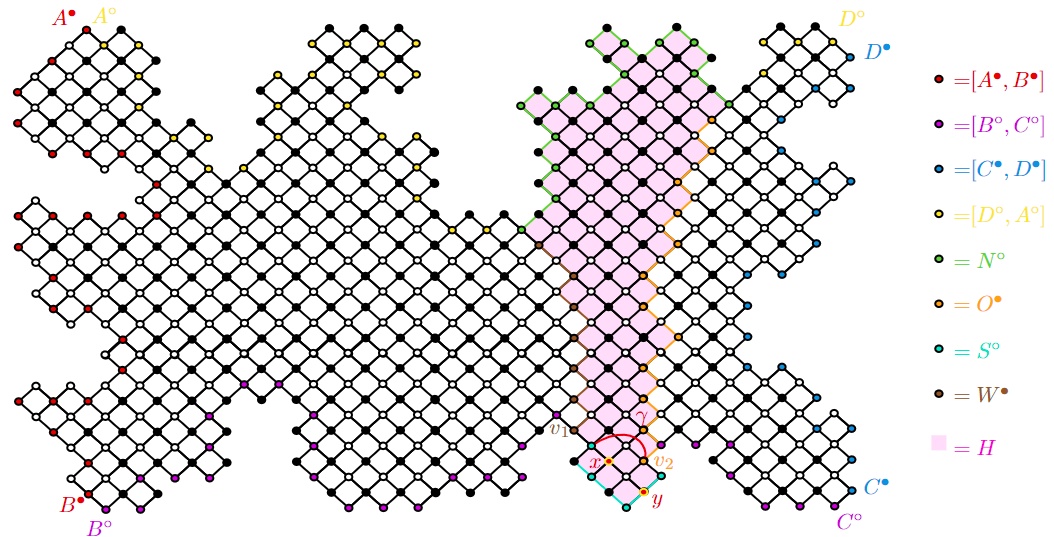}  
    \caption{The subrectangle $H$ of $G$ with its distinguished boundary arcs $N^{\circ}, O^{\bullet}, S^{\circ}$ and $W^{\bullet}$. For simplicity, our orthodiagonal map $G$ here is just a chunk of the square grid but, in general, this need not be the case.}
\end{figure}
 \noindent Since $v_{1}\in{W^{\bullet}}$, $v_{2}\in{O^{\bullet}}$ and each of these points lies within $\varepsilon$ of $\gamma$, there exists a crosscut $\gamma'$ of $\widehat{H}$ of length at most $d_{cc}^{\widehat{G}}(x,y)+3\varepsilon$ joining $O^{\bullet}$ and $W^{\bullet}$ in $\widehat{H}$ and thereby separating $N^{\circ}$ and $S^{\circ}$ in $\widehat{H}$. Hence, by the same argument we used in cases 1 and 2.1, we have that: 
 \begin{equation*}
    \lambda(N^{\circ}\leftrightarrow{S^{\circ}};H^{\circ})\gtrsim{\log\Big(\frac{d'}{d_{cc}^{\widehat{G}}(x,y)}\Big)}
 \end{equation*}
 By duality: 
\begin{equation*}
    \lambda(O^{\bullet}\leftrightarrow{W^{\bullet}};H^{\bullet})\lesssim{\frac{1}{\log\Big(\frac{d'}{d_{cc}^{\widehat{G}}(x,y)}\Big)}}    
\end{equation*}
Plugging this into Equation \ref{eqn: |h(y)-h(x)| bounded by extremal length case 2.2}, the desired result follows. \\ \\
Having shown that $h$ has the prescribed modulus of continuity, observe that the analogous estimate for $\tilde{h}$ follows by the same argument. This completes our proof. 
\end{proof}

\subsection{Two- Sided Estimates for Extremal Length on Orthodiagonal Maps}
\label{subsec: Bounds for EL and Dirichlet Energy}

In this section, we prove two- sided estimates for the discrete extremal length between opposite boundary arcs for an orthodiagonal rectangle that is a $(\delta, \varepsilon)$- good approximation of some conformal rectangle $(\Omega, A,B,C,D)$. This gives us uniform boundedness of the tiling maps $(\phi_{n})_{n=1}^{\infty}$ in Theorem \ref{thm: convergence of rectangle tiling maps}. 
\begin{prop}
\label{prop: Discrete and Continuous EL Comparable}
    Suppose $(\Omega,A,B,C,D)$ is a conformal rectangle and $(G,A^{\bullet}, B^{\bullet}, C^{\bullet}, D^{\bullet})$ is a $(\delta, \varepsilon)$- good interior approximation of $(\Omega, A,B,C,D)$. Suppose $\delta<\frac{\ell\wedge{\ell'}}{2}$, where: 
    \begin{align*}
        \ell=\inf\{\text{length}(\gamma):\text{$\gamma$ is a curve in $\Omega$ joining $[A,B]_{\partial{\Omega^{\ast}}}$ and $[C,D]_{\partial{\Omega^{\ast}}}$}\}, \\
        \ell'=\inf\{\text{length}(\gamma):\text{$\gamma$ is a curve in $\Omega$ joining $[B,C]_{\partial{\Omega^{\ast}}}$ and $[D,A]_{\partial{\Omega^{\ast}}}$}\}
    \end{align*}
    Then:
    \begin{equation*}
        \frac{(\ell-2\delta)^{2}}{2\cdot{\text{Area}(\Omega)}}\leq{\lambda([A^{\bullet},B^{\bullet}]\leftrightarrow{[C^{\bullet},D^{\bullet}]};G^{\bullet})}\leq{\frac{2\cdot\text{Area}(\Omega)}{(\ell'-2\delta)^{2}}}
    \end{equation*}  
\end{prop}

\begin{proof}
    Let: 
    \begin{align*}
    \lambda^{\bullet}&=\lambda([A^{\bullet}, B^{\bullet}]\leftrightarrow{[C^{\bullet}, D^{\bullet}]};G^{\bullet}), & \lambda^{\circ}&=\lambda([B^{\circ},C^{\circ}]\leftrightarrow{[D^{\circ}, A^{\circ}]};G^{\circ})
    \end{align*}
    Plugging the metric $\rho(e^{\bullet})=|e^{\bullet}|$ into the variational problem for $\lambda^{\bullet}$, we have that: 
    \begin{align*}
        \lambda^{\bullet}&\geq{\frac{\Big(\inf\{\sum\limits_{e^{\bullet}\in{\gamma}}|e^{\bullet}|: \text{$\gamma$ is a path from $[A^{\bullet}, B^{\bullet}]$ to $[C^{\bullet},D^{\bullet}]$ in $G^{\bullet}$}\}\Big)^{2}}{\sum\limits_{e^{\bullet}\in{E^{\bullet}}}\frac{|e^{\circ}|}{|e^{\bullet}|}|e^{\bullet}|^{2}}} \\ 
        &=\frac{\Big(\inf\{\text{length}(\gamma): \text{$\gamma$ is a path from $[A^{\bullet}, B^{\bullet}]$ to $[C^{\bullet},D^{\bullet}]$ in $G^{\bullet}$}\}\Big)^{2}}{2\cdot\text{Area}(\widehat{G})} \\ 
        &\geq{\frac{(\ell-2\delta)^{2}}{2\cdot{\text{Area}(\Omega)}}}
    \end{align*}
    The equality on the second line follows from the fact that if $Q$ is an inner face of $G$ with primal diagonal $e^{\bullet}$ and dual diagonal $e^{\circ}$, then $\text{Area}(Q)=\frac{1}{2}|e^{\bullet}|{|e^{\circ}|}$. Thus: 
    $$
    \sum_{e^{\bullet}\in{E^{\bullet}}}|e^{\bullet}||e^{\circ}|=2\sum_{Q\in{F_{in}}}\text{Area}(Q)=2\cdot{\text{Area}(\widehat{G})}
    $$ 
    In the inequality on the third line, two things are going on. On the one hand, since $G$ is an interior approximation of $\Omega$, $\text{Area}(\widehat{G})\leq{\text{Area}(\Omega)}$. On the other hand, suppose $\gamma$ is a path from $[A^{\bullet}, B^{\bullet}]$ to $[C^{\bullet}, D^{\bullet}]$ in $G^{\bullet}$. Since $(G,A^{\bullet}, B^{\bullet}, C^{\bullet}, D^{\bullet})$ is a $(\delta, \varepsilon)$- good approximation of $(\Omega,A,B,C,D)$ (the ``$\delta$" is really the relevant part here), we can modify any such path to get a path of length at most $\text{length}(\gamma)+2\delta$ from $[A,B]_{\partial{\Omega^{\ast}}}$ to $[C,D]_{\partial{\Omega^{\ast}}}$ in $\Omega$. By the definition of $\ell$: 
    \begin{equation*}
         \text{length}(\gamma)+2\delta\geq{\ell}      
    \end{equation*}
    Since this is true of any curve $\gamma$ from $[A^{\bullet}, B^{\bullet}]$ to $[C^{\bullet}, D^{\bullet}]$ in $G^{\bullet}$, the desired result follows. \\ \\
    Similarly, plugging the metric $\rho(e^{\circ})=|e^{\circ}|$ into the variational problem for $\lambda^{\circ}$ we have that: 
     \begin{align*}
        \lambda^{\circ}&\geq{\frac{\Big(\inf\{\sum\limits_{e^{\circ}\in{\gamma}}|e^{\circ}|: \text{$\gamma$ is a path from $[B^{\circ}, C^{\circ}]$ to $[D^{\circ},A^{\circ}]$ in $G^{\circ}$}\}\Big)^{2}}{\sum\limits_{e^{\bullet}\in{E^{\circ}}}\frac{|e^{\bullet}|}{|e^{\circ}|}|e^{\circ}|^{2}}} \\ 
        &=\frac{\Big(\inf\{\text{length}(\gamma): \text{$\gamma$ is a path from $[B^{\circ}, C^{\circ}]$ to $[D^{\circ},A^{\circ}]$ in $G^{\circ}$}\}\Big)^{2}}{2\cdot\text{Area}(\widehat{G})} \\ 
        &\geq{\frac{(\ell'-2\delta)^{2}}{2\cdot{\text{Area}(\Omega)}}}
    \end{align*}
    By Ford- Fulkerson duality for discrete rectangles (Corollary \ref{FF_discrete_rectangles}):
    \begin{equation*}
        \lambda^{\bullet}=\frac{1}{\lambda^{\circ}}\leq{\frac{2\cdot\text{Area}(\Omega)}{(\ell'-2\delta)^{2}}}
    \end{equation*}
\end{proof}
\noindent The two- sided estimate in Proposition \ref{prop: Discrete and Continuous EL Comparable} is very coarse, but it is sufficient for our purposes. It is however worth noting that, as a consequence of our estimates for the modulus of continuity of rectangle tiling maps as well as the corresponding limiting conformal map, we can do significantly better if we are dealing with continuous approximations of a continuous domain or discrete approximations of a discrete domain.
\begin{prop}
\label{prop: Rate of Convergence for Continuous Extremal Length}
   If $(\Omega',A',B',C',D')$ is a $\delta$- good approximation of the conformal rectangle $(\Omega,A,B,C,D)$, $L$ is the extremal length between $[A,B]_{\partial{\Omega^{\ast}}}$ and $[C,D]_{\partial{\Omega^{\ast}}}$ in $\Omega$, $L'$ is the extremal length between $[A',B']_{\partial{\Omega'}}$ and $[C',D']_{\partial{\Omega'}}$ in $\Omega'$, and $\delta>0$ satisfies $\delta\leq{\frac{d'}{2}e^{-\frac{4\pi}{L}}}$ and $\delta\leq{\frac{d}{2}e^{-16\pi{L}}}$, where:
        \begin{align*}
        d&= \inf\{\text{diameter}(\gamma): \text{$\gamma$ is a curve in $\Omega$ joining $[A,B]_{\partial{\Omega^{\ast}}}$ and $[C,D]_{\partial{\Omega^{\ast}}}$}\}\\
        d'&=\inf\{\text{diameter}(\gamma): \text{$\gamma$ is a curve in $\Omega$ joining $[B,C]_{\partial{\Omega^{\ast}}}$ and $[D,A]_{\partial{\Omega^{\ast}}}$}\}    
    \end{align*}
    Then: 
    \begin{equation*}
        -\frac{8\pi}{\log\big(\frac{d'}{2\delta}\big)}\leq{L'-L}\leq{\frac{16\pi{L}^{2}}{\log\big(\frac{d}{2\delta}\big)}}
    \end{equation*}
\end{prop}
\begin{proof}
    Let $\phi:\Omega\rightarrow{\mathcal{R}_{L}}$ be the conformal map from $\Omega$ to the rectangle $\mathcal{R}_{L}$ so that the prime ends $A,B,C,D$ of $\Omega$ are mapped to the four corners of $\mathcal{R}_{L}$ and in particular, $\phi(A)=i$. We write $\phi=h+i\widetilde{h}$ where $h$ is the real part and $\widetilde{h}$ is the imaginary part of $\phi$. Then $\|\nabla{h}\|=\|\nabla{\widetilde{h}}\|$ is the extremal metric giving us the extremal length between $[A,B]_{\partial{\Omega^{\ast}}}$ and $[C,D]_{\partial{\Omega^{\ast}}}$ in $\Omega$. By Theorem \ref{thm: modulus of continuity for conformal map}:
    \begin{align*}
        h(z)&\leq{\frac{2\pi}{\log\big(\frac{d'}{2\delta}\big)}} \hspace{5pt}\text{for $z\in{[A',B']_{\partial{\Omega'}}}$} & h(z)&\geq{L-\frac{2\pi}{\log\big(\frac{d'}{2\delta}\big)}} \hspace{5pt}\text{for $z\in{[C',D']_{\partial{\Omega'}}}$} \\
        \widetilde{h}(z)&\leq{\frac{2\pi{L}}{\log\big(\frac{d}{2\delta}\big)}} \hspace{5pt}\text{for $z\in{[B',C']_{\partial{\Omega'}}}$} & \widetilde{h}(z)&\geq{1-\frac{2\pi{L}}{\log\big(\frac{d}{2\delta}\big)}} \hspace{5pt}\text{for $z\in{[D',A']_{\partial{\Omega'}}}$}
    \end{align*}
    The condition that $\delta\leq{\frac{d'}{2}e^{-\frac{4\pi}{L}}}$ ensures that the values of $h$ on $[C',D']_{\partial{\Omega'}}$ are larger than the values of $h$ on $[A',B']_{\partial{\Omega'}}$. Analogously, the condition that $\delta\leq{\frac{d'}{2}e^{-16\pi{L}}}$ ensures that the values of $\widetilde{h}$ on $[D',A']_{\partial{\Omega'}}$ are larger than the values of $\widetilde{h}$ on $[B',C']_{\partial{\Omega'}}$. Plugging the metric $\|\nabla{h}\|$ into the variational formula for $L'$, we have that: 
    \begin{equation}
    \label{eqn: lower bound for cts EL}
        L'\geq{\frac{\left(L-\frac{4\pi}{\log\big(\frac{d'}{2\delta}\big)}\right)^{2}}{L}}\geq{L-\frac{8\pi}{\log\big(\frac{d'}{2\delta}\big)}}
    \end{equation}
    Similarly, plugging $\|\nabla{\widetilde{h}}\|$ into the variational formula for the dual problem, we have that: 
    \begin{equation}
    \label{eqn: upper bound for cts EL}   
        \frac{1}{L'}\geq{\frac{\left(1-\frac{4\pi{L}}{\log\big(\frac{d}{2\delta}\big)}\right)^{2}}{L}}  \hspace{20pt}\implies\hspace{20pt} L'\leq{\frac{L}{\left(1-\frac{8\pi{L}}{\log\big(\frac{d}{2\delta}\big)}\right)}}\leq{L+\frac{16\pi{L}^{2}}{\log\big(\frac{d}{2\delta}\big)}}
    \end{equation}
    Combining Equations \ref{eqn: lower bound for cts EL} and \ref{eqn: upper bound for cts EL}, we arrive at the desired result. 
\end{proof}

\begin{prop}
\label{prop: Rate of Convergence of Discrete Extremal Length}
    Suppose $(H,W^{\bullet}, X^{\bullet}, Y^{\bullet}, Z^{\bullet})$ is a suborthodiagonal rectangle of $(G,A^{\bullet}, B^{\bullet}, C^{\bullet}, D^{\bullet})$ so that \\$(\widehat{H},W^{\bullet},X^{\bullet},Y^{\bullet}, Z^{\bullet})$ is a $\delta$- good interior approximation of $(\widehat{G},A^{\bullet}, B^{\bullet}, C^{\bullet}, D^{\bullet})$ and $G=(V^{\bullet}\sqcup{V^{\circ}}, E)$ is an orthodiagonal map with edges of length at most $\varepsilon$. Let $L$ denote the extremal length between $[A^{\bullet},B^{\bullet}]$ and $[C^{\bullet}, D^{\bullet}]$ in $G^{\bullet}$ and let $L'$ denote the extremal length between $[W^{\bullet}, X^{\bullet}]$ and $[Y^{\bullet}, Z^{\bullet}]$ in $H^{\bullet}$. Suppose also that $\delta$ and $\varepsilon$ satisfy $(\delta\vee{\varepsilon})\leq{d'e^{-\frac{2K}{L}}}$ and $(\delta\vee{\varepsilon})\leq{de^{-8KL}}$, where: 
    \begin{align*}
        d&= \inf\{\text{diameter}(\gamma): \text{$\gamma$ is a curve in $\Omega$ joining $[A^{\bullet},B^{\bullet}]_{\partial{\widehat{G}}}$ and $[C^{\bullet},D^{\bullet}]_{\partial{\widehat{G}}}$}\}\\
        d'&=\inf\{\text{diameter}(\gamma): \text{$\gamma$ is a curve in $\Omega$ joining $[B^{\bullet},C^{\bullet}]_{\partial{\widehat{G}}}$ and $[D^{\bullet},A^{\bullet}]_{\partial{\widehat{G}}}$}\}   
    \end{align*}
    and $K>0$ is the absolute constant from Theorem \ref{thm: modulus of continuity for tiling maps}. Then:
    \begin{equation*}
        -\frac{4K}{\log\big(\frac{d'}{\delta\vee{\varepsilon}}\big)}\leq{L'-L\leq{\frac{8KL^{2}}{\log\big(\frac{d}{\delta\vee{\varepsilon}}\big)}}}
    \end{equation*}
\end{prop}
\begin{proof}
    Let $h:V^{\bullet}\rightarrow{\mathbb{R}}$ and $\widetilde{h}:V^{\circ}\rightarrow{\mathbb{R}}$ be the real and imaginary parts of the tiling map associated to the orthodiagonal rectangle $(G,A^{\bullet}, B^{\bullet}, C^{\bullet}, D^{\bullet})$. Then the metric $\rho^{\bullet}$ on $G^{\bullet}$ given by the formula $\rho^{\bullet}(u,v)=|h(u)-h(v)|$ for any edge $\{u,v\}\in{E^{\bullet}}$ is the extremal metric giving us the extremal length between $[A^{\bullet}, B^{\bullet}]$ and $[C^{\bullet}, D^{\bullet}]$ in $G^{\bullet}$. Similarly, the metric $\rho^{\circ}$ on $G^{\circ}$ given by the formula $\rho^{\circ}(u,v)=|\widetilde{h}(u)-\widetilde{h}(v)|$ for any edge $\{u,v\}\in{E^{\circ}}$ is the extremal metric giving us the extremal length between $[B^{\circ}, C^{\circ}]$ and $[D^{\circ}, A^{\circ}]$ in $G^{\circ}$. By Theorem \ref{thm: modulus of continuity for tiling maps}, 
    \begin{align*}
        h(z)&\leq{\frac{K}{\log\big(\frac{d'}{\delta\vee{\varepsilon}}\big)}} \hspace{5pt}\text{for $z\in{[W^{\bullet},X^{\bullet}]}$} & h(z)&\geq{L-\frac{K}{\log\big(\frac{d'}{\delta\vee{\varepsilon}}\big)}} \hspace{5pt}\text{for $z\in{[Y^{\bullet},Z^{\bullet}]}$} \\
        \widetilde{h}(z)&\leq{\frac{K\hspace{1pt}L}{\log\big(\frac{d}{\delta\vee{\varepsilon}}\big)}} \hspace{5pt}\text{for $z\in{[X^{\circ},Y^{\circ}]}$} & \widetilde{h}(z)&\geq{1-\frac{K\hspace{1pt}L}{\log\big(\frac{d}{\delta\vee{\varepsilon}}\big)}} \hspace{5pt}\text{for $z\in{[Z^{\circ},W^{\circ}]}$}
    \end{align*}
    The condition that $(\delta\vee{\varepsilon})\leq{d'e^{-\frac{2K}{L}}}$ ensures that the values of $h$ on $[Y^{\bullet}, Z^{\bullet}]$ are larger than the values of $h$ on $[W^{\bullet},X^{\bullet}]$. Similarly, the condition that $(\delta\vee{\varepsilon})\leq{de^{-8KL}}$ ensures that the values of $\widetilde{h}$ on $[Z^{\circ},W^{\circ}]$ are larger than the values of $\widetilde{h}$ on $[X^{\circ},Y^{\circ}]$. Plugging the metric $\rho^{\bullet}$ into the variational formula for $L'$, we have that: 
    \begin{equation}
    \label{eqn: lower bound for discrete EL}
        L'\geq{\frac{\left(L-\frac{2K}{\log\big(\frac{d'}{\delta\vee{\varepsilon}}\big)}\right)^{2}}{L}}\geq{L-\frac{4K}{\log\big(\frac{d'}{\delta\vee{\varepsilon}}\big)}}
    \end{equation}
    Similarly, plugging $\rho^{\circ}$ into the variational formula for the dual problem, we have that: 
    \begin{equation}
    \label{eqn: upper bound for discrete EL}   
        \frac{1}{L'}\geq{\frac{\left(1-\frac{2KL}{\log\big(\frac{d}{\delta\wedge{\varepsilon}}\big)}\right)^{2}}{L}}  \hspace{20pt}\implies\hspace{20pt} L'\leq{\frac{L}{\left(1-\frac{4KL}{\log\big(\frac{d}{\delta\wedge{\varepsilon}}\big)}\right)}}\leq{L+\frac{8K{L}^{2}}{\log\big(\frac{d}{\delta\vee{\varepsilon}}\big)}}
    \end{equation}
    Combining Equations \ref{eqn: lower bound for discrete EL} and \ref{eqn: upper bound for discrete EL}, we arrive at the desired result.
\end{proof}

\section{Limits of Discrete Holomorphic Functions are Holomorphic}
\label{sec: Limits of Discrete Holomorphic Functions are Holomorphic}

Given an orthodiagonal map $G=(V^{\bullet}\sqcup{V^{\circ}}, E)$ recall that a function $F:V^{\bullet}\sqcup{V^{\circ}}\rightarrow{\mathbb{C}}$ is discrete holomorphic on $G$ if for any interior face $Q$ of $G$ with primal diagonal $e^{\bullet}=\{u_{1}, u_{2}\}$ and dual diagonal $e^{\circ}=\{v_{1}, v_{2}\}$ we have that: 
\begin{equation*}
\frac{F(u_{2})-F(u_{1})}{u_{2}-u_{1}}=\frac{F(v_{2})-F(v_{1})}{v_{2}-v_{1}}
\end{equation*}
In particular, notice that if $F|_{V^{\bullet}}$ is strictly real then $F|_{V^{\circ}}$ is strictly imaginary up to an additive constant. This situation is typical of applications of discrete complex analysis. That is, the real part of our discrete holomorphic function typically lives on the primal graph $G^{\bullet}=(V^{\bullet},E^{\bullet})$ and the imaginary part lives on the dual graph $G^{\circ}=(V^{\circ},E^{\circ})$. \\ \\
In this section, we prove the following result of independent interest:
\begin{thm}
\label{thm: limit of discrete holomorphic functions is holomorphic}
    Let $\Omega$ be a subdomain of $\mathbb{C}$ and let $\Omega_{n}=(V_{n}^{\bullet}\sqcup{V_{n}^{\circ}},E_{n})$ be a sequence of orthodiagonal maps so that the edges of $\Omega_{n}$ are of length at most $\varepsilon_{n}$ and $\varepsilon_{n}\rightarrow{0}$ as $n\rightarrow{\infty}$. Suppose that for any compact set $K\subseteq{\Omega}$ there exists $N\in{\mathbb{N}}$ such that for all $n\geq{N}$, $K\subseteq{\widehat{\Omega}_{n}}$. For each $n\in{\mathbb{N}}$, let $F_{n}:V_{n}^{\bullet}\sqcup{V_{n}^{\circ}}\rightarrow{\mathbb{C}}$ be a discrete holomorphic function on $\Omega_{n}$ so that:
    \begin{equation*}
        \text{Re}(F_{n}(z))=0 \hspace{5pt}\text{for all $z\in{V^{\circ}}$}, \hspace{20pt}  \text{Im}(F_{n}(z))=0 \hspace{5pt}\text{ for all $z\in{V^{\bullet}}$}  
    \end{equation*}
    That is, the real part of $F_{n}$ lives on the primal graph $\Omega_{n}^{\bullet}$ and the imaginary part of $F_{n}$ lives on the dual graph $\Omega_{n}^{\circ}$. Suppose also that the Dirichlet energies of the $F_{n}$'s are uniformly bounded on compacts. That is, for any compact set $K\subseteq{\Omega}$ and $N\in{\mathbb{N}}$ such that $K\subseteq{\Omega_{n}}$ for all $n\geq{N}$ we have that: 
    \begin{equation*}
        \sup_{n\geq{N}}\mathcal{E}^{\bullet}_{K}(\text{Re}(F_{n}))=\frac{1}{2}\sup_{n\geq{N}}\Big(\sum_{\substack{Q\in{F_{in}(\Omega_{n})} \\ Q\subseteq{K}}}\frac{|e^{\circ}_{Q}|}{|e^{\bullet}_{Q}|}\Big(\text{Re}(F_{n}((e_{Q}^{\bullet})^{+}))-\text{Re}(F_{n}((e_{Q}^{\bullet})^{-}))\Big)^{2}\Big)<\infty
    \end{equation*}
    Or equivalently, since $\text{Re}(F_{n})$ and $\text{Im}(F_{n})$ are conjugate harmonic functions: 
    \begin{equation*}
        \sup_{n\geq{N}}\mathcal{E}^{\circ}_{K}(\text{Im}(F_{n}))=\frac{1}{2}\sup_{n\geq{N}}\Big(\sum_{\substack{Q\in{F_{in}(\Omega_{n})} \\ Q\subseteq{K}}}\frac{|e^{\bullet}_{Q}|}{|e^{\circ}_{Q}|}\Big(\text{Im}(F_{n}((e_{Q}^{\circ})^{+}))-\text{Im}(F_{n}((e_{Q}^{\circ})^{-}))\Big)^{2}\Big)<\infty
    \end{equation*}
    Let $\widehat{F}_{n}:\widehat{\Omega}_{n}\rightarrow{\mathbb{C}}$ be a sequence of continuous functions on $\widehat{\Omega}_{n}$ so that: 
    \begin{equation*}
        \text{Re}(\widehat{F}_{n}(z))=F_{n}(z) \hspace{5pt}\text{for all $z\in{V^{\bullet}_{n}}$}, \hspace{20pt} \text{Im}(\widehat{F}_{n}(z))=F_{n}(z) \hspace{5pt}\text{for all $z\in{V^{\circ}_{n}}$}
    \end{equation*}
    That is, for each $n\in{\mathbb{N}}$, $\widehat{F}_{n}$ is some sort of sensible extension of $F_{n}$ to a continuous function on $\widehat{\Omega}_{n}$. If: 
    \begin{equation*}
        \widehat{F}_{n}\rightarrow{F} \hspace{4pt}\text{uniformly on compacts in $\Omega$}
    \end{equation*}
    Then $F:\Omega\rightarrow{\mathbb{C}}$ is holomorphic. 
\end{thm}
\noindent In Section \ref{sec: Convergence of Tiling Maps} we will use this result to show that any subsequential limit of our tiling maps is holomorphic. This is crucial to showing that our tiling maps converge to the relevant conformal map.
\begin{rem}
    Note that while Theorem \ref{thm: limit of discrete holomorphic functions is holomorphic} is sufficient for our purposes, stronger results already exist in the literature. Namely, in \cite{CLR23}, Chelkak, Laslier, and Russkikh prove that local uniform limits of discrete holomorphic functions on t- embeddings are holomorphic (see Proposition 6.15 of \cite{CLR23}). Since every orthodiagonal map is a t-embedding (see Section 8.1 of \cite{CLR23}), it follows that the same is true of discrete holomorphic functions on orthodiagonal maps. \\ \\ Having been initially unfamiliar with their work, we found an independent proof of this result in the more restrictive orthodiagonal setting, with the additional condition that the Dirichlet energies of your discrete holomorphic functions must be uniformly bounded on compacts. While this is a strictly weaker result in a less general setting, we still think it is worthwhile to present the proof. The reasons for this are twofold:  
    \begin{enumerate}
        \item The proof is elementary and takes place in a simpler setting. 
        \item It gives us an excuse to introduce Lemma \ref{lem: finding short contours in orthodiagonal maps} which tells us that we can approximate continuous contours by discrete contours in our orthodiagonal map, that are close to the corresponding continuous contour in Hausdorff distance and have comparable length. We use this in \cite{P24} to show that the discrete Dirichlet problem on orthodiagonal maps converges with a polynomial rate to the corresponding continuous Dirichlet problem, for H\"older boundary data.
    \end{enumerate}
\end{rem}
\noindent As per the discussion above, to prove Theorem \ref{thm: limit of discrete holomorphic functions is holomorphic}, we first need the following lemma:
\begin{lem}
\label{lem: finding short contours in orthodiagonal maps}
Suppose $G=(V^{\bullet}\sqcup{V^{\circ}}, E)$ is an orthodiagonal map with mesh at most $\varepsilon$, and $\delta$ is a positive real number so that $\delta\geq{4\varepsilon}$. Suppose $\ell$ is a line segment in $\mathbb{C}$ so that $L=\text{length}(\ell)\geq{8\varepsilon}$ and $\widehat{G}$ contains a $\delta$- neighborhood of $\ell$. Then there exist nearest- neighbor paths $\gamma^{\bullet}$ in $G^{\bullet}$ and $\gamma^{\circ}$ in $G^{\circ}$ that are $\delta$- close to $\ell$ in Hausdorff distance, and:
\begin{equation*}
    \text{length}(\gamma^{\bullet}), \text{length}(\gamma^{\circ})\leq{2L\Big(1+\frac{4\varepsilon}{\delta}\Big)}    
\end{equation*}
Furthermore, the endpoints of $\gamma^{\bullet}$ and $\gamma^{\circ}$ both lie within $\delta$ of the endpoints of $\ell$. 
\end{lem}
\begin{proof}
    Let $G=(V^{\bullet}\sqcup{V^{\circ}}, E)$ be an orthodiagonal map with edges of length at most $\varepsilon$. It follows that the edges of $G^{\bullet}$ and $G^{\circ}$, which correspond to diagonals of inner faces of $G$, have length at most $2\varepsilon$. Let $\ell$ be a line segment in $\mathbb{C}$ so that $\widehat{G}$ contains a $\delta$- neighborhood of $\ell$. Without loss of generality, suppose $\ell=[0,L]$, the line segment between $0$ and $L$ in $\mathbb{C}$. Let $R=[0,L]\times{[-\delta,\delta]}$. Since $\widehat{G}$ contains a $\delta$- neighborhood of $\ell$, $R\subseteq{\widehat{G}}$. \\ \\
    Let $H=(V_{H}^{\bullet}\sqcup{V_{H}^{\circ}},E_{H})$ be the suborthodiagonal map of $G$ formed by taking the union of all the inner faces of $G$ contained in $R$. Since $R$ is convex, $H$ is simply- connected with a unique, distinguished boundary face. Moreover, observe that any pair of neighboring vertices along the boundary of $H$, must be part of an inner face of $G$ that intersects the boundary of the rectangle $R$. Thus, the vertices and edges on the boundary of $H$ all lie within $\varepsilon$ of $\partial{R}$. Let $A^{\bullet}, B^{\bullet}, C^{\bullet}, D^{\bullet}$ be the points of $\partial{V}_{H}^{\bullet}$ closest to the four corners $(0,-\delta), (0,\delta), (L,\delta), (L,-\delta)$ of $R$. Then $(H,A^{\bullet}, B^{\bullet}, C^{\bullet}, D^{\bullet})$ is a $(2\varepsilon, \varepsilon)$- good interior approximation of $R$. 
    \begin{figure}[H]
        \centering
        \includegraphics[scale=0.47]{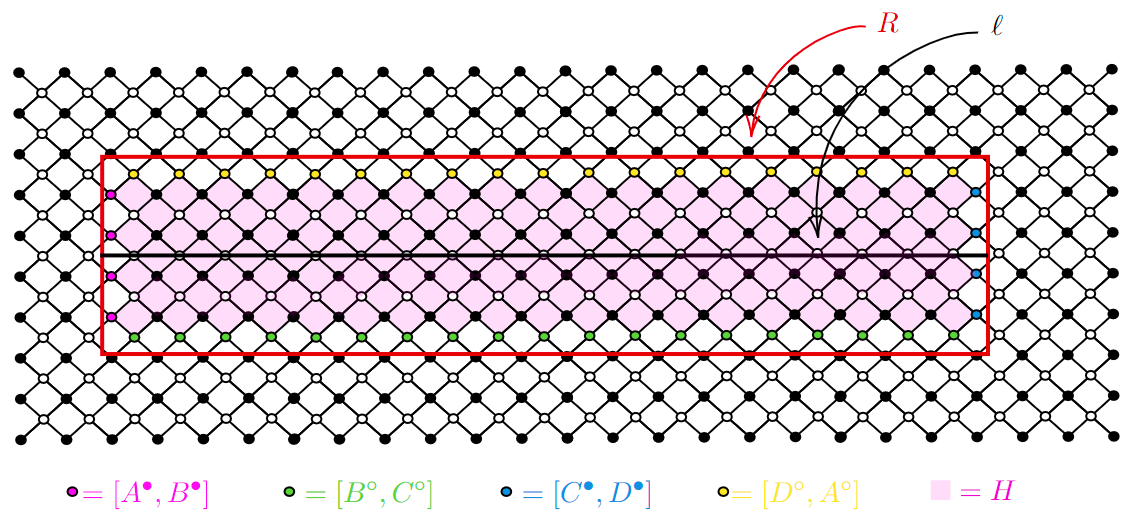}  
        \caption{The $(\varepsilon,2\varepsilon)$- good approximation of the conformal rectangle $R$ (with its four corners as the distinguished boundary points) that we construct in the proof of Lemma \ref{lem: finding short contours in orthodiagonal maps}. For simplicity, our orthodiagonal map $G$ here is just a chunk of the square grid but, in general, this need not be the case.}
    \end{figure}
    \noindent By Proposition \ref{prop: Discrete and Continuous EL Comparable}:
    \begin{equation*}
        \frac{(L-4\varepsilon)^{2}}{4\delta{L}}\leq{\lambda([A^{\bullet},B^{\bullet}]\leftrightarrow{[C^{\bullet},D^{\bullet}]};H^{\bullet})}\leq{\frac{4\delta{L}}{(2\delta-4\varepsilon)^{2}}}
    \end{equation*}
    Plugging the metric $\rho(e^{\bullet})=|e^{\bullet}|$ into the variational problem for $\lambda([A^{\bullet},B^{\bullet}]\leftrightarrow{[C^{\bullet},D^{\bullet}]};H^{\bullet})$ we have that: 
    \begin{align}
        \frac{4\delta{L}}{(2\delta-4\varepsilon)^{2}}&\geq{\lambda([A^{\bullet},B^{\bullet}]\leftrightarrow{[C^{\bullet},D^{\bullet}]};G^{\bullet})} \\ &\geq{\frac{\big(\inf\{\text{length}(\gamma^{\bullet}):\text{$\gamma^{\bullet}$ is a curve in $H^{\bullet}$ joining $[A^{\bullet},B^{\bullet}]$ and $[C^{\bullet},D^{\bullet}]$}\}\big)^{2}}{A(\rho)}} \label{eqn: proof of short contours in orthodiagonal maps}
    \end{align}
    where:
    \begin{align*}
        A(\rho)&=\sum_{e^{\bullet}\in{E^{\bullet}}}\frac{|e^{\circ}|}{|e^{\bullet}|}|e^{\bullet}|^{2}\stackrel{(\ast)}{=}2\sum_{Q\in{F_{in}(H)}}\text{Area}(Q)=2\hspace{1pt}\text{Area}(\widehat{H})\leq{4\delta{L}}
    \end{align*}
    Equality ``$(\ast)$" follows from the fact that if $Q$ is an inner face of $G$ with primal diagonal $e^{\bullet}$ and dual diagonal $e^{\circ}$, then $\text{Area}(Q)=\frac{1}{2}|e^{\bullet}||e^{\circ}|$. Plugging our estimate for $A(\rho)$ into Equation \ref{eqn: proof of short contours in orthodiagonal maps}, we have that: 
    \begin{align*}
        \inf\{\text{length}(\gamma^{\bullet}):\text{$\gamma^{\bullet}$ is a curve in $H^{\bullet}$ joining $[A^{\bullet},B^{\bullet}]$ and $[C^{\bullet},D^{\bullet}]$}\}\leq{\frac{2L}{\big(1-\frac{2\varepsilon}{\delta}\big)}}\leq{2L\Big(1+\frac{4\varepsilon}{\delta}\Big)}
    \end{align*}
    Since any curve in $H^{\bullet}$ joining $[A^{\bullet},B^{\bullet}]$ and $[C^{\bullet},D^{\bullet}]$ is $\delta$- close to $\ell$ in Hausdorff distance, this completes our proof. The proof that we can find a curve $\gamma^{\circ}$ in $G^{\circ}$ with the desired properties follows by the same argument, verbatim, with $G^{\circ}$ in place of $G^{\bullet}$.  
\end{proof}
\noindent Armed with this lemma, we are now ready to prove Theorem \ref{thm: limit of discrete holomorphic functions is holomorphic}:
\begin{proof}
    Since $F$ is the local uniform limit of continuous functions, $F$ must be continuous. Hence, to prove that $F$ is holomorphic, by Problem 3 in Chapter 2 of \cite{SteinShakarchi}, it suffices to check that: 
    $$
    \oint_{\gamma}F(z)dz=0
    $$
    for every simple closed curve $\gamma$ in $\Omega$ that traces out the boundary of a rectangle. \\ \\ 
    With this in mind, let $\gamma$ be a simple closed curve in $\Omega$, oriented counterclockwise, that traces out the boundary of a rectangle $R$ with length $l$ and width $w$. Let $d$ be the distance between the rectangle $R$ and the boundary of $\Omega$. Let $K$ be a compact subset of $\Omega$ that contains a $d/2$- neighborhood of $R$ and let $N$ be a natural number so that $K\subseteq{\widehat{\Omega}_{n}}$ for all $n\geq{N}$.  \\ \\
    Since $(\widehat{F}_{n})_{n=1}^{\infty}$ is a sequence of continuous functions that converges to $F$ uniformly on compacts, by Arzela-Ascoli, it follows that the functions $\widehat{F}_{n}$ are equicontinuous and uniformly bounded on compacts. With this in mind, for any $\delta>0$, let $\omega(\delta)$ denote the modulus of continuity of the family of functions  $(\widehat{F}_{n})_{n=N}^{\infty}$ on $K$. That is, for any $n\geq{N}$ and any points $x,y\in{K}$, we have that:
    $$
    |F_{n}(y)-F_{n}(x)|\leq{\omega(|x-y|)}
    $$
   where:
   $$
   \lim\limits_{\delta\rightarrow{0^{+}}}\omega(\delta)=0
   $$
    Recall that by discrete Morera's theorem, for each $n\in{\mathbb{N}}$, the fact that $F_{n}$ is discrete holomorphic on $\Omega_{n}$ tells us that for any simple closed directed curve $\eta_{n}$ in $\Omega_{n}$: 
    \begin{equation*}
        \sum_{\substack{\vec{e}\in{\eta_{n}} \\ \vec{e}=(e^{-}, e^{+})}}\big(F_{n}(e^{-})+F_{n}(e^{+})\big)(e^{+}-e^{-})=0
    \end{equation*}
    By Lemma \ref{lem: finding short contours in orthodiagonal maps}, for $n\geq{N}$ sufficiently large, we can pick simple closed contours $\gamma_{n}$ and $\eta_{n}$ in $\Omega_{n}$ so that: 
    \begin{enumerate}
        \item $\gamma_{n}$ and $\eta_{n}$ are both oriented counterclockwise.
        \item $\gamma_{n}$ lies outside of the rectangle $R$ and $d_{Haus}(\gamma_{n},\partial{R})=O(\varepsilon_{n})$. 
        \item $\eta_{n}$ lies inside the rectangle $R$ and $d_{Haus}(\eta_{n},\partial{R})=O(\varepsilon_{n})$.
        \item If $\gamma_{n}=(w_{1}, x_{1}, w_{2}, x_{2},...,w_{k_{n}}, x_{k_{n}},w_{1})$ where $w_{1}, w_{2}, ..., w_{k_{n}}\in{V_{n}^{\bullet}}$ and $x_{1}, x_{2}, ..., x_{k_{n}}\in{V_{n}^{\circ}}$, then $w_{i}\sim{w_{i+1}}$ for all $i$, where our indices $i$ are viewed modulo $k_{n}$, so that $\gamma_{n}^{\bullet}=(w_{1}, w_{2}, ..., w_{k_{n}},w_{1})$ is a simple closed contour in $\Omega_{n}^{\bullet}$. Furthermore:
        \begin{equation*}
            \text{length}(\gamma_{n}^{\bullet})=\sum_{i=1}^{k_{n}}|w_{i+1}-w_{i}|=O(l+w)
        \end{equation*}
        \item If $\eta_{n}=(y_{1},z_{1},y_{2},z_{2},..,y_{m_{n}},z_{m_{n}},y_{1})$ where $y_{1},y_{2},...,y_{m_{n}}\in{V_{n}^{\bullet}}$, $z_{1},z_{2},...,z_{m_{n}}\in{V_{n}^{\circ}}$, then $z_{i}\sim{z_{i+1}}$ for all $i$, where our indices are viewed modulo $m_{n}$, so that $\eta_{n}^{\circ}=(z_{1},z_{2},...,z_{m_{n}},z_{1})$ is a simple closed contour in $\Omega_{n}^{\circ}$. Furthermore: 
        \begin{equation*}
            \text{length}(\eta_{n}^{\circ})=\sum_{i=1}^{m_{n}}|z_{i+1}-z_{i}|=O(l+w)
        \end{equation*}
    \end{enumerate}
    \begin{figure}[H]
        \centering
        \includegraphics[scale=0.60]{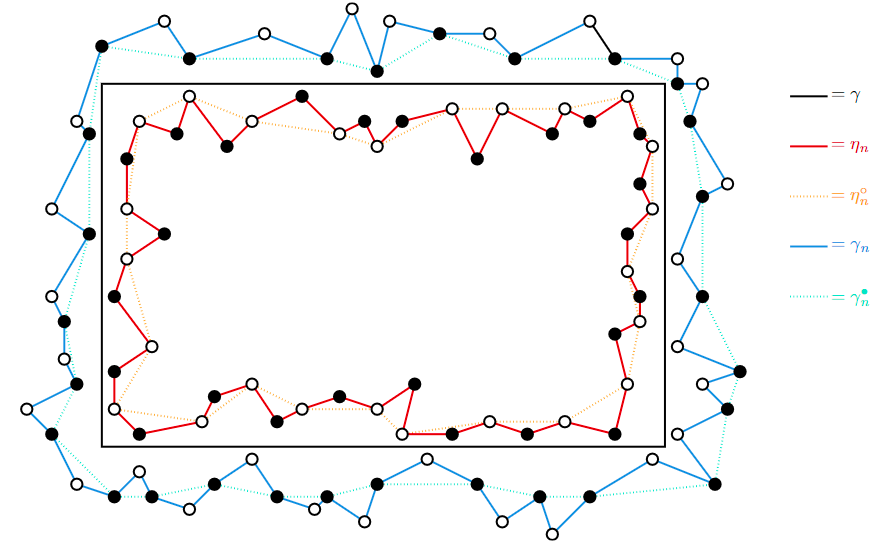} 
        \caption{The contours $\gamma$,  $\gamma_{n}$, $\gamma_{n}^{\bullet}$, $\eta_{n}$, and $\eta_{n}^{\circ}$ for some orthodiagonal map.}
\end{figure}
    \noindent Consider the discrete contour integral of $F_{n}$ over $\gamma_{n}$:
    \begin{equation}
    \label{discrete Morera}
        \sum_{\substack{\vec{e}\hspace{1pt}\in{\gamma_{n}}\\ \vec{e}=(e^{-},e^{+})}}\big(F_{n}(e^{-})+F_{n}(e^{+})\big)(e^{+}-e^{-})=\sum_{i=1}^{k_{n}}F_{n}(x_{i})(w_{i+1}-w_{i})+\sum_{i=1}^{k_{n}}F_{n}(w_{i})(x_{i}-x_{i-1})=0
    \end{equation}
    Th first equality follows by rewriting the original sum over directed edges as a sum over vertices. The second equality is just discrete Morera. Since $\text{Im}(\widehat{F}_{n})$ agrees with $F_{n}$ on $V^{\circ}$, for each directed edge $(w_{i},w_{i+1})$, we have that:  
    \begin{equation*}
        \big|\int_{w_{i}}^{w_{i+1}}\text{Im}(\widehat{F}_{n}(z))dz-F_{n}(x_{i})(w_{i+1}-w_{i})\big|\leq{\omega(\varepsilon_{n})\cdot|w_{i+1}-w_{i}|}
    \end{equation*}
    provided that $n$ is large enough so that the contour $\gamma_{n}$ is contained in $K$. Summing over directed edges we have that: 
    \begin{equation}
    \label{estimate of imaginary part}
        \Big|\oint_{\gamma_{n}^{\bullet}}\text{Im}(\widehat{F}_{n}(z))dz-\sum_{i=1}^{k_{n}}F_{n}(x_{i})(w_{i+1}-w_{i})\Big|=O(\omega(\varepsilon_{n})\cdot{(l+w)})
    \end{equation}
    In other words, we see that the discrete contour integral of $\text{Im}(F_{n})$ over $\gamma_{n}$ is close to the continuous contour integral of $\text{Im}(\widehat{F}_{n})$ over $\gamma_{n}^{\bullet}$. In contrast, since the $x_{i}$'s don't form a simple closed contour in $\Omega_{n}^{\circ}$ and we don't have any control over the quantity: 
    \begin{equation*}
        \sum_{i=1}^{k_{n}}|x_{i+1}-x_{i}|
    \end{equation*}
    it is not clear that we have a similar result comparing the discrete contour integral of $\text{Re}(F_{n})$ over $\gamma_{n}$ to some continuous contour integral of $\text{Re}(\widehat{F}_{n})$. This is where our second contour $\eta_{n}$ comes in.\\ \\ 
    By the same argument we used to handle the behavior of $\text{Im}(F_{n})$ on $\gamma_{n}$, since $\text{Re}(\widehat{F}_{n})$ agrees with $F_{n}$ on $V^{\bullet}$, for any directed edge $(z_{i-1},z_{i})$ we have that:
    \begin{equation*}
        \big|\int_{z_{i-1}}^{z_{i}}\text{Re}(\widehat{F}_{n}(z))dz-F_{n}(y_{i})(z_{i}-z_{i-1})\big|\leq{\omega(\varepsilon_{n})\cdot|z_{i}-z_{i-1}|}
    \end{equation*}
    Summing over directed edges we have that: 
    \begin{equation}
    \label{estimate of real part 1} \Big|\oint_{\eta_{n}^{\circ}}\text{Re}(\widehat{F}_{n}(z))dz-\sum_{i=1}^{m_{n}}F_{n}(y_{i})(z_{i}-z_{i-1})\Big|=O(\omega(\varepsilon_{n})\cdot{(l+w)})
    \end{equation}
    Let $H$ be the suborthodiagonal map of $G$ bounded by the curves $\eta_{n}$ and $\gamma_{n}$. By discrete Green's theorem applied to the function $\text{Re}(F_{n})$ on $H$, we have that: 
    \begin{equation*}
        \sum_{i=1}^{k_{n}}F_{n}(w_{i})(x_{i}-x_{i-1})-\sum_{i=1}^{m_{n}}F_{n}(y_{i})(z_{i}-z_{i-1})=\sum_{\substack{Q\in{F_{in}(H)} \\ Q=[u_{1},v_{1},u_{2},v_{2}]}}\big(F_{n}(u_{2})-F_{n}(u_{1})\big)(v_{2}-v_{1})
    \end{equation*}
    Applying the Cauchy-Schwarz inequality, we have that:
    \begin{align*}
        &\hspace{2pt}\Big|\sum_{\substack{Q\in{F_{in}(H)} \\ Q=[u_{1},v_{1},u_{2},v_{2}]}}\big(F_{n}(u_{2})-F_{n}(u_{1})\big)(v_{2}-v_{1})\Big| \\ \leq&\hspace{2pt}\Big(\sum_{\substack{Q\in{F_{in}(H)} \\ Q=[u_{1},v_{1},u_{2},v_{2}]}}\frac{|v_{2}-v_{1}|}{|u_{2}-u_{1}|}\big(F_{n}(u_{2})-F_{n}(u_{1})\big)^{2}\Big)^{1/2}\Big(\sum_{\substack{Q\in{F_{in}(H)} \\ Q=[u_{1},v_{1},u_{2},v_{2}]}}\frac{|u_{2}-u_{1}|}{|v_{2}-v_{1}|}|v_{2}-v_{1}|^{2}\Big)^{1/2} \\ 
        \leq&\hspace{2pt}\mathcal{E}_{K}^{\bullet}(\text{Re}(F_{n}))^{1/2}\Big(\sum_{\substack{Q\in{F_{in}(H)}  \\ Q=[u_{1},v_{1},u_{2},v_{2}]}}|u_{2}-u_{1}||v_{2}-v_{1}|\Big)^{1/2}=\sqrt{2}\hspace{2pt}\mathcal{E}_{K}^{\bullet}(\text{Re}(F_{n}))^{1/2}(\text{Area}(\widehat{H}))^{1/2} \\
        =&\hspace{2pt}\mathcal{E}_{K}^{\bullet}(\text{Re}(F_{n}))^{1/2}(l+w)^{1/2}\cdot{O(\varepsilon_{n}^{1/2})}\leq{\big(\sup_{n\geq{N}}\mathcal{E}_{K}^{\bullet}(\text{Re}(F_{n}))\big)^{1/2}}(l+w)^{1/2}\cdot{O(\varepsilon_{n}^{1/2})}
    \end{align*}
    Thus: 
    \begin{equation}
    \label{estimate of real part 2}
        \Big|\sum_{i=1}^{k_{n}}F_{n}(w_{i})(x_{i}-x_{i-1})-\sum_{i=1}^{m_{n}}F_{n}(y_{i})(z_{i}-z_{i-1})\big|\leq{\big(\sup_{n\geq{N}}\mathcal{E}_{K}^{\bullet}(\text{Re}(F_{n}))\Big)^{1/2}}(l+w)^{1/2}\cdot{O(\varepsilon_{n}^{1/2})}
    \end{equation}
    By \ref{estimate of imaginary part}, we have that:
    \begin{equation}
    \label{imaginary limit}
        \lim_{n\rightarrow{\infty}}\Big(\oint_{\gamma_{n}^{\bullet}}\text{Im}(\widehat{F}_{n}(z))dz-\sum_{i=1}^{k_{n}}F_{n}(x_{i})(w_{i+1}-w_{i})\Big)=0
    \end{equation}
    Similarly, since the Dirichlet energies of the $F_{n}$'s are uniformly bounded on compacts, combining \ref{estimate of real part 1} and \ref{estimate of real part 2}, we have that:
    \begin{equation}
    \label{real limit}
        \lim_{n\rightarrow{\infty}}\Big(\oint_{\eta_{n}^{\circ}}\text{Re}(\widehat{F}_{n}(z))dz-\sum_{i=1}^{k_{n}}F_{n}(w_{i})(x_{i}-x_{i-1})\Big)=0
    \end{equation}
   Combining \ref{imaginary limit} and \ref{real limit}, we have that:
    \begin{equation*}                   \lim_{n\rightarrow{\infty}}\Big(\oint_{\eta_{n}^{\circ}}\text{Re}(\widehat{F}_{n}(z))dz+\oint_{\gamma_{n}^{\bullet}}\text{Im}(\widehat{F}_{n}(z))dz\Big)=0
    \end{equation*}
    Thus, to prove that:
    \begin{equation*}
        \oint_{\gamma}F(z)dz=\oint_{\gamma}\text{Re}(F(z))dz+\oint_{\gamma}\text{Im}(F(z))dz=0
    \end{equation*}
    It suffices to show that: 
    \begin{equation*}
        \lim_{n\rightarrow{\infty}}\oint_{\eta_{n}^{\circ}}\text{Re}(\widehat{F}_{n}(z))dz=\oint_{\gamma}\text{Re}(F(z))dz
    \end{equation*}
    and: 
     \begin{equation*}
        \lim_{n\rightarrow{\infty}}\oint_{\gamma_{n}^{\bullet}}\text{Im}(\widehat{F}_{n}(z))dz=\oint_{\gamma}\text{Im}(F(z))dz
    \end{equation*}
    To this effect, let $\psi$ be any smooth, real- valued function on $\Omega$. By the triangle inequality: 
    \begin{align*}
    \Big|\oint_{\eta_{n}^{\circ}}\text{Re}(\widehat{F}_{n}(z))dz-\oint_{\gamma}\text{Re}(F(z))dz\Big|\leq&\hspace{2pt}{\Big|\oint_{\eta_{n}^{\circ}}\text{Re}(\widehat{F}_{n}(z))dz-\oint_{\eta_{n}^{\circ}}\text{Re}(F(z))dz\Big|} +\Big|\oint_{\eta_{n}^{\circ}}\text{Re}(F(z))dz-\oint_{\eta_{n}^{\circ}}\psi(z)dz\Big|  \\
    &+\Big|\oint_{\eta_{n}^{\circ}}\psi(z)dz-\oint_{\gamma}\psi(z)dz\Big|+\Big|\oint_{\gamma}\psi(z)dz-\oint_{\gamma}\text{Re}(F(z))dz\Big|
    \end{align*}
    We bound the first term as follows: 
     \begin{align*}
    \Big|\oint_{\eta_{n}^{\circ}}\text{Re}(\widehat{F}_{n}(z))dz-\oint_{\eta_{n}^{\circ}}\text{Re}(F(z))dz\Big|&\leq{\text{length}(\eta_{n}^{\circ})\cdot{\|\text{Re}(\widehat{F}_{n})-\text{Re}(F)\|}}=O\big((l+w)\cdot\|\text{Re}(\widehat{F}_{n})-\text{Re}(F)\|\big)
    \end{align*}
     where $\|\text{Re}(\widehat{F}_{n})-\text{Re}(F)\|$ is the sup norm of $\text{Re}(\widehat{F}_{n})-\text{Re}(F)$ on $K$. We can handle the second and fourth terms analogously:
      \begin{align*}
    \Big|\oint_{\eta_{n}^{\circ}}\text{Re}(F(z))dz-\oint_{\eta_{n}^{\circ}}\psi(z)dz\Big|&\leq{\text{length}(\eta_{n}^{\circ})\cdot{\|\text{Re}(F)-\psi\|}}=O\big((l+w)\cdot{\|\text{Re}(F)-\psi\|}\big)
    \end{align*}
    \begin{equation*}
        \Big|\oint_{\gamma}\psi(z)dz-\oint_{\gamma}\text{Re}(F(z))dz\Big|\leq{\text{length}(\gamma)\cdot{\|\psi-\text{Re}(F)\|}}=2(l+w)\cdot{\|\psi-\text{Re}(F)\|}
    \end{equation*}
    To bound the third term, recall that if $\sigma$ is a simple closed Lipschitz curve, oriented counterclockwise, and $h$ is a complex- valued function whose real and imaginary part are both $C^{2}$, Green's theorem tells us that:
    \begin{equation*}
        \oint_{\sigma}h(z)dz=\int_{D}\partial_{\overline{z}}h(x+iy)dxdy
    \end{equation*}
    where $D$ is the region bounded by $\sigma$. We know that $\gamma$ bounds the rectangle $R$. Let $R_{n}$ denote the region bounded by $\eta_{n}^{\circ}$. By Green's theorem,
    \begin{align*}
        \Big|\oint_{\gamma}\psi(z)dz-\oint_{\eta_{n}^{\circ}}\psi(z)dz\Big|&=\Big|\int_{R}\partial_{\overline{z}}\psi(x+iy)dxdy-\int_{R_{n}}\partial_{\overline{z}}\psi(x+iy)dxdy\Big| \\
        &=O\big((l+w)\cdot\|\partial_{\overline{z}}\psi\|\cdot\varepsilon_{n}\big)
    \end{align*}
    Letting $n\rightarrow{\infty}$, the first and third terms vanish, leaving us with: 
    \begin{equation*}
        \limsup_{n\rightarrow{\infty}}\Big|\oint_{\eta_{n}^{\circ}}\text{Re}(\widehat{F}_{n}(z))dz-\oint_{\gamma}\text{Re}(F(z))dz\Big|=O\big((l+w)\cdot\|\text{Re}(F)-\psi\|\big)    
    \end{equation*}
    Since smooth functions are dense in the space of continuous functions on a compact set with the sup norm, we can choose $\psi$ so that $\|\text{Re}(F)-\psi\|$ is arbitrarily small. Thus: 
    \begin{equation*}
        \limsup_{n\rightarrow{\infty}}\Big|\oint_{\eta_{n}^{\circ}}\text{Re}(\widehat{F}_{n}(z))dz-\oint_{\gamma}\text{Re}(F(z))dz\Big|=0
    \end{equation*}
    From which we conclude that: 
    \begin{equation*}
        \lim_{n\rightarrow{\infty}}\oint_{\eta_{n}^{\circ}}\text{Re}(\widehat{F}_{n}(z))dz=\oint_{\gamma}\text{Re}(F(z))dz
    \end{equation*}
    Applying the same argument to $\big(\oint_{\gamma_{n}^{\bullet}}\text{Im}(\widehat{F}_{n}(z))dz-\oint_{\gamma}\text{Im}(F(z))dz\big)$, we get that:
    \begin{equation*}
        \lim_{n\rightarrow{\infty}}\oint_{\gamma_{n}^{\bullet}}\text{Im}(\widehat{F}_{n}(z))dz=\oint_{\gamma}\text{Im}(F(z))dz
    \end{equation*}
\end{proof} 
   
    \section{Convergence of Tilings to Conformal Maps}
    \label{sec: Convergence of Tiling Maps}
    \subsection{Proof of Theorem \ref{thm: convergence of rectangle tiling maps}}
    \begin{proof}
    Suppose $\Omega\subsetneq{\mathbb{C}}$  is a simply connected domain with distinguished prime ends $A,B,C,D$ listed in counterclockwise order and $\delta_{n}, \varepsilon_{n}>0$ are sequences of positive reals so that: 
    \begin{equation*}
        (\delta_{n},\varepsilon_{n})\rightarrow{(0,0)} \hspace{5pt}\text{as $n\rightarrow{\infty}$}
    \end{equation*}
    For each $n\in{\mathbb{N}}$, let $\Omega_{n}=(V_{n}^{\bullet}\sqcup{V_{n}^{\circ}},E)$ be an orthodiagonal map so that $(\Omega_{n},A_{n}^{\bullet}, B_{n}^{\bullet}, C_{n}^{\bullet}, D_{n}^{\bullet})$ is a $(\delta_{n}, \varepsilon_{n})$- good interior approximation to $(\Omega,A,B,C,D)$ for some choice of distinguished boundary points $A_{n}^{\bullet}, B_{n}^{\bullet}, C_{n}^{\bullet}, D_{n}^{\bullet}\in{\partial{V_{n}^{\bullet}}}$. Let $\phi_{n}$ be the tiling map associated to the orthodiagonal rectangle $(\Omega_{n},A_{n}^{\bullet}, B_{n}^{\bullet}, C_{n}^{\bullet}, D_{n}^{\bullet})$ with real part $h_{n}$ and imaginary part $\widetilde{h}_{n}$. That is, $\widetilde{h}_{n}:V_{n}^{\circ}\rightarrow{\mathbb{R}}$ solves the following boundary value problem on $\Omega_{n}^{\circ}$: 
    \begin{align*}
         \widetilde{h}_{n}(x)&=1 \hspace{5pt}\text{for all $x\in{[B_{n}^{\circ}, C_{n}^{\circ}]}$} \\
         \widetilde{h}_{n}(x)&=0 \hspace{5pt}\text{for all $x\in{[D_{n}^{\circ}, A_{n}^{\circ}]}$} \\
         \Delta^{\circ}\widetilde{h}_{n}(x)&=0 \hspace{5pt}\text{for all $x\in{V_{n}^{\circ}\setminus{\big([B_{n}^{\circ}, C_{n}^{\circ}]\cup{[D_{n}^{\circ}, A_{n}^{\circ}]}\big)}}$}
    \end{align*}
    and $h_{n}:V^{\bullet}\rightarrow{\mathbb{R}}$ is the harmonic conjugate of $\widetilde{h}_{n}$ which solves the following boundary value problem on $\Omega_{n}^{\bullet}$: 
     \begin{align*}
         h_{n}(x)&=0 \hspace{5pt}\text{for all $x\in{[A_{n}^{\bullet}, B_{n}^{\bullet}]}$} \\
        h_{n}(x)&=L_{n} \hspace{5pt}\text{for all $x\in{[C_{n}^{\bullet}, D_{n}^{\bullet}]}$} \\
         \Delta^{\bullet}h_{n}(x)&=0 \hspace{5pt}\text{for all $x\in{V_{n}^{\bullet}\setminus{\big([A_{n}^{\bullet}, B_{n}^{\bullet}]\cup{[C_{n}^{\bullet}, D_{n}^{\bullet}]}\big)}}$}
    \end{align*}
   where $L_{n}$ is the extremal length from $[A_{n}^{\bullet}, B_{n}^{\bullet}]$ to $[C_{n}^{\bullet}, D_{n}^{\bullet}]$ in $\Omega_{n}^{\bullet}$. Let $F_{n}$ be the discrete holomorphic function on $\Omega_{n}$ that agrees with $h_{n}$ on $V_{n}^{\bullet}$ and agrees with $i\widetilde{h}_{n}$ on $V_{n}^{\circ}$. Let $\widehat{F}_{n}$ be any sensible extension of $F_{n}$ to a continuous function on $\widehat{\Omega}_{n}$ so that:
   \begin{equation*}
       \text{Re}(\widehat{F}_{n}(z))=h_{n}(z) \hspace{5pt}\text{for all $z\in{V^{\bullet}_{n}}$}\hspace{20pt} \text{Im}(\widehat{F}_{n}(z))=\widetilde{h}_{n}(z) \hspace{5pt}\text{for all $z\in{V^{\circ}_{n}}$}
   \end{equation*}
  One natural way to do this is to triangulate the faces of $\Omega_{n}^{\bullet}$ and define the real part of $\widehat{F}_{n}$ on each triangle by interpolating linearly between the values of $h_{n}$ at the corner vertices. Analogously, triangulating each face of $\Omega_{n}^{\circ}$, we can define the imaginary part of $\widehat{F}_{n}$ on each triangle by interpolating linearly between the values of $\widetilde{h}_{n}$ at the corner vertices. If $\phi_{n}$ is the tiling map associated with the orthodiagonal rectangle $(\Omega_{n}^{\bullet}, A^{\bullet}_{n}, B^{\bullet}_{n}, C^{\bullet}_{n},D^{\bullet}_{n})$, our estimates for the modulus of continuity of $h_{n}$ and $\widetilde{h}_{n}$ in Theorem \ref{thm: modulus of continuity for tiling maps} tell us that for any $n\in{\mathbb{N}}$ and any $z\in{\widehat{\Omega}_{n}}$, we have that:
  \begin{equation}
  \label{eqn: difference between tiling map and F_n-hat at a point}
      |\phi_{n}(z)-\widehat{F}_{n}(z)|\leq{\frac{K(L_{n}\vee{1})}{\log\big(\frac{d_{n}\wedge{d_{n}'}}{\varepsilon_{n}}\big)}}
  \end{equation}
   where:
   \begin{align*}
       d_{n}&= \inf\{\text{diameter}(\gamma):\text{$\gamma$ is a curve in $\widehat{\Omega}_{n}$ joining $[A_{n}^{\bullet},B_{n}^{\bullet}]_{\partial{\widehat{\Omega}}_{n}}$ and $[C_{n}^{\bullet},D_{n}^{\bullet}]_{\partial{\widehat{\Omega}}_{n}}$}\} \\ 
       d_{n}'&= \inf\{\text{diameter}(\gamma):\text{$\gamma$ is a curve in $\widehat{\Omega}_{n}$ joining $[B_{n}^{\bullet},C_{n}^{\bullet}]_{\partial{\widehat{\Omega}_{n}}}$ and $[D_{n}^{\bullet},A_{n}^{\bullet}]_{\partial{\widehat{\Omega}_{n}}}$}\}
   \end{align*}
   and $K>0$ is an absolute constant. By Proposition \ref{prop: Discrete and Continuous EL Comparable}, the sequence $(L_{n})_{n=1}^{\infty}$ is uniformly bounded. Meanwhile, Lemma \ref{lem: diamaters of curves crossing rectangles converge} tells us that $d_{n}\wedge{d_{n}'}\rightarrow{d\wedge{d'}}$ as $n\rightarrow{\infty}$. Hence, the quantity on the RHS of this inequality tends to $0$ uniformly in $z$ as $n\rightarrow{\infty}$. Thus, to show that the tiling maps $\phi_{n}$ converge to the relevant conformal map $\phi$, it suffices to show this is true of the functions $\widehat{F}_{n}$.\\ \\
  The uniform boundedness of the $L_{n}$'s also tells us that the functions $\widehat{F}_{n}$ are uniformly bounded. By Theorem \ref{thm: modulus of continuity for tiling maps}, these functions are also equicontinuous on compacts. Hence, by Arzela-Ascoli, the functions $\widehat{F}_{n}$ are precompact with respect to the topology of uniform convergence on compacts in $\Omega$. Since $L_{n}$ is precisely the discrete Dirichlet energy of $\widehat{F}_{n}$ and the sequence $(L_{n})_{n=1}^{\infty}$ is uniformly bounded, Theorem \ref{thm: limit of discrete holomorphic functions is holomorphic} tells us that any subsequential limit of our discrete holomorphic functions $\widehat{F}_{n}$ is holomorphic. \\ \\
As per our discussion in Section \ref{subsec: Organization}, we will now identify our limiting function $F$ by its boundary behavior. Suppose $(\widehat{F}_{n_{k}})_{k\geq{1}}$ is a subsequence of $(\widehat{F}_{n})_{n\geq{1}}$ such that $\widehat{F}_{n_{k}}$ converges uniformly on compacts to some holomorphic function $F$ and $L_{n_{k}}$ converges to some $L>0$ as $k\rightarrow{\infty}$. We can always pick such a subsequence because the sequence $(\widehat{F}_{n})_{n\geq{1}}$ is precompact and the sequence $(L_{n})_{n\geq{1}}$ is uniformly bounded away from $0$ and $\infty$. Fix $x,y\in{\Omega}$. By Theorem \ref{thm: modulus of continuity for tiling maps}: 
\begin{equation*}
    |\widehat{F}_{n_{k}}(y)-\widehat{F}_{n_{k}}(x)|\leq{\frac{K(L_{n_{k}}\vee{1})}{\log{\Big(\frac{d_{n_{k}}\wedge{d_{n_{k}}'}}{(d_{\widehat{\Omega}_{n_{k}}}(x,y)\wedge{d_{cc}^{\widehat{\Omega}_{n_{k}}}(x,y)})\vee{\varepsilon_{n_{k}}}}\Big)}}}
\end{equation*}
for all $k$ large enough so that $x,y\in{\widehat{\Omega}_{n_{k}}}$. Taking the limit as $k\rightarrow{\infty}$ of both sides of this inequality, by Corollary \ref{lem: convergence of crosscut and ambient distance} and Lemma \ref{lem: diamaters of curves crossing rectangles converge} of Appendix \ref{appendix: elementary results in topology and complex analysis}, we have that: 
\begin{equation}
\label{eqn: modulus of continuity for limiting holomorphic function}
    |F(y)-F(x)|\leq{\frac{K(L\vee{1})}{\log\Big(\frac{d\wedge{d'}}{d_{\Omega}(x,y)\wedge{d_{cc}^{\Omega}(x,y)}}\Big)}}
\end{equation}
Thus, $F:\Omega\rightarrow{\mathbb{C}}$ extends continuously to $\Omega^{\ast}$. What is more, as we alluded to earlier, our understanding of the boundary behavior of the $\widehat{F}_{n_{k}}$'s will allow us to recover enough of the boundary behavior of $F$ to uniquely identify $F$ as the conformal map $\phi$ in the statement of Theorem \ref{thm: convergence of rectangle tiling maps}.\\ \\
Suppose $z\in{(A,B)_{\partial{\Omega^{\ast}}}}$. Since $z$ does not lie on the boundary arcs $[B,C]_{\partial{\Omega^{\ast}}}$, $[C,D]$ or $[D,A]_{\partial{\Omega^{\ast}}}$: 
\begin{equation*}
    c(z)=\min\{d_{cc}^{\Omega}(z,[B,C]_{\partial{\Omega^{\ast}}};[A,D]_{\partial{\Omega^{\ast}}}), d_{cc}^{\Omega}(z,[C,D]_{\partial{\Omega^{\ast}}};[A,B]_{\partial{\Omega^{\ast}}}), d_{cc}^{\Omega}(z,[D,A]_{\partial{\Omega^{\ast}}};[B,C]_{\partial{\Omega^{\ast}}})\}>0
\end{equation*}
Let $(z_{m})_{m=1}^{\infty}$ be a sequence of points in $\Omega$ so that $z_{m}\rightarrow{z}$ in Carath\`eodory distance. Equivalently, fixing some point $z_{0}\in\Omega$, for each $z_{m}$ there exists a crosscut $\gamma_{m}$ of $\Omega$ so that $z$ and $z_{m}$ lie in the same connected component of $\Omega^{\ast}\setminus{\gamma_{m}}$, $z_{0}$ lies in the other connected component, and $\ell_{m}=\text{length}(\gamma_{m})\rightarrow{0}$. By the same reasoning we used to arrive at Equation \ref{eqn: modulus of continuity for limiting holomorphic function}, restricting our attention to the real part of $F$, we have that: 
\begin{equation*}
    |Re(F)(z)-Re(F)(z_{m})|\leq{\frac{K}{\log\big(\frac{d'}{\ell_{m}}\big)}}
\end{equation*}
where $K>0$ is an absolute constant. Since $\widehat{F}_{n_{k}}$ converges to $F$ uniformly on compacts: 
\begin{equation*}
    |Re(F)(z_{m})-Re(\widehat{F}_{n_{k}})(z_{m})|\rightarrow{0} \hspace{10pt}\text{as $k\rightarrow{\infty}$}
\end{equation*}
Choose $m$ large enough so that $\ell_{m}<\frac{c(z)}{3}$. Having done this, choose $k$ large enough so that $z_{m}\in{\widehat{\Omega}_{n_{k}}}$ and $\delta_{n_{k}},\varepsilon_{n_{k}}<\frac{c(z)}{3}$. Since $\partial{\widehat{\Omega}}_{n_{k}}$ is a closed simple curve separating $z_{m}$ from $\partial{\Omega^{\ast}}$, $\gamma_{m}$ must intersect $\partial{\widehat{\Omega}}_{n_{k}}$. The resulting points of intersections must all lie on the arc $[A_{n_{k}}^{\bullet},B_{n_{k}}^{\bullet}]_{\partial{\widehat{\Omega}_{n_{k}}}}$. To understand why this is the case, suppose to the contrary that one of these points of intersection lies on the arc $[B^{\bullet}_{n_{k}}, C^{\bullet}_{n_{k}}]_{\partial{\widehat{\Omega}_{n_{k}}}}$. Let $x_{k}$ be a point of intersection between $\gamma_{m}$ and $\partial{\widehat{\Omega}}_{n_{k}}$ that lies on $[B^{\bullet}_{n_{k}}, C^{\bullet}_{n_{k}}]_{\partial{\widehat{\Omega}_{n_{k}}}}$. Then we can find a point $x_{k}'\in{[B^{\circ}_{n_{k}}, C^{\circ}_{n_{k}}]}\subseteq{\partial{V^{\circ}}}$ that lies within $\varepsilon_{n_{k}}$ of $x_{k}$. Since $(\Omega_{n_{k}},A_{n_{k}}^{\bullet},B_{n_{k}}^{\bullet},C_{n_{k}}^{\bullet},D_{n_{k}}^{\bullet})$ is a $(\delta_{n_{k}},\varepsilon_{n_{k}})$- good interior approximation of $(\Omega,A,B,C,D)$, $x_{k}'$ is $\delta_{n_{k}}$- close to $[B,C]_{\partial{\Omega^{\ast}}}$ in crosscut distance. That is, we can find a crosscut $\eta_{k}$ of $\Omega$ of length at most $\delta_{n_{k}}$, joining $x_{k}'$ to $[B,C]_{\partial{\Omega^{\ast}}}$, and separating $x_{k}'$ from $[D,A]_{\partial{\Omega^{\ast}}}$. Stitching together $\gamma_{m}$, $\eta_{k}$, and the line segment between $x_{k}$ and $x_{k}'$, we conclude that $z$ is $(\ell_{m}+\delta_{n_{k}}+\varepsilon_{n_{k}})$- close to $[B,C]_{\partial{\Omega^{\ast}}}$ in crosscut distance, where:
\begin{equation*}
    \ell_{m}+\delta_{n_{k}}+\varepsilon_{n_{k}}<c(z)    
\end{equation*}
This is a contradiction, since:
\begin{equation*}
    c(z)\leq{d_{cc}^{\Omega}}(z,[B,C]_{\partial{\Omega^{\ast}}};[D,A]_{\partial{\Omega^{\ast}}})
\end{equation*}
where the quantity on the RHS of this inequality is precisely the crosscut distance between $z$ and $[B,C]_{\partial{\Omega^{\ast}}}$ in $\Omega$. The exact same argument allows us to rule out points of intersection lying on the arcs $[C^{\bullet}_{n_{k}}, D^{\bullet}_{n_{k}}]_{\partial{\widehat{\Omega}_{n_{k}}}}$ and $[D^{\bullet}_{n_{k}}, A^{\bullet}_{n_{k}}]_{\partial{\widehat{\Omega}_{n_{k}}}}$. Hence, $(\gamma_{m}\cap{\partial{\widehat{\Omega}_{n_{k}}}})\subseteq{[A_{n_{k}}^{\bullet},B_{n_{k}}^{\bullet}]_{\partial{\widehat{\Omega}_{n_{k}}}}}$. \\ \\
Since the points of intersection of $\gamma_{m}$ and $\partial{\widehat{\Omega}_{n_{k}}}$ all lie along $[A_{n_{k}}^{\bullet},B_{n_{k}}^{\bullet}]_{\partial{\widehat{\Omega}_{n_{k}}}}$, $z_{m}$ is $\ell_{m}$- close to the boundary arc $[A_{n_{k}}^{\bullet},B_{n_{k}}^{\bullet}]_{\partial{\widehat{\Omega}_{n_{k}}}}$ of $\widehat{\Omega}_{n_{k}}$, along which $Re(\widehat{F}_{n_{k}})$ is identically $0$. Hence, by our regularity estimates in Theorem \ref{thm: modulus of continuity for tiling maps}:
\begin{equation*}
    |Re(\widehat{F}_{n_{k}})(z_{m})|\leq{\frac{K}{\log\big(\frac{d'}{\ell_{m}}\big)}}
\end{equation*}
Putting all this together, we have that: 
\begin{equation*}
    |Re(F)(z)|\leq{\frac{2K}{\log\big(\frac{d'}{\ell_{m}}\big)}}
\end{equation*}
for any $m$ sufficiently large. Since $\ell_{m}\rightarrow{0}$ as $m\rightarrow{\infty}$, we conclude that $Re(F)(z)=0$. Since the point $z\in{(A,B)_{\partial{\Omega^{\ast}}}}$ was arbitrary, $Re(F)(z)=0$ for all $z\in{[A,B]_{\partial{\Omega^{\ast}}}}$. The fact that $Re(F)|_{[C,D]_{\partial{\Omega^{\ast}}}}=L$, $Im(F)|_{[B,C]_{\partial{\Omega^{\ast}}}}=0$, and $Im(F)|_{[D,A]_{\partial{\Omega^{\ast}}}}=1$, all follow by the same argument. Via the argument principle, this information about the boundary behavior of $F$ tells us that $F$ is a conformal map from $\Omega$ to the rectangle $\mathcal{R}_{L}$. Since $F(A)=i$, $F(B)=0$, $F(C)=L$, and $F(D)=L+i$, we conclude that $F=\phi$. Furthermore, $L_{n}$ converges to $L$, where $L$ is the extremal length between $[A,B]_{\partial{\Omega^{\ast}}}$ and $[C,D]_{\partial{\Omega^{\ast}}}$ in $\Omega$. \end{proof}

\begin{appendices}
\section{Appendix}
\subsection{Elementary Results in Topology and Complex Analysis}
\label{appendix: elementary results in topology and complex analysis}
In this section, we prove Proposition \ref{prop: we can approximate any simply connected domain} which tells us that any conformal rectangle admits a $\delta$-good interior approximation, along with a number of elementary results in topology and complex analysis that are needed for the proof of our main result, Theorem \ref{thm: convergence of rectangle tiling maps}. While the proofs of these results are all fairly straightforward, so as not to distract the reader from the main thrust of our argument, we've relegated them to this Appendix.
\begin{lem}
\label{lem: K in Omega_n for n sufficiently large}
    For each $n\in{\mathbb{N}}$, let $(\Omega_{n},A_{n},B_{n},C_{n},D_{n})$ be a $\delta_{n}$- good approximation of the conformal rectangle $(\Omega,A,B,C,D)$, where $\delta_{n}\rightarrow{0}$ as $n\rightarrow{\infty}$. Then for any compact set $K\subseteq{\Omega}$, there exists $N\in{\mathbb{N}}$ such that $K\subseteq{\Omega_{n}}$ for all $n\geq{N}$.  
\end{lem}

\begin{proof}
    Let $\phi:\Omega\rightarrow{\mathcal{R}_{L}}$ be the conformal map from $\Omega$ to the rectangle $\mathcal{R}_{L}$ so that the prime ends $A,B,C,D$ of $\Omega$ are mapped to the four corners of $\mathcal{R}_{L}$ and in particular, $\phi(A)=i$. Here, $L$ is the extremal length between $[A,B]_{\partial{\Omega^{\ast}}}$ and $[C,D]_{\partial{\Omega^{\ast}}}$ in $\Omega$. We write $\phi=h+i\widetilde{h}$ where $h$ and $\widetilde{h}$ are the real and imaginary parts of $\phi$. Since $[A_{n},B_{n}]_{\partial{\Omega_{n}}}$ and $[C_{n}, D_{n}]_{\partial{\Omega_{n}}}$ are $\delta_{n}$-close to $[A,B]_{\partial{\Omega^{\ast}}}$ and $[C,D]_{\partial{\Omega^{\ast}}}$ in crosscut distance, our regularity estimates for $h$ in Theorem \ref{thm: modulus of continuity for conformal map} tell us that: 
    \begin{align*}
        h(z)&\leq{\frac{2\pi}{\log\big(\frac{d'}{2\delta_{n}}\big)}} \hspace{5pt}\text{for all $z\in{[A_{n},B_{n}]_{\partial{\Omega_{n}}}}$}, &
        h(z)&\geq{L-\frac{2\pi}{\log\big(\frac{d'}{2\delta_{n}}\big)}} \hspace{5pt}\text{for all $z\in{[C_{n},D_{n}]_{\partial{\Omega_{n}}}}$}
    \end{align*}
    Similarly, since $[B_{n},C_{n}]_{\partial{\Omega_{n}}}$ and $[D_{n},A_{n}]_{\partial{\Omega_{n}}}$ are $\delta_{n}$- close to $[B,C]$ and $[D,A]$ in crosscut distance, our regularity estimates for $\widetilde{h}$ in Theorem \ref{thm: modulus of continuity for conformal map} tell us that: 
    \begin{align*}
        \widetilde{h}(z)&\geq{1-\frac{2\pi{L}}{\log\big(\frac{d}{2\delta_{n}}\big)}} \hspace{5pt}\text{for all $z\in{[B_{n},C_{n}]_{\partial{\Omega_{n}}}}$}, & \widetilde{h}(z)&\leq{\frac{2\pi{L}}{\log\big(\frac{d}{2\delta_{n}}\big)}} \hspace{5pt}\text{for all $z\in{[B_{n},C_{n}]_{\partial{\Omega_{n}}}}$}
    \end{align*}
    Since $A_{n}$ lies on both $[A_{n}, B_{n}]_{\partial{\widehat{\Omega}_{n}}}$ and $[C_{n}, D_{n}]_{\partial{\widehat{\Omega}_{n}}}$, we have that:  
    \begin{equation*}
        |\phi(A_{n})-\phi(A)|\leq{\frac{2\pi}{\log\big(\frac{d'}{2\delta_{n}}\big)}+\frac{2\pi{L}}{\log\big(\frac{d}{2\delta_{n}}\big)}}\leq{\frac{4\pi(L\vee{1})}{\log\big(\frac{d\wedge{d'}}{2\delta_{n}}\big)}}
    \end{equation*}
    By the same reasoning: 
    \begin{equation*}
        |\phi(B_{n})-\phi(B)|, |\phi(C_{n})-\phi(C)|, |\phi(D_{n})-\phi(D)|\leq{\frac{4\pi(L\vee{1})}{\log\big(\frac{d\wedge{d'}}{2\delta_{n}}\big)}} 
    \end{equation*}
    Consider the simple closed curve $\phi(\partial{\Omega_{n}})$ whose interior is precisely $\phi(\Omega_{n})$. We can split this curve into four pieces as follows: 
    \begin{equation*}
        \phi(\partial{\Omega}_{n})=\phi([A_{n}, B_{n}]_{\partial{\Omega_{n}}})\cup{\phi([B_{n}, C_{n}]_{\partial{\Omega_{n}}})}\cup{\phi([C_{n}, D_{n}]_{\partial{\Omega_{n}}})}\cup{\phi([D_{n}, A_{n}]_{\partial{\Omega_{n}}})}
    \end{equation*}
    By our estimates from earlier:
    \begin{itemize}
        \item The curve $\phi([A_{n}, B_{n}]_{\partial{\Omega}_{n}})$ travels from $\phi(A_{n})$ to $\phi(B_{n})$, where $\phi(A_{n})$ lies within $\frac{4\pi(L\vee{1})}{\log(\frac{d\wedge{d'}}{2\delta_{n}})}$ of $i$ and $\phi(B_{n})$ lies within $\frac{4\pi(L\vee{1})}{\log(\frac{d\wedge{d'}}{2\delta_{n}})}$ of $0$. In between, the curve $\phi([A_{n}, B_{n}]_{\partial{\Omega}})$ lies within $\frac{2\pi}{\log(\frac{d'}{2\delta_{n}})}$ of the boundary arc $\{0\}\times{[0,1]}$ of $\mathcal{R}_{L}$.
         \item The curve $\phi([B_{n}, C_{n}]_{\partial{\Omega}_{n}})$ travels from $\phi(B_{n})$ to $\phi(C_{n})$, where $\phi(B_{n})$ lies within $\frac{4\pi(L\vee{1})}{\log(\frac{d\wedge{d'}}{2\delta_{n}})}$ of $0$ and $\phi(B_{n})$ lies within $\frac{4\pi(L\vee{1})}{\log(\frac{d\wedge{d'}}{2\delta_{n}})}$ of $L$. In between, the curve $\phi([B_{n}, C_{n}]_{\partial{\Omega}})$ lies within $\frac{2\pi{L}}{\log(\frac{d}{2\delta_{n}})}$ of the boundary arc $[0,L]\times{\{0\}}$ of $\mathcal{R}_{L}$.
          \item The curve $\phi([C_{n}, D_{n}]_{\partial{\Omega}_{n}})$ travels from $\phi(C_{n})$ to $\phi(D_{n})$, where $\phi(C_{n})$ lies within $\frac{4\pi(L\vee{1})}{\log(\frac{d\wedge{d'}}{2\delta_{n}})}$ of $L$ and $\phi(D_{n})$ lies within $\frac{4\pi(L\vee{1})}{\log(\frac{d\wedge{d'}}{2\delta_{n}})}$ of $L+i$. In between, the curve $\phi([C_{n}, D_{n}]_{\partial{\Omega}})$ lies within $\frac{2\pi}{\log(\frac{d'}{2\delta_{n}})}$ of the boundary arc $\{L\}\times{[0,1]}$ of $\mathcal{R}_{L}$.
            \item The curve $\phi([D_{n}, A_{n}]_{\partial{\Omega}_{n}})$ travels from $\phi(D_{n})$ to $\phi(A_{n})$, where $\phi(D_{n})$ lies within $\frac{4\pi(L\vee{1})}{\log(\frac{d\wedge{d'}}{2\delta_{n}})}$ of $L+i$ and $\phi(A_{n})$ lies within $\frac{4\pi(L\vee{1})}{\log(\frac{d\wedge{d'}}{2\delta_{n}})}$ of $i$. In between, the curve $\phi([D_{n}, A_{n}]_{\partial{\Omega}})$ lies within $\frac{2\pi{L}}{\log(\frac{d}{2\delta_{n}})}$ of the boundary arc $[0,L]\times{\{1\}}$ of $\mathcal{R}_{L}$.
    \end{itemize}
   Putting all this together, it follows that:
   \begin{equation*}
       \left[\frac{4\pi(L\vee{1})}{\log(\frac{d\wedge{d'}}{2\delta_{n}})}, L-\frac{4\pi(L\vee{1})}{\log(\frac{d\wedge{d'}}{2\delta_{n}})}\right]\times{\left[\frac{4\pi(L\vee{1})}{\log(\frac{d\wedge{d'}}{2\delta_{n}})}, 1-\frac{4\pi(L\vee{1})}{\log(\frac{d\wedge{d'}}{2\delta_{n}})}\right]}\subseteq{\phi(\Omega_{n})}
   \end{equation*}
    Fix a compact set $K\subseteq{\Omega}$ and let $D=\text{dist}(\phi(K),\mathcal{R}_{L})$. Then for any $n$ sufficiently large so that $\frac{4\pi(L\vee{1})}{\log(\frac{d\wedge{d'}}{2\delta_{n}})}\leq{D}$, we have that $K\subseteq{\Omega_{n}}$. 
\end{proof}
\begin{rem}
    Notice that if the orthodiagonal rectangle $(G,A^{\bullet},B^{\bullet},C^{\bullet},D^{\bullet})$ is a $(\delta_{n}, \varepsilon_{n})$- good interior approximation of the conformal rectangle $(\Omega,A,B,C,D)$ then $(\widehat{G},A^{\bullet},B^{\bullet},C^{\bullet},D^{\bullet})$ is a $(\delta_{n}+2\varepsilon_{n})$- good interior approximation of $(\Omega,A,B,C,D)$. Hence, the analogue of Lemma \ref{lem: K in Omega_n for n sufficiently large} for discrete approximations of a continuous domain also holds.
\end{rem}
\begin{lem}
\label{lem: convergence of crosscut and ambient distance}
    Suppose $(\Omega,A,B,C,D)$ is a conformal rectangle and for each $n\in{\mathbb{N}}$,  $(\Omega_{n},A_{n},B_{n},C_{n},D_{n})$ is a $\delta_{n}$-good interior approximation of $(\Omega,A,B,C,D)$, where $\delta_{n}\rightarrow{0}$ as $n\rightarrow{\infty}$. Then: 
    \begin{align*}
        \lim_{n\rightarrow{\infty}}d_{\Omega_{n}}(x,y)&=d_{\Omega}(x,y), & \lim_{n\rightarrow{\infty}}d_{cc}^{\Omega_{n}}(x,y)&=d_{cc}^{\Omega}(x,y)
    \end{align*}
    for any $x,y\in{\Omega}$. 
\end{lem}
\begin{proof}
    First we address the convergence of ambient distance. Fix $x,y\in{\Omega}$. By Lemma \ref{lem: K in Omega_n for n sufficiently large}, we know that for any $n$ sufficiently large, $x,y\in{\Omega_{n}}$. For such $n$, $d_{\Omega}(x,y)\leq{d_{\Omega_{n}}(x,y)}$, since $\Omega_{n}\subseteq{\Omega}$. Hence: 
    \begin{equation*}
        d_{\Omega}(x,y)\leq{\liminf_{n\rightarrow{\infty}}d_{\Omega_{n}}(x,y)}
    \end{equation*}
    To get the reverse inequality, fix $\varepsilon>0$ and let $\gamma$ be a smooth curve from $x$ to $y$ in $\Omega$ so that $\text{length}(\gamma)\leq{d_{\Omega}(x,y)+\varepsilon}$. Since $\gamma$ is a compact subset of $\Omega$, Lemma \ref{lem: K in Omega_n for n sufficiently large} tells us that for all $n$ sufficiently large, $\gamma\subseteq{\Omega_{n}}$, and so $d_{\Omega_{n}}(x,y)\leq{d_{\Omega}(x,y)+\varepsilon}$. Hence: 
    \begin{equation*}
        \limsup_{n\rightarrow{\infty}}d_{\Omega_{n}}(x,y)\leq{d_{\Omega}(x,y)+\varepsilon}
    \end{equation*}
    Since $\varepsilon>0$ was arbitrary, letting $\varepsilon$ tend to $0$, we have that: 
    \begin{equation*}
        \limsup_{n\rightarrow{\infty}}d_{\Omega_{n}}(x,y)\leq{d_{\Omega}(x,y)}
    \end{equation*}
    Regarding the convergence of crosscut distance, again, fix $x,y\in{\Omega}$ and suppose $n$ is sufficiently large so that $x,y\in{\Omega_{n}}$. We will show that: 
    \begin{equation}
    \label{eqn: convergence of crosscut distance wrt one boundary arc}
        \lim_{n\rightarrow{\infty}}d_{cc}^{\Omega_{n}}(\{x,y\},[A_{n},B_{n}]_{\partial{\Omega_{n}}}; [C_{n},D_{n}]_{\partial{\Omega_{n}}})=d_{cc}^{\Omega}(\{x,y\},[A,B]_{\partial{\Omega^{\ast}}};[C,D]_{\partial{\Omega^{\ast}}})
    \end{equation}
    The corresponding result for the other three pairs of boundary arcs follows by the exact same argument. Hence, if we can verify Equation \ref{eqn: convergence of crosscut distance wrt one boundary arc}, it follows that: 
    \begin{equation*}
        \lim_{n\rightarrow{\infty}}d_{cc}^{\Omega_{n}}(x,y)=d_{cc}^{\Omega}(x,y)
    \end{equation*}
    Suppose $\gamma_{n}$ is a crosscut of $\Omega_{n}$ that joins $x$ and $y$ to $[A_{n},B_{n}]_{\partial{\Omega_{n}}}$ and separates $x$ and $y$ from $[C_{n},D_{n}]_{\partial{\Omega_{n}}}$. Let $u_{n}$ and $v_{n}$ be the endpoints of $\gamma_{n}$. Then there are three possibilities: 
    \begin{enumerate}
        \item Both $u_{n}$ and $v_{n}$ lie on $[A_{n},B_{n}]_{\partial{\Omega_{n}}}$.
        \item One point lies on $[A_{n},B_{n}]_{\partial{\Omega_{n}}}$, while the other lies on a neighboring arc: either $[B_{n},C_{n}]_{\partial{\Omega_{n}}}$ or $[D_{n},A_{n}]_{\partial{\Omega_{n}}}$.
        \item One points lies on $[B_{n},C_{n}]$ and the other point lies on $[D_{n},A_{n}]_{\partial{\Omega_{n}}}$.
    \end{enumerate}
    In all three cases, the argument is analogous, so for simplicity, let's assume we are in the first situation. Since $(\Omega_{n},A_{n},B_{n},C_{n},D_{n})$ is a $\delta_{n}$-good interior approximation of $(\Omega,A,B,C,D)$, there exist crosscuts $\eta_{n}^{(1)}$ and $\eta_{n}^{(2)}$ of $\Omega$ so that: 
    \begin{itemize}
        \item  $\eta_{n}^{(1)}$ joins $u_{n}$ to $[A,B]_{\partial{\Omega^{\ast}}}$ and separates $u_{n}$ from $[C,D]_{\partial{\Omega^{\ast}}}$
        \item $\eta_{n}^{(2)}$ joins $v_{n}$ to $[A,B]_{\partial{\Omega^{\ast}}}$ and separates $v_{n}$ from $[C,D]_{\partial{\Omega^{\ast}}}$
        \item $\text{length}(\eta_{n}^{(1)}), \text{length}(\eta_{n}^{(2)})<{\delta_{n}}$
    \end{itemize}
    Stitching together $\eta_{n}^{(1)}, \gamma_{n}$ and $\eta_{n}^{(2)}$, we can produce a crosscut of $\Omega$ that joins $x$ and $y$ to $[A,B]_{\partial{\Omega^{\ast}}}$ and separates $x$ and $y$ from $[C,D]_{\partial{\Omega^{\ast}}}$ with length at most $\text{length}(\gamma_{n})+2\delta_{n}$. Since this is true for any such crosscut $\gamma_{n}$ of $\Omega_{n}$, it follows that: 
    \begin{equation*}
        d_{cc}^{\Omega}(\{x,y\},[A,B]_{\partial{\Omega^{\ast}}};[C,D]_{\partial{\Omega^{\ast}}})<d_{cc}^{\Omega_{n}}(\{x,y\},[A_{n},B_{n}]_{\partial{\Omega_{n}}};[C_{n},D_{n}]_{\partial{\Omega_{n}}})+2\delta_{n}
    \end{equation*}
    for all $n$ sufficiently large so that $x,y\in{\Omega_{n}}$. Hence: 
    \begin{equation*}
        d_{cc}^{\Omega}(\{x,y\},[A,B]_{\partial{\Omega^{\ast}}};[C,D]_{\partial{\Omega^{\ast}}})\leq{\liminf_{n\rightarrow{\infty}}d_{cc}^{\Omega_{n}}(\{x,y\},[A_{n},B_{n}]_{\partial{\Omega_{n}}};[C_{n},D_{n}]_{\partial{\Omega_{n}}})}
    \end{equation*}
    \begin{figure}[H]
        \centering
        \includegraphics[scale=0.65]{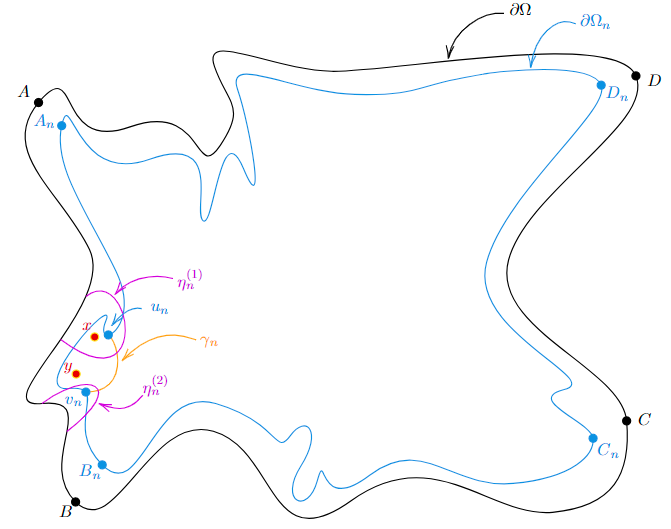}  
        \caption{Stitching together $\eta_{n}^{(1)}$, $\eta_{n}^{(2)}$, and $\gamma_{n}$ gives us a crosscut of $\Omega$ that joins $x$ and $y$ to $[A,B]_{\partial{\Omega^{\ast}}}$ and separates $x$ and $y$ from $[C,D]_{\partial{\Omega^{\ast}}}$.}
    \end{figure}
    \noindent To get the reverse inequality, suppose $\gamma$ is a crosscut of $\Omega$ that joins $x$ and $y$ to $[A,B]_{\partial{\Omega^{\ast}}}$ and separates $x$ and $y$ from $[C,D]_{\partial{\Omega^{\ast}}}$. Let $\gamma:[0,1]\rightarrow{\Omega^{\ast}}$ be an explicit parametrization of $\gamma$ and let $u,v\in{\partial{\Omega^{\ast}}}$ be the endpoints of $\gamma$: $u=\gamma(0)$, $v=\gamma(1)$. Since any crosscut of $\Omega$ that starts or ends at a corner $A,B,C,D$ of our conformal rectangle is close to a crosscut of $\Omega$ with almost the same length that doesn't use any of the corners of $(\Omega,A,B,C,D)$, we can assume without loss of generality that neither $u$ nor $v$ is a corner of $(\Omega,A,B,C,D)$. From here, there are three possibilities: 
    \begin{enumerate}
        \item Both $u$ and $v$ lie on $(A,B)_{\partial{\Omega^{\ast}}}$. 
        \item One point lies on $(A,B)_{\partial{\Omega^{\ast}}}$ while the other lies on a neighboring arc: either $(B,C)_{\partial{\Omega^{\ast}}}$ or $(D,A)_{\partial{\Omega^{\ast}}}$.
        \item One point lies on $(B,C)_{\partial{\Omega^{\ast}}}$ and the other point lies on $(D,A)_{\partial{\Omega^{\ast}}}$.
    \end{enumerate}
    The argument in all three cases is analogous, so for simplicity, suppose we are in the first situation. Then: 
    \begin{align*}
        c(u)&=\min\{d_{cc}^{\Omega}(u,[B,C]_{\partial{\Omega^{\ast}}};[A,D]_{\partial{\Omega^{\ast}}}), d_{cc}^{\Omega}(u,[C,D]_{\partial{\Omega^{\ast}}};[A,B]_{\partial{\Omega^{\ast}}}), d_{cc}^{\Omega}(u,[D,A]_{\partial{\Omega^{\ast}}};[B,C]_{\partial{\Omega^{\ast}}})\}>0 \\ 
        c(v)&=\min\{d_{cc}^{\Omega}(v,[B,C]_{\partial{\Omega^{\ast}}};[A,D]_{\partial{\Omega^{\ast}}}), d_{cc}^{\Omega}(v,[C,D]_{\partial{\Omega^{\ast}}};[A,B]_{\partial{\Omega^{\ast}}}), d_{cc}^{\Omega}(v,[D,A]_{\partial{\Omega^{\ast}}};[B,C]_{\partial{\Omega^{\ast}}})\}>0
    \end{align*}
    Choose $s,t\in(0,1)$ so that:
    \begin{align*}
        \text{length}(\gamma[0,s])&<\frac{c(u)}{4\pi}, & \text{length}(\gamma[t,1])&<\frac{c(v)}{4\pi}
    \end{align*}
     Since $\gamma[s,t]$ is a compact subset of $\Omega$, by Lemma \ref{lem: K in Omega_n for n sufficiently large}, $\gamma[s,t]\subseteq{\Omega_{n}}$ for all $n$ sufficiently large. In this case, since $\gamma$ starts and ends outside $\Omega_{n}$, $\gamma$ must intersect $\Omega_{n}$ at least twice: once when it enters $\Omega_{n}$ and once when it exits. Since $\gamma[s,t]\subseteq{\Omega_{n}}$, the points of intersection of $\gamma$ with $\partial{\Omega_{n}}$ must all lie in $\gamma[0,s]$ or $\gamma[t,1]$. Since $u=\gamma(0)\in{(A,B)_{\partial{\Omega^{\ast}}}}$ and $\text{diam}(\gamma[0,s])<\frac{c(u)}{4\pi}$, there is a subarc $\chi$ of $\{z\in{\Omega^{\ast}}:|z-u|=\frac{c(u)}{4\pi}\}$ so that $[C,D]_{\partial{\Omega^{\ast}}}$ and $\gamma[0,s]$ lie in distinct connected components of $\Omega^{\ast}\setminus{\chi}$. In other words: 
    \begin{equation*}
        d_{cc}^{\Omega}(\gamma[0,s],[A,B]_{\partial{\Omega^{\ast}}};[C,D]_{\partial{\Omega^{\ast}}})<\frac{c(u)}{2}
    \end{equation*}
     \begin{figure}[h]
        \centering
        \includegraphics[scale=0.52]{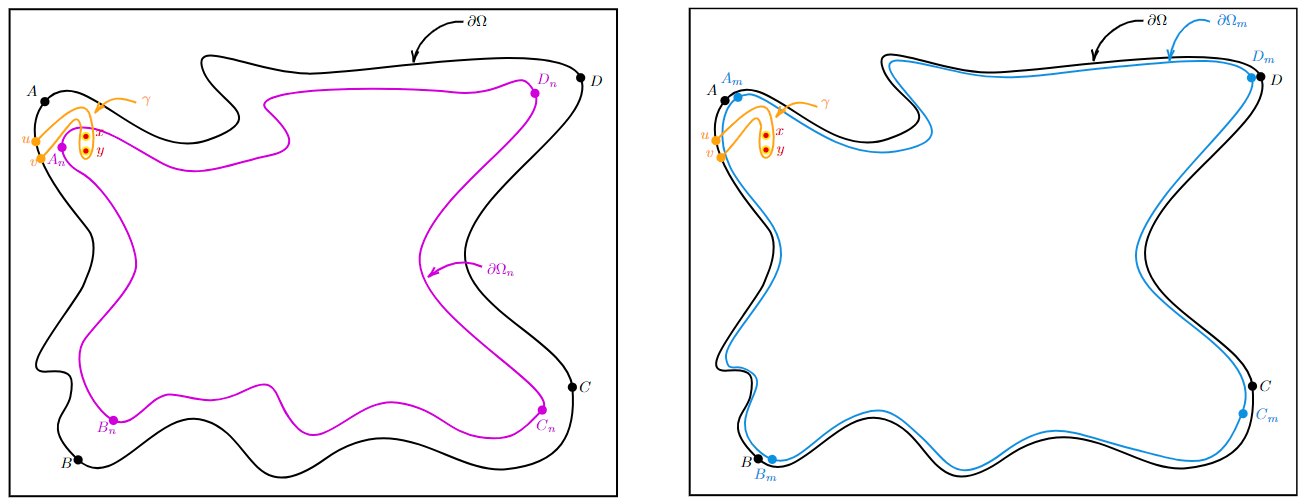}  
        \caption{On the left: The situation we are trying to avoid: the crosscut $\gamma$ which joins $x$ and $y$ to $[A,B]_{\partial{\Omega^\ast}}$ and separates $x$ and $y$ from $[C,D]_{\partial{\Omega^\ast}}$ intersects $\partial{\Omega_{n}}$ along the arc $[D_{n},A_{n}]_{\partial{\Omega_{n}}}$ rather than $[A_{n},B_{n}]_{\partial{\Omega_{n}}}$. On the right: For a fixed crosscut $\gamma$, we can avoid the picture on the left by taking $m>n$ sufficiently large.}
    \end{figure}
    Hence, the points of intersection of $\gamma[0,s]$ and $\partial{\Omega_{n}}$ must all lie on $[A_{n},B_{n}]_{\partial{\Omega_{n}}}$. To see why, suppose $w_{n}$ is a point of intersection of $\partial{\Omega_{n}}$ and $\gamma[0,s]$, and $w_{n}$ lies on one of the other boundary arcs of $(\Omega_{n},A_{n},B_{n},C_{n},D_{n})$. Without loss of generality, suppose $w_{n}\in{[B_{n},C_{n}]_{\partial{\Omega_{n}}}}$. Since $(\Omega_{n},A_{n},B_{n},C_{n},D_{n})$ is a $\delta_{n}$-good interior approximation of $(\Omega,A,B,C,D)$, there exists a crosscut $\zeta_{n}$ of $\Omega$ joining $w_{n}$ to $[B,C]_{\partial{\Omega^{\ast}}}$, and separating $w_{n}$ from $[D,A]_{\partial{\Omega^{\ast}}}$ so that $\text{length}(\zeta_{n})<\delta_{n}$. Stitching together $\chi$ and $\zeta_{n}$ we can produce a crosscut of $\Omega$ with length at most $\delta_{n}+\frac{c(u)}{2}$ that joins $u$ to $[B,C]_{\partial{\Omega^{\ast}}}$ and separates $u$ from $[D,A]_{\partial{\Omega^{\ast}}}$. Hence, for $n$ sufficiently large so that $\delta_{n}<\frac{c(u)}{2}$, we have that: 
    \begin{equation*}
        d_{cc}^{\Omega}(u,[B,C]_{\partial{\Omega^{\ast}}},[D,A]_{\partial{\Omega^{\ast}}})<\delta_{n}+\frac{c(u)}{2}<c(u)
    \end{equation*}
    This contradicts the definition of $c(u)$. Thus, the points of intersection of $\gamma[0,s]$ and $\partial{\Omega_{n}}$ all lie on $[A_{n},B_{n}]_{\partial{\Omega^{\ast}}}$. By the exact same argument, the points of intersection of $\gamma[t,1]$ and $\partial{\Omega_{n}}$ also lie on $[A_{n},B_{n}]_{\partial{\Omega^{\ast}}}$, for $n$ sufficiently large. Hence, there exists a subarc $\gamma_{n}$ of $\gamma$ so that $\gamma_{n}$ is a crosscut of $\Omega_{n}$ that joins $x$ and $y$ to $[A_{n},B_{n}]_{\partial{\Omega_{n}}}$ and separates $x$ and $y$ from $[C_{n},D_{n}]_{\partial{\Omega_{n}}}$. Since $\text{length}(\gamma_{n})\leq{\text{length}(\gamma)}$, and we can build such a crosscut $\gamma_{n}$ of $\Omega_{n}$ for any crosscut $\gamma$ of $\Omega$ that joins $x$ and $y$ to $[A,B]_{\partial{\Omega^{\ast}}}$ in $\Omega$ and separates $x$ and $y$ from $[C,D]_{\partial{\Omega^{\ast}}}$, provided that $n$ is sufficiently large, it follows that: 
    \begin{equation*}
        \limsup_{n\rightarrow{\infty}}d^{\Omega_{n}}_{cc}(\{x,y\},[A_{n},B_{n}]_{\partial{\Omega_{n}}};[C_{n},D_{n}]_{\partial{\Omega_{n}}})\leq{d^{\Omega}_{cc}(\{x,y\},[A,B]_{\partial{\Omega^{\ast}}};[C,D]_{\partial{\Omega^{\ast}}})}
    \end{equation*}
\end{proof}

\begin{lem}
\label{lem: diamaters of curves crossing rectangles converge}
    Suppose $(\Omega,A,B,C,D)$ is a conformal rectangle and for each $n\in{\mathbb{N}}$,  $(\Omega_{n},A_{n},B_{n},C_{n},D_{n})$ is a $\delta_{n}$-good interior approximation of $(\Omega,A,B,C,D)$, where $\delta_{n}\rightarrow{0}$ as $n\rightarrow{\infty}$. Suppose: 
    \begin{equation*}
        d=\inf\{\text{diam}(\gamma): \text{$\gamma$ is a curve in $\Omega$ joining $[A,B]_{\partial{\Omega^{\ast}}}$ and $[C,D]_{\partial{\Omega^{\ast}}}$}\}
    \end{equation*}
    and for each $n\in{\mathbb{N}}$: 
    \begin{equation*}
        d_{n}=\inf\{\text{diam}(\gamma): \text{$\gamma$ is a curve in $\Omega_{n}$ joining $[A_{n},B_{n}]_{\partial{\Omega_{n}}}$ and $[C_{n},D_{n}]_{\partial{\Omega_{n}}}$}\}
    \end{equation*}
    Then: 
    \begin{equation*}
        \lim_{n\rightarrow{\infty}}d_{n}=d
    \end{equation*}
\end{lem}
\begin{proof}
    Suppose $\gamma_{n}$ is a curve in $\Omega_{n}$ from $[A_{n},B_{n}]_{\partial{\Omega_{n}}}$ to $[C_{n},D_{n}]_{\partial{\Omega_{n}}}$. Let $u_{n}$ be the endpoint of $\gamma_{n}$ on $[A_{n},B_{n}]_{\partial{\Omega_{n}}}$ and let $v_{n}$ be the endpoint of $\gamma_{n}$ on $[C_{n},D_{n}]_{\partial{\Omega_{n}}}$. Since $(\Omega_{n},A_{n},B_{n},C_{n},D_{n})$ is a $\delta_{n}$-good interior approximation of $(\Omega,A,B,C,D)$, we can find crosscuts $\eta_{n}^{(1)}$ and $\eta_{n}^{(2)}$ of $\Omega$ so that $\eta_{n}^{(1)}$ joins $u_{n}$ to $[A,B]_{\partial{\Omega^{\ast}}}$ and separates $u_{n}$ from $[C,D]_{\partial{\Omega^{\ast}}}$ in $\Omega$, $\eta_{n}^{(2)}$ joins $v_{n}$ to $[C,D]_{\partial{\Omega^{\ast}}}$ and separates $v_{n}$ from $[A,B]_{\partial{\Omega^{\ast}}}$ in $\Omega$, and $\text{length}(\eta_{n}^{(1)}), \text{length}(\eta_{n}^{(2)})<\delta_{n}$. Stitching together $\eta_{n}^{(1)}, \gamma_{n}$ and $\eta_{n}^{(2)}$ we can produce a simple curve from $[A,B]_{\partial{\Omega^{\ast}}}$ to $[C,D]_{\partial{\Omega^{\ast}}}$ in $\Omega$ with diameter at most $\text{diam}(\gamma_{n})+2\delta_{n}$. Hence, $d\leq{d_{n}+2\delta_{n}}$ for all $n\in{\mathbb{N}}$, which tells us that: 
    \begin{equation*}
        d\leq{\liminf_{n\rightarrow{\infty}}d_{n}}
    \end{equation*}
    \begin{figure}[H]
        \centering
        \includegraphics[scale=0.55]{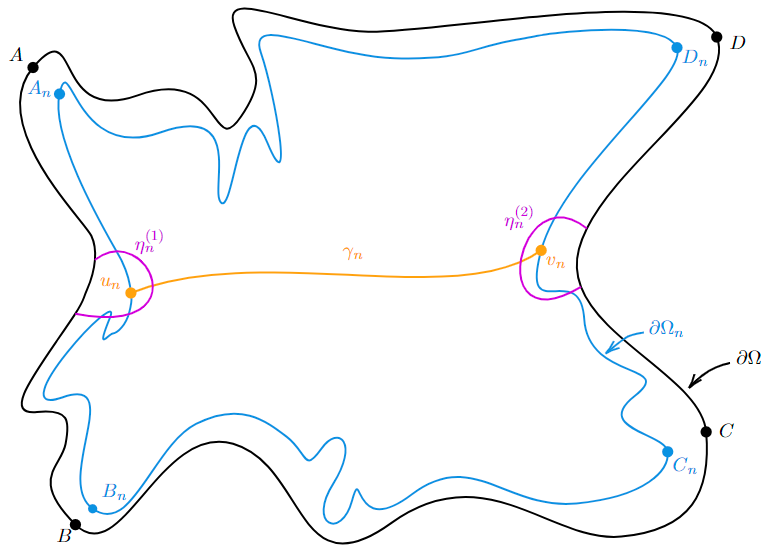}  
        \caption{Given an curve $\gamma_{n}$ in $\Omega_{n}$ from $[A_{n},B_{n}]_{\partial{\Omega_{n}}}$ to $[C_{n},D_{n}]_{\partial{\Omega_{n}}}$ with endpoints $u_{n}\in{[A_{n},B_{n}]_{\partial{\Omega_{n}}}}$ and $v_{n}\in{[C_{n},D_{n}]_{\partial{\Omega_{n}}}}$, we can stitch together $\gamma_{n}$ with the crosscuts $\eta_{n}^{(1)}$ and $\eta_{n}^{(2)}$ to produce a curve from $[A,B]_{\partial{\Omega^{\ast}}}$ to $[C,D]_{\partial{\Omega^{\ast}}}$ in $\Omega$ whose diameter is close to that of $\gamma_{n}$.} 
    \end{figure}
    \noindent Suppose $\gamma$ is a crosscut of $\Omega$ joining $[A,B]_{\partial{\Omega^{\ast}}}$ and $[C,D]_{\partial{\Omega^{\ast}}}$ with endpoints $u$ in $[A,B]_{\partial{\Omega^{\ast}}}$ and $v$ in $[C,D]_{\partial{\Omega^{\ast}}}$. Since any path from $[A,B]_{\partial{\Omega^{\ast}}}$ to $[C,D]_{\partial{\Omega^{\ast}}}$ in $\Omega$ that starts or ends at a corner $A,B,C,D$ of our conformal rectangle is close to a path from $[A,B]_{\partial{\Omega^{\ast}}}$ to $[C,D]_{\partial{\Omega^{\ast}}}$ in $\Omega$ that has almost the same diameter and doesn't use any of the corners of $(\Omega,A,B,C,D)$, we can assume without loss of generality that $u\in{(A,B)}_{\partial{\Omega^{\ast}}}$ and $v\in{(C,D)_{\partial{\Omega^{\ast}}}}$. Hence: 
    \begin{align*}
         c(u)&=\min\{d_{cc}^{\Omega}(u,[B,C]_{\partial{\Omega^{\ast}}};[A,D]_{\partial{\Omega^{\ast}}}), d_{cc}^{\Omega}(u,[C,D]_{\partial{\Omega^{\ast}}};[A,B]_{\partial{\Omega^{\ast}}}), d_{cc}^{\Omega}(u,[D,A]_{\partial{\Omega^{\ast}}};[B,C]_{\partial{\Omega^{\ast}}})\}>0 \\
         c(v)&=\min\{d_{cc}^{\Omega}(v,[B,C]_{\partial{\Omega^{\ast}}};[A,D]_{\partial{\Omega^{\ast}}}), d_{cc}^{\Omega}(v,[C,D]_{\partial{\Omega^{\ast}}};[A,B]_{\partial{\Omega^{\ast}}}), d_{cc}^{\Omega}(v,[D,A]_{\partial{\Omega^{\ast}}};[B,C]_{\partial{\Omega^{\ast}}})\}>0
    \end{align*}
   Let $\gamma:[0,1]\rightarrow{\Omega^{\ast}}$ be an explicit parametrization of $\gamma$. Choose $s,t\in{(0,1)}$ so that: 
    \begin{align*}
        \text{diam}(\gamma[0,s])&<\frac{c(u)}{4\pi}, & \text{diam}(\gamma[t,1])&<\frac{c(v)}{4\pi}
    \end{align*}
    Since $\gamma[s,t]$ is a compact subset of $\Omega$, by Lemma \ref{lem: K in Omega_n for n sufficiently large}, $\gamma[s,t]\subseteq{\Omega_{n}}$ for all $n$ sufficiently large. In this case, since $\gamma$ starts and ends outside $\Omega_{n}$, $\gamma$ must intersect $\Omega_{n}$ at least twice: once when it enters $\Omega_{n}$ and once when it exits. Since $\gamma[s,t]\subseteq{\Omega_{n}}$, the points of intersection of $\gamma$ with $\partial{\Omega_{n}}$ must all lie in $\gamma[0,s]$ or $\gamma[t,1]$. Since $u=\gamma(0)\in{(A,B)_{\partial{\Omega^{\ast}}}}$ and $\text{diam}(\gamma[0,s])<\frac{c(u)}{4\pi}$, there is a subarc $\chi$ of $\{z\in{\Omega^{\ast}}:|z-u|=\frac{c(u)}{4\pi}\}$ so that $[C,D]_{\partial{\Omega^{\ast}}}$ and $\gamma[0,s]$ lie in distinct connected components of $\Omega^{\ast}\setminus{\chi}$. In other words: 
    \begin{equation*}
        d_{cc}^{\Omega}(\gamma[0,s],[A,B]_{\partial{\Omega^{\ast}}};[C,D]_{\partial{\Omega^{\ast}}})<\frac{c(u)}{2}
    \end{equation*}
    \begin{figure}[h]
        \centering
        \includegraphics[scale=0.52]{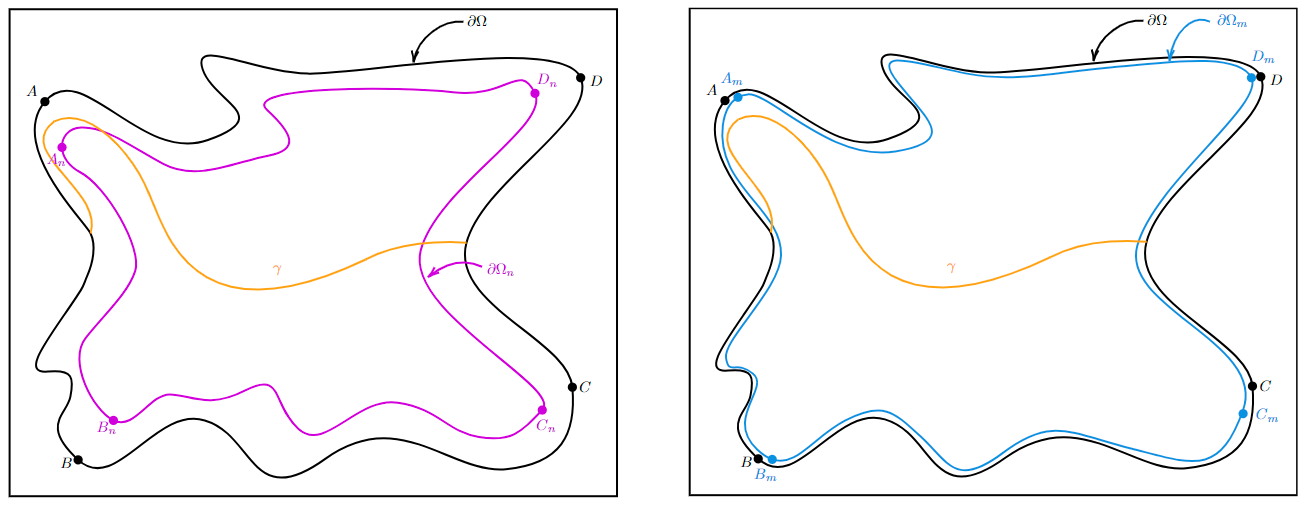}  
        \caption{On the left: The situation we are trying to avoid: the curve $\gamma$ between $[A,B]_{\partial{\Omega^{\ast}}}$ and $[C,D]_{\partial{\Omega^{\ast}}}$ in $\Omega$ enters $\Omega_{n}$ via the boundary arc $[D_{n},A_{n}]_{\partial{\Omega_{n}}}$ rather than $[A_{n},B_{n}]_{\partial{\Omega_{n}}}$. On the right: For a fixed curve $\gamma$, we can avoid the picture on the left by taking $m>n$ sufficiently large.}
    \end{figure}
    Hence, the points of intersection of $\gamma[0,s]$ and $\partial{\Omega_{n}}$ must all lie on $[A_{n},B_{n}]_{\partial{\Omega_{n}}}$. To see why, suppose $w_{n}$ is a point of intersection of $\partial{\Omega_{n}}$ and $\gamma[0,s]$, and $w_{n}$ lies on one of the other boundary arcs of $(\Omega_{n},A_{n},B_{n},C_{n},D_{n})$. Without loss of generality, suppose $w_{n}\in{[B_{n},C_{n}]_{\partial{\Omega_{n}}}}$. Since $(\Omega_{n},A_{n},B_{n},C_{n},D_{n})$ is a $\delta_{n}$-good interior approximation of $(\Omega,A,B,C,D)$, there exists a crosscut $\zeta_{n}$ of $\Omega$ joining $w_{n}$ to $[B,C]_{\partial{\Omega^{\ast}}}$, and separating $w_{n}$ from $[D,A]_{\partial{\Omega^{\ast}}}$ so that $\text{length}(\zeta_{n})<\delta_{n}$. Stitching together $\chi$ and $\zeta_{n}$ we can produce a crosscut of $\Omega$ with length at most $\delta_{n}+\frac{c(u)}{2}$ that joins $u$ to $[B,C]_{\partial{\Omega^{\ast}}}$ and separates $u$ from $[D,A]_{\partial{\Omega^{\ast}}}$. Hence, for $n$ sufficiently large so that $\delta_{n}<\frac{c(u)}{2}$, we have that: 
    \begin{equation*}
        d_{cc}^{\Omega}(u,[B,C]_{\partial{\Omega^{\ast}}},[D,A]_{\partial{\Omega^{\ast}}})<\delta_{n}+\frac{c(u)}{2}<c(u)
    \end{equation*}
    This contradicts the definition of $c(u)$. Thus, the points of intersection of $\gamma[0,s]$ and $\partial{\Omega_{n}}$ all lie on $[A_{n},B_{n}]_{\partial{\Omega^{\ast}}}$. By the exact same argument, the points of intersection of $\gamma[t,1]$ and $\partial{\Omega_{n}}$ all lie on $[C_{n},D_{n}]_{\partial{\Omega^{\ast}}}$, for $n$ sufficiently large. Hence, there exists a subarc $\gamma_{n}$ of $\gamma$ so that $\gamma_{n}$ is a crosscut of $\Omega_{n}$ joining $[A_{n},B_{n}]_{\partial{\Omega_{n}}}$ and $[C_{n},D_{n}]_{\partial{\Omega_{n}}}$ such that $\gamma[s,t]\subseteq{\gamma_{n}}$, which tells us that $\text{diam}(\gamma_{n})\leq{\text{diam}(\gamma)}$. Since we can do this for any curve $\gamma$ from $(A,B)_{\partial{\Omega^{\ast}}}$ to $(C,D)_{\partial{\Omega^{\ast}}}$ in $\Omega$, we conclude that: 
    \begin{equation*}
        \limsup_{n\rightarrow{\infty}}d_{n}\leq{d}
    \end{equation*}
\end{proof}
\noindent In order to prove Proposition \ref{prop: we can approximate any simply connected domain}, we need the following analogue of Wolff's Lemma (see Proposition 2.2 of \cite{pommerenke}) for conformal mappings of rectangles:
\begin{lem}
\label{lem: Wolff's Lemma for rectangles}
    Consider the conformal rectangle $(\mathcal{R}_{L},i,0,L,L+i)$, which is just $\mathcal{R}_{L}$ with its four corners as the four distinguished boundary points. Let $\Omega$ be a bounded simply connected domain and let $\psi:\mathcal{R}_{L}\rightarrow{\Omega}$ be a conformal map from $\mathcal{R}_{L}$ to $\Omega$. If $z_{0}$ is a point of $\mathcal{R}_{L}$ so that $d=\text{dist}(z_{0},\partial{\mathcal{R}_{L}})<\frac{1\wedge{L}}{2}$ and $\mathcal{I}$ and $\mathcal{J}$ are boundary arcs of $(\mathcal{R}_{L},i,0,L,L+i)$ so that $\mathcal{I}$ is the boundary arc closest to $z_{0}$, and $\mathcal{J}$ is opposite to $\mathcal{I}$, then there exists a crosscut of $\Omega$ of length less than $\sqrt{\frac{\pi\text{Area}(\Omega)}{2\log\big(\frac{1\wedge{L}}{2d}\big)}}$ joining $\psi(z_{0})$ to $\psi(\mathcal{I})$ and separating $\psi(z_{0})$ from $\psi(\mathcal{J})$ in $\Omega$. Equivalently: 
    \begin{equation*}
        d_{cc}^{\Omega}(\psi(z_{0}),\psi(\mathcal{I});\psi(\mathcal{J}))<\sqrt{\frac{\pi\text{Area}(\Omega)}{2\log\big(\frac{1\wedge{L}}{2d}\big)}}
    \end{equation*}
\end{lem}
\noindent The proof of this result is identical to the proof of Wolff's Lemma. For the convenience of the reader, we reproduce the argument below: 
\begin{proof}
   If $z_{0}$ is a point of $\mathcal{R}_{L}$ so that $d=\text{dist}(z_{0},\partial{\Omega})=\text{dist}(z_{0},\mathcal{I})<\frac{1\wedge{L}}{2}$, then for each $s\in{(d,\frac{1\wedge{L}}{2})}$,
   \begin{equation*}
       \gamma_{s}=\mathcal{R}_{L}\cap{\{z:|z-z_{0}|=s\}}
   \end{equation*}
   is a crosscut of $\mathcal{R}_{L}$ that joins $z_{0}$ to $\mathcal{I}$ and separates $z_{0}$ from $\mathcal{J}$. It follows that $\psi(\gamma_{s})$ is a crosscut of $\Omega$ that joins $\psi(z_{0})$ to $\psi(\mathcal{I})$ and separates $\psi(z_{0})$ from $\psi(\mathcal{J})$. We will now show that one of these crosscuts $\psi(\gamma_{s})$ is sufficiently short, using the length-area trick. Namely:
   \begin{align*}
       \text{Area}(\Omega)&>\text{Area}(\psi\Big(\mathcal{R}_{L}\cap{\{z:d<|z-z_{0}|<\frac{1\wedge{L}}{2}\}}\Big))=\int_{\mathcal{R}_{L}\cap{\{z:d<|z-z_{0}|<\frac{1\wedge{L}}{2}\}}}|\psi'(x+iy)|^{2}dxdy \\ 
       &=\int_{d}^{\frac{1\wedge{L}}{2}}\int_{\theta_{1}(s)}^{\theta_{2}(s)}s|\psi'(z_{0}+se^{i\theta})|^{2}d\theta{ds}\overset{(\ast)}{\geq}{\int_{d}^{\frac{1\wedge{L}}{2}}\frac{2}{\pi{s}}\Big(\int_{\theta_{1}(s)}^{\theta_{2}(s)}s|\psi'(z_{0}+se^{i\theta})|d\theta\Big)^{2}ds} \\
       &=\frac{2}{\pi}\int_{d}^{\frac{1\wedge{L}}{2}}\frac{\big(\text{length}(\psi(\gamma_{s}))\big)^{2}}{s}ds\geq{\frac{2}{\pi}\inf_{d<s<\frac{1\wedge{L}}{2}}\big(\text{length}(\psi(\gamma_{s}))\big)^{2}\hspace{1pt}\log\Big(\frac{1\wedge{L}}{2d}\Big)}
   \end{align*}
   The inequality ``$(\ast)$" follows by Cauchy-Schwarz and the observation that since we're working in $\mathcal{R}_{L}$, $(\theta_{2}(s)-\theta_{1}(s))\geq{\frac{\pi}{2}}$ for all $s\in{(d,\frac{1\wedge{L}}{2})}$. Rearranging, we get that: 
   \begin{equation*}
       \inf_{d<s<\frac{1\wedge{L}}{2}}\text{length}(\psi(\gamma_{s}))<\sqrt{\frac{\pi\text{Area}(\Omega)}{2\log\big(\frac{1\wedge{L}}{2d}\big)}}
   \end{equation*}
\end{proof}
\noindent Having established Lemma \ref{lem: Wolff's Lemma for rectangles}, Proposition \ref{prop: we can approximate any simply connected domain} follows as an immediate corollary: 
\begin{proof}
    (of Proposition \ref{prop: we can approximate any simply connected domain}) Let $\phi:\Omega\rightarrow{\mathcal{R}_{L}}$ be the conformal map from $\Omega$ to $\mathcal{R}_{L}$ so that the four distinguished prime ends of $\Omega$ are mapped to the four corners of $\mathcal{R}_{L}$ and in particular, $\phi(A)=i$. Fix $r\in{(0,\frac{1\wedge{L}}{2})}$. Consider the subrectangle $(\Omega_{r},A_{r},B_{r},C_{r},D_{r})$ of $(\Omega,A,B,C,D)$, where:
    \begin{multicols}{2}
        \begin{itemize}
            \item $\Omega_{r}=\phi^{-1}\big((r,L-r)\times{(r,1-r)}\big)$
            \item $A_{r}=\phi^{-1}(r+(1-r)i)$
            \item $B_{r}=\phi^{-1}(r+ri)$
            \item $C_{r}=\phi^{-1}((L-r)+ri)$
            \item $D_{r}=\phi^{-1}((L-r)+(1-r)i)$
            \item[\vspace{\fill}]
        \end{itemize}
    \end{multicols}
    \noindent In other words, $\Omega_{r}$ is the preimage of the rectangle $(r,L-r)\times{(r,1-r)}$ under the conformal map $\phi$, and the prime ends $A_{r},B_{r},C_{r},D_{r}$ are the preimages under $\phi$ of the four corners of this rectangle. By Lemma \ref{lem: Wolff's Lemma for rectangles},
    \begin{equation*}
        d_{cc}^{\Omega}([A_{r},B_{r}]_{\partial{\Omega_{r}}},[A,B]_{\partial{\Omega^{\ast}}};[C,D]_{\partial{\Omega^{\ast}}})\leq{\sqrt{\frac{\pi\text{Area}(\Omega)}{2\log\big(\frac{1\wedge{L}}{2r}\big)}}}
    \end{equation*}
    By the exact same reasoning, the analogous result is true for the other three boundary arcs of $(\Omega_{r},A_{r},B_{r},C_{r},D_{r})$. Thus, choosing $r<\big(\frac{1\wedge{L}}{2}\big)e^{-\frac{\pi\text{Area}(\Omega)}{2\delta^{2}}}$ for some fixed $\delta>0$, the desired result follows. 
\end{proof}
\end{appendices}


\begin{thebibliography}{}

\bibitem{ACFP18}
N. Albin, J. Clemens, N. Fernando and P. Poggi-Corradini. 
Blocking duality for $p$- modulus on networks and applications.
\textit{Annali di Matematica Pura ed Applicata}, 198(3): 973-999, 2018. 
\url{https://doi.org/10.1007/s10231-018-0806-0}

\bibitem{AKP20}
N. Albin, K. Kottegoda and P. Poggi- Corradini.
A Polynomial- Time Algorithm for Spanning Tree Modulus.
\textit{ArXiv e-print}, September 2020.
\url{https://doi.org/10.48550/arXiv.2009.03736}

\bibitem{ALP23}
N. Albin, J. Lind and P. Poggi- Corradini. 
Convergence of the Probabilistic Interpretation of Modulus.
\textit{ArXiv e-prints}, June 2021.
\url{https://doi.org/10.48550/arXiv.2106.11418}

\bibitem{BS96}
I. Benjamini and O. Schramm. 
Random walks and harmonic functions on infinite planar graphs using square tilings.
\textit{Ann. Probab.}, 24(3): 1219-1238, 1996. \url{https://doi.org/10.1214/aop/1065725179}

\bibitem{BGS24}
F. Bertacco, E. Gwynne, and S. Sheffield.
Scaling limits of planar maps under the Smith embedding. 
\textit{ArXiv e-prints}, October 2024. 
\url{
https://doi.org/10.48550/arXiv.2306.02988}

\bibitem{BRY14}
I. Binder, C. Rojas and M. Yampolsky.
Computable Carath\'eodory Theory.
\textit{Adv. Math.}, 256: 280-312, 2014.
\url{https://doi.org/10.1016/j.aim.2014.07.039}

\bibitem{BG24}
A. Bou-Rabee and E. Gwynne. 
Random walk on sphere packings and Delaunay triangulations in arbitrary dimension. 
\textit{ArXiv e-prints}, May 2024. 
\url{https://doi.org/10.48550/arXiv.2405.11673}

\bibitem{BSST40}
R. L. Brooks, C. A. B. Smith, A. H. Stone and W. T. Tutte.
The dissection of rectangles into squares. 
\textit{Duke Math. J.}, 7(1): 312-340, 1940. 
\url{https://doi.org/10.1215/S0012-7094-40-00718-9}

 \bibitem{CN07}
 F. Camia and C.M. Newman. 
 Critical percolation exploration path and $SLE_{6}$: a proof of convergence.
 \textit{Probab. Theory Relat. Fields}, 139: 473–519, 2007. 
\url{https://doi.org/10.1007/s00440-006-0049-7}
 
\bibitem{Ising}
 D. Chelkak, H. Duminil- Copin, C. Hongler, A. Kemppainen and S. Smirnov. 
 Convergence of Ising Interfaces to Schramm SLE curves.
 \textit{C. R. Math.}, 352(2): 157-161, 2014. 
\url{https://doi.org/10.1016/j.crma.2013.12.002}

\bibitem{CLR23}
D. Chelkak, B. Laslier and M. Russkikh.
Dimer model and holomorphic functions on t-embeddings of planar graphs.
\textit{Proc. Lond. Math. Soc.}, 126(3): 1656–1739, 2023.
\url{https://doi.org/10.1112/plms.12516}

\bibitem{CS11}
D. Chelkak and S. Smirnov. 
Discrete complex analysis on isoradial graphs. 
\textit{Adv. Math.}, 228(3): 1590-1630, 2011.
\url{https://doi.org/10.1016/j.aim.2011.06.025}

\bibitem{stfluor}
N. Curien. 
Peeling Random Planar Maps.
Lecture Notes in Mathematics, 
\textit{Springer Cham}, 2023.

\bibitem{D62}
R.J. Duffin. 
The extremal length of a network. 
\textit{J. Math. Anal. Appl.}, 5(2): 200-215, 1962.
\url{https://doi.org/10.1016/S0022-247X(62)80004-3}

\bibitem{FF56}
 L. R. Ford Jr. and D. R. Fulkerson.
 Maximal flow through a network. \textit{Canad. J. Math.}, 8:
399–404, 1956.  
\url{ https://doi.org/10.4153/CJM-1956-045-5}

\bibitem{GM05}
J. Garnett and D. Marshall. 
Harmonic Measure. 
\textit{Cambridge University Press}, 2005.

\bibitem{G16}
A. Georgakopoulos. 
The boundary of a square tiling of a graph coincides with the Poisson
boundary.
\textit{Invent. math.}, 203(3): 773–821, 2016.
\url{https://doi.org/10.1007/s00222-015-0601-0}
 
 \bibitem{GP20}
 A. Georgakopoulos and C. Panagiotis. 
 Convergence of square tilings to the Riemann map.
 \textit{ArXiv e- prints}, March 2020. 
\url{https://doi.org/10.48550/arXiv.1910.06886}
 
 \bibitem{GJN20}
 O. Gurel-Gurevich, D. Jerison and A. Nachmias. 
 The Dirichlet problem for orthodiagonal maps. 
 \textit{Adv. Math.}, vol. 374, 2020.
\url{https://doi.org/10.1016/j.aim.2020.107379}

\bibitem{HS96}
Z-X. He and O. Schramm.
On the convergence of circle packings to the Riemann map. 
\textit{Invent. Math.}, 125: 285–305, 1996. \url{https://doi.org/10.1007/s002220050076}

\bibitem{H18}
S. Hersonsky. 
Approximation of conformal mappings and novel applications to shape recognition of planar domains. 
\textit{J. Supercomput.}, 74(11): 6333-6368, 2018. 
\url{https://doi.org/10.1007/s11227-018-2564-6} 

\bibitem{HP17}
T. Hutchcroft and Y. Peres. Boundaries of planar graphs: a unified approach.
\textit{Electron. J. Probab.}, 22: 1-20, 2017. \url{https://doi.org/10.1214/17-EJP116}

\bibitem{Ke02}
R. Kenyon. 
The Laplacian and Dirac operators on critical planar graphs. 
\textit{Invent. Math.}, 150: 409-439, 2002.
\url{https://doi.org/10.1007/s00222-002-0249-4}

\bibitem{CIPbook}
G. Lawler. 
Conformally Invariant Processes in the Plane.
Mathematical Surveys and Monographs, \textit{American Mathematical Society}, 2005.

\bibitem{pmrwcp}
A. Nachmias. 
Planar Maps, Random Walks and Circle Packing.
Lecture Notes in Mathematics, \textit{Springer Open}, 2020.

\bibitem{P24}
D. Pechersky. 
A polynomial rate of convergence for the Dirichlet problem on orthodiagonal maps.
\textit{ArXiv e- prints}, March 2025.
\url{https://doi.org/10.48550/arXiv.2503.20284} 

\bibitem{pommerenke}
C. Pommerenke.
Boundary Behavior of Conformal Maps.
Grundlehren der mathematischen Wissenschaften, \textit{Springer- Verlag}, 1992.

\bibitem{RS87}
B. Rudin and D. Sullivan. 
The convergence of circle packings to the Riemann mapping. 
\textit{J. Differential Geom.}, 26(2): 349-360, 1987. 
\url{http://doi.org/10.4310/jdg/1214441375} 

\bibitem{SS01}
 Stanislav Smirnov.
 Critical percolation in the plane: conformal invariance, Cardy's formula, scaling limits.
 \textit{C. R. Acad. Sci. Paris}, 333(1): 239–244, 2001.
 \url{https://doi.org/10.1016/S0764-4442(01)01991-7}

\bibitem{Stephenson}
K. Stephenson. 
Introduction to Circle Packing: The Theory of Discrete Analytic Functions. 
\textit{Cambridge University Press}, 2005. 

\bibitem{Th85}
B. Thurston. 
The finite Riemann mapping theorem.
Invited address at the International Symposium in Celebration of the Proof of the Bieberbach Conjecture. 
Purdue University, 1985.

\bibitem{SteinShakarchi}
R. Shakarchi and E.M. Stein.
Complex Analysis. 
\textit{Princeton University Press}, 2003.

\bibitem{YY11}
A. Yadin and A. Yehudayoff.
Loop-erased random walk and Poisson kernel on planar graphs.
\textit{Ann. Probab.}, 39(4): 1243-1285, 2011.
\url{https://doi.org/10.1214/10-AOP579} 
 
\end{thebibliography}
\end{document}